\documentclass[12 pt]{article}
\usepackage[utf8]{inputenc}
\usepackage[fleqn]{amsmath}
\usepackage{array}
\usepackage{appendix}
\usepackage{romannum}
\usepackage{amsfonts}
\usepackage{amsthm}
\newtheorem{definition}{Definition}
\usepackage{enumitem}
\usepackage{soul}
\usepackage{amssymb,latexsym}
\usepackage{graphics}
\usepackage{float}
\usepackage{subcaption}
\usepackage{graphicx}
\usepackage{comment}
\usepackage{mathtools}
\usepackage{amsfonts}
\usepackage{amsthm}
\newtheorem{theorem}{\textbf{Theorem}}
\newtheorem{lemma}{\textbf{Lemma}}
\usepackage{multicol}
\usepackage{cite}
\usepackage{tikz}
\usepackage{amsmath, amsthm, amscd, amsfonts, amssymb, graphicx}
\usepackage[left=1.5 cm,right=1.5 cm,top=1.5cm,bottom=1.5cm]{geometry}
\title{
Exploring the refuge-induced bubbling phenomenon and harvesting in a three species food chain model that incorporates memory effect and odour effect}
\author{Dipam Das$^{1,2}$,Debasish Bhattacharjee$^1$\\
$^{1}$Department of Mathematics, Gauhati University, Assam, India \\
$^{2}$Department of Mathematics, B B Kishan College, Assam, India\\
}

\date{}

\begin{document}

\maketitle

\begin{abstract}
In this study, an odour-mediated system is developed and studied. In an odor-mediated systems, the sense of smell or odour of species plays a critical role in the interactions between predators and prey. It is widely recognised in scientific literature that these systems are very common and essential across natural ecosystems. These systems are crucial for various behaviors, including foraging, mating, and avoiding predators. In this paper, it is assumed that the presence of prey odour aids the predator in its hunting efforts. It is assumed that both prey and intermediate predators seek refuge against their respective predators upon detecting the odour of their predators. In other words, the odour of predators assists prey species in evaluating the danger and seeking refuge for hiding. This model incorporates the prey species' harvesting as well. We also explore the impact of fading memory on the system by incorporating fractional derivatives into the model. The conditions for both the existence and local stability of the non-negative equilibria are derived. The current model system undergoes both Hopf and transcritical bifurcation when the parameter values are appropriately chosen. The dynamic behaviour of the system is showcased and thoroughly analysed using a range of diagrams, highlighting the impact of prey refuge and predator odour parameters. This paper extensively examines the long-term impacts of harvesting within the system. The extent to which prey odour influences the system is investigated, and it emerges that prey odour can play a significant function within the system. 
It has been noted that when the refuge for intermediate predators gets bigger, it becomes more challenging for all three populations to coexist within the system. 
Furthermore, it is apparent that the prey refuge parameter $m_1$ induces bubbling phenomena in the system. 
The presence of prey odour plays a significant role in promoting a long-term cohabitation dynamic within this specific system. It has been observed that when individuals within the system have a strong memory, it positively affects the stability of the system. Numerical simulations are conducted in order to demonstrate and validate the usefulness of the model being considered, therefore supporting the analytical conclusions.

\end{abstract}

\pagenumbering{arabic}

\section{Introduction} 
In ecology, odour-mediated predator-prey systems can be distinguished by the pivotal function of the odour in the interactions between predators and prey. The odour of a species is one of the key elements that significantly affects the functional response in a predator-prey system. The olfactory sense plays a vital part in essential biological activities such as predator evasion, mate selection, trail and territory identification, and food acquisition \cite{dar}. It has been noticed that a considerable proportion of animal species heavily depend on olfaction as a means to capture and analyse environmental information. Marsupials (\textit{Marsupialia}) have been observed utilising feather odour signals of crimson rosella (\textit{Platycercus elegans}) as a means to evaluate the condition of nest hollows. This knowledge may assist them in detecting avian prey or enhancing their vigilance when competing with parrots (\textit{Psittaciformes}) for nest hollows \cite{miha}. Insects possess the ability to detect plants by olfactory cues. The olfactory cues emitted by plants are also significant in the process of identifying and choosing food sources among mammals\cite{pad}. The parasitoid species \textit{Microplitis croceipes} relies on kairomones emitted by the host organism for the purpose of locating it \cite{vet}. Female moths (\textit{Lepidoptera}) have been observed to employ a fascinating mechanism involving the release of odorant compounds. These compounds, which are known to contain specialised enzymes, play a crucial role in the transmission of vital information among individuals of the same species \cite{xu}. Also, the use of olfactory cues by wolves (\textit{Canis lupus}) in the pursuit of prey is a well-documented phenomenon \cite{wu}. The odours secreted by prey aid the grasshopper mouse \textit{(Onychomys leucogaster}) in its predation efforts \cite{lang}. Although the impact of species' odour is widely recognised in ecological literature, there is a limited amount of research that explores its effect in mathematical models of predator-prey systems. Xu et al. \cite{xu} studied the effect of predator odour on prey species in a predator-prey system. Shen et al. \cite{shen} investigated a predator-prey model in the presence of predator odour disturbance. Bhattacharjee et al.\cite{bhat} discussed a three species food chain model incorporating the effect of prey odour. Das et al. \cite{dipam} explored how a species' scent can exert both harmful and advantageous influences on other species within a predator-prey interaction model.

Odour-mediated systems exhibit significant diversity and are present in various ecological interactions. There are numerous diverse effects that can occur in an odor-mediated system. It is found that the odour of predators can elicit defence mechanisms or anti-predator behaviours in prey species, which is widely recognised in ecological literature \cite{kro,jach,apfe,ylo,sax,miya,chi}. One such defensive behaviour frequently observed in nature is the refuge phenomenon. Therefore, the existence of prey refuge in an odor-mediated system is a prevalent kind of diversification. The notion of refuges is extensively acknowledged and examined within the disciplines of biology and ecology \cite{saha,sk,guin,pan,das}. The term "refuge-seeking behaviour" pertains to a phenomenon observed in organisms whereby they actively place themselves in areas that are either secluded or inaccessible to predators as a means of seeking protection from predation. This particular method enables the organisms to efficiently avoid predation and improve their likelihood of survival. The phenomenon of prey hiding has been observed to potentially exert a stabilising effect on predator-prey dynamics \cite{das,berry}.  However, as far as the authors are aware, no study has been undertaken that explores the connection between a prey's refuge behaviour and the odour of its predator.

 The prevalence of harvesting is an intriguing diversity among these odour-mediated systems. For instance, hunters utilise various scents, like doe urine, to effectively attract bucks and increase their chances of a successful hunt \cite{smith,moore}. Trappers use baits with potent scents, such as fish or honey, to entice bears into traps \cite{mas,schl}, etc.. Throughout the course of human history, the practice of harvesting different kinds of animals and insects, whether through fishing, hunting, or resource extraction, has been a vital and indispensable endeavour. Nevertheless, the expansion of human populations, developments in technology, and the increasing worldwide need for food and resources have resulted in heightened harvesting practices that have the potential to exert pressure on ecological systems and pose a threat to biodiversity \cite{pauly}. Intensive harvesting can harm populations, destabilise ecosystems, and lead to species extinction \cite{hut}. Thus, it makes sense to investigate the effect of harvesting in an odour-mediated predator prey system.


Despite the presence of varied odour-mediated systems in ecological settings, there has been scant research conducted in the field of mathematical modelling considering the influence of odour on predator-prey systems. This study aims to examine the impact of prey odour on predators as well as the influence of predator odour on prey in a food chain that includes harvesting. To the best of the authors' knowledge, there is currently no mathematical model in the literature that can accomplish the same objective.

There is a general consensus \cite{djo,dok,podl} that living organisms possess memory that is connected to their physical structure or mental processes . The role of species memories in shaping ecological systems is of great importance.  Thus, it is essential for a predator-prey mathematical model to consider this memory effect in order to enhance its relevance. Fractional order differential equations (FDEs) have additional benefits over integer order differential equations (ODEs) when it comes to modelling these processes. Fractional differential equations (FDEs) are operators that exhibit non-local behaviour, indicating that the future state of a function is influenced not exclusively by its current state but also by all of its prior states. On the other hand, integer order differential equations (ODEs) are limited to capturing only particular changes or features at a specific step of the process \cite{ans,kha}. Thus, fractional calculus has emerged as a crucial field of study for comprehending real-world problems throughout the years \cite{das,ans,kha,li,paul}. The integration of fractional calculus has had a significant impact on intricate dynamic systems, leading to remarkable advancements in the field of ecology. Multiple fractional derivatives have been thoroughly discussed in the literature on fractional calculus, one of which is Caputo's derivative. Given the limitations of the integer-order derivative (ODE) in collecting full memory and fully portraying the physical behaviour of the model, we will also examine our biological system by employing a fractional derivative in this paper.

Based on the aforementioned discussion and literature review, we developed an odour-mediated food chain model comprising three species in this work. Predators are presumed to derive advantages from the odour of their prey species. Both prey and intermediate predators are believed to profit from the presence of predator odour, as they utilise refuge-seeking as a strategic reaction to predator odour in order to reduce predation rates. The act of harvesting of prey is also considered. We will analyse the same model from two distinct ecological perspectives:\\
(1) Utilising fractional order derivatives to integrate the memory effect into the system.\\
(2) Utilising integer order derivatives that indicate the absence of memory in the system.\\

The paper is structured in the following manner: The paper's primary contributions to the scientific literature are covered in Section (\ref{rc}). In Section (\ref{mf}), the problem is mathematically modelled. The Section (\ref{p}) provides a comprehensive discussion on all the necessary prerequisites for the Caputo fractional order derivative. The concept of well-posedness, including properties such as positivity and boundedness, is addressed in Section (\ref{wp}). The paper discusses the existence and local stability of all ecologically feasible equilibrium points in Section (\ref{es}). The occurrence of different bifurcations is explored in Section (\ref{hb}). The analytical findings are validated using numerical simulations in Section (\ref{ns}). Conclusions are drawn in Section (\ref{con}). Furthermore, the visual representation of the paper's structure may be seen in figure (\ref{resmethod}).

\section{Research contribution} \label{rc} Based on the aforementioned literature review and subsequent study, it becomes evident that in an odour-mediated predator-prey system, the species' odour has a direct influence on the population dynamics of the system, and it is a prevalent phenomenon in predator-prey interactions. These odour-mediated systems exhibit significant diversity and are present in various ecological interactions. Even though a substantial amount of research has been conducted in the field of mathematical modelling of a predator-prey system, there are still substantial voids in the field that need to be addressed. The following outline highlights some of these gaps:

1) Odor-mediated predator-prey dynamics are common in nature and have been thoroughly examined in ecological research.  However, there exists a significant paucity of research concerning odor-driven predator-prey systems within the domain of mathematical modeling. To the best of the authors' knowledge, only a limited number of studies \cite{xu,shen,bhat,dipam} have explored the impact of odour in the mathematical modeling of predator-prey systems.

2) In the paper \cite{xu}, the authors examined the impact of predator odour on prey, identifying it as a detrimental factor affecting the prey's growth rate.  In the work \cite{shen}, the researchers examined the impact of predator odour disturbance.  In the study \cite{bhat}, the investigators examined the beneficial impact of prey odour on predators.  In the research work \cite{dipam}, the authors examined the predator odour's beneficial impact on prey and detrimental effect on its competitors.  It follows that a system wherein prey odour influences its predators while concurrently predator odour impacts the prey is not addressed within a mathematical model of predator-prey dynamics.

3) To the best of the authors knowledge, no research has been conducted in the realm of mathematical modelling of a predator-prey system that incorporates the concept of predator-odour induced refuge in prey.

4) Although it is well acknowledged in existing literature \cite{saiv,chu,czac,rein} that a relationship exists between the odour of a species and the memory of its predators or prey, there has been a lack of exploration into an odour-mediated predator-prey system in relation to the memory effect within the context of mathematical modelling of population dynamics.


To address these research gaps in the existing literature, an odour-mediated three species food chain model is proposed that examines the dynamics of a predator-prey system in the presence of prey odour, predator odour-induced prey refuge, and harvesting. It is proposed that the odour of prey aids predators in their predation efforts, but prey also get benefited from the odour of their predators since they instinctively seek refuge upon detecting their predator's odour. In other words, the predator odour directly aids in prey refuge. The harvesting of the prey is also taken into account. The model aims to provide insights into the population dynamics of an odour-mediated predator-prey system. In this paper, we will explore the same model using both an ODE (system of Ordinary differential equations) framework and an FDE (system of Fractional differential equations) framework. The ODE framework of the system offers insights into the system in the absence of memory effect, whereas the FDE framework provides insights into the system in the presence of memory effect. 
\section{Model formulation} \label{mf} This paper aims to investigate an odour-mediated predator-prey system, where organisms communicate and interact through odour signals, a common and essential feature across natural ecosystems. In these systems, an important type of diversity is the use of predator odour signals by prey to seek refuge and evade predation. Within the animal kingdom, the phenomenon of seeking refuge as a prey species in response to the odour of the predator is a commonly observed occurrence to avoid predation. Thus, it is crucial to incorporate prey refuge into the mathematical modelling of an odour-mediated predator-prey system in order to enhance its realism. Another notable diversification within an odour-mediated predator-prey system is the use of harvesting practices. It is observed that odour of a species can be used for efficient and effective resource extraction from the environment. Thus, when examining an odour mediated predator prey system, it is logical to incorporate the concept of harvesting.

Thus, in this paper, our objective is to analyse the influence of odour on a prey harvested three-species food chain model, specifically considering the effects of prey odour on predators and predator odour on prey.
Through this investigation, valuable insights into the ecological dynamics of refuge and odour effect in a prey harvested food chain can be gained. Building on the preceding discussion, we propose the subsequent model:
\begin{equation} \label{Firsteq}
       \begin{split}
           &\frac{d x_1}{dt}=  r_1 x_1 (1-x_1) -r_2 (1-m_{fp} a_1) (1+ \beta x_1) x_1 x_2- qr x_1 \\
           &\frac{d x_2}{dt}=r_3 r_2 (1-m_{fp} a_1)(1+ \beta x_1) x_1 x_2 - \frac{r_4 (1-m_{sp} a_2) x_2 x_3}{(1 + b x_2)} - d_1 x_2 \\
           &\frac{d x_3}{dt}=\frac{r_5 r_4 (1-m_{sp} a_2) x_2 x_3}{(1 + b x_2)} - d_2 x_3
       \end{split}
   \end{equation}
with initial conditions: \\
$$x_1(0)= x_1^0>0 ,x_2(0) = x_2^0 >0 ,x_3(0)=x_3^0>0$$.\\
For simplicity, taking $m_1=m_{fp} a_1$ and $m_2=m_{sp} a_2$, the system (\ref{Firsteq}) becomes

\begin{equation} \label{Final ode eq}
       \begin{split}
           &\frac{d x_1}{dt}=  r_1 x_1 (1-x_1) -r_2 (1-m_1) (1+ \beta x_1) x_1 x_2- qr x_1 \\
           &\frac{d x_2}{dt}=r_3 r_2 (1-m_{1})(1+ \beta x_1) x_1 x_2 - \frac{r_4 (1-m_2) x_2 x_3}{(1 + b x_2)} - d_1 x_2 \\
           &\frac{d x_3}{dt}=\frac{r_5 r_4 (1-m_2) x_2 x_3}{(1 + b x_2)} - d_2 x_3
       \end{split}
   \end{equation}
with initial conditions: \\
$$x_1(0)= x_1^0>0 ,x_2(0) = x_2^0 >0 ,x_3(0)=x_3^0>0$$.

 
The system under consideration, as described in equation (\ref{Final ode eq}), involves a food chain consisting of three species. The variables $x_1$, $x_2$, and $x_3$ represent the population densities of the prey, intermediate predator, and top predator, respectively. It is considered that the top predator (of density $x_3$) preys upon the intermediate predator (of density $x_2$), while the intermediate predator, in turn, hunts upon prey (of density $x_1$). Logistic growth is considered for prey, with the parameter $r_1$ denoting the intrinsic growth rate. It is considered that prey odour helps intermediate predators in hunting
and $\beta$ is the coefficient of odour effect produced by a single prey. The harvesting of the prey is also considered. The variable $q$ represents the catchability constant, while the parameter $r$ reflects the level of harvesting effort. The variable $r_2$ represents the feeding rate at which the intermediate predator consumes its prey, whereas $r_4$ represents the feeding rate at which the top predator consumes the intermediate predator. The conversion rates for the intermediate predator and top predator are denoted as $r_3$ and $r_5$ respectively. Prey and intermediate predators are considered to instinctively seek refuge when they detect the odour of their predators and it is hypothesised that the amount of predator odour is directly proportional to the number of prey individuals seeking refuge. Here, $m_1=m_{fp} a_1$ and $m_2=m_{sp} a_2$ represent the quantity of prey and intermediate predators that take refuge when they detect the odour of their predators, rendering themselves unreachable. The prey's awareness of potential threats is presumed to rise with the heightened intensity of predator odours. This results in competition to locate refuges. Consequently, as the amount of predator odour escalates, the likelihood of prey obtaining refuge diminishes. The values $m_{fp}$ and $m_{sp}$ denote the probability of prey and intermediate predator obtaining refuge in the presence of predator odour. Thus, $0<m_{fp},m_{sp}<1$. The parameter $a_1$ represents the coefficient of intensity of the intermediate predator's odour. On the other hand, the value $a_2$ represents the coefficient of intensity of the top predator's odour. Here, it is assumed that $0<m_1,m_2<1$. The variables $d_1$ and $d_2$ represent the natural death rates of the intermediate predator and top predator, respectively. Figure (\ref{redfoxgmousegrassflowchrt}) provides a visual representation of an ecological system that closely resembles the system to be investigated.
\begin{figure}[H]
         \centering
         \includegraphics[width=200mm]{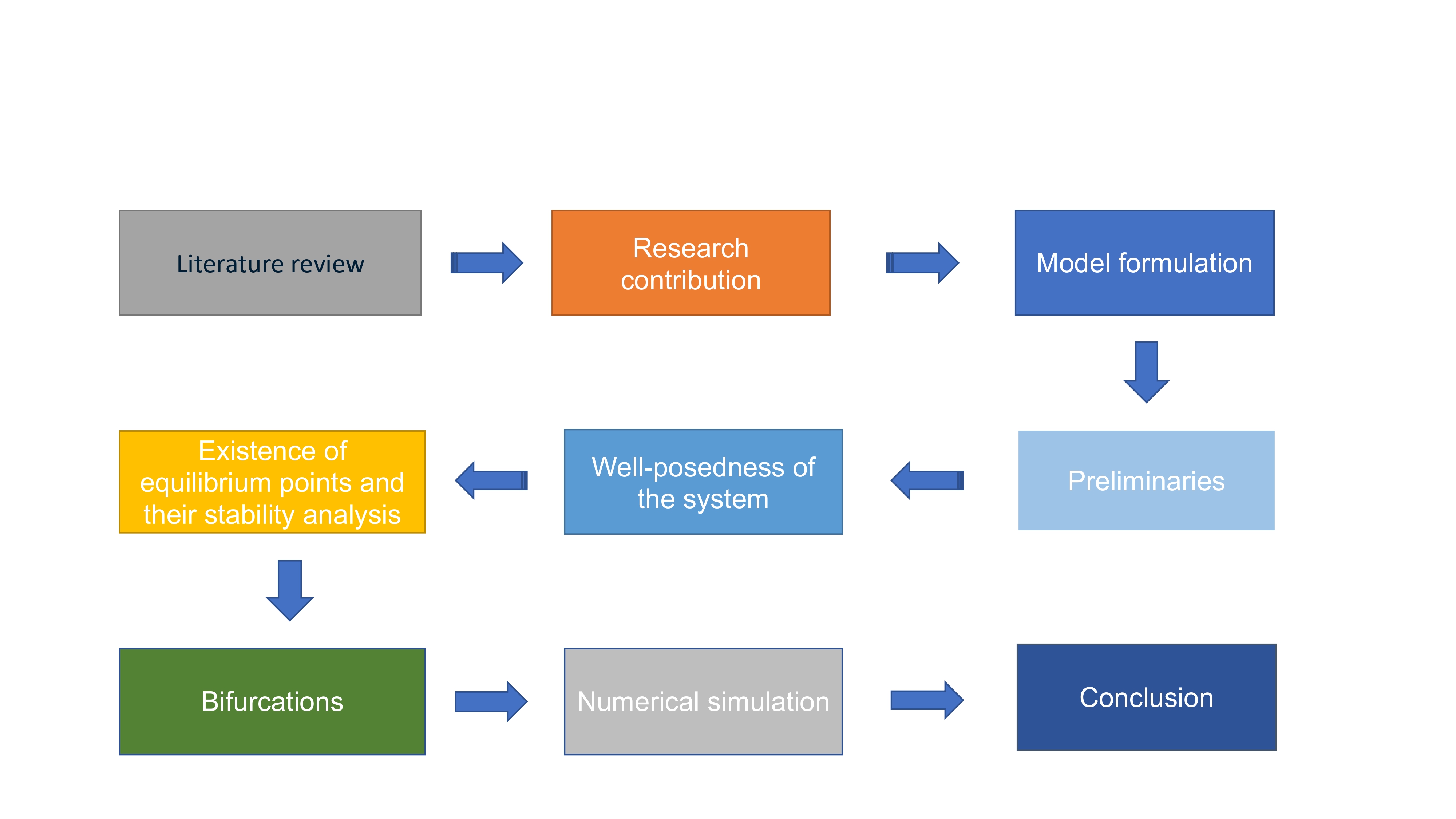}
         \caption{\emph{A flow chart illustrating the research methodology employed for the model under investigation in this scholarly article.}}
         \label{resmethod}
     \end{figure}
In order to integrate the memory effect into the system (\ref{Final ode eq}), we have introduced a Caputo-type fractional-order derivative with an order of $\alpha$ (where, $0<\alpha\le 1$) into the model system. This modification results in the following revised model:

 \begin{equation} \label{Frac eq}
       \begin{split}
         &^{c}D_{t}^{\alpha}(x_1(t))= r_1 x_1 (1-x_1) -r_2 (1-m_1) (1+ \beta x_1) x_1 x_2- qr x_1 \\
           &^{c}D_{t}^{\alpha}(x_2(t))=r_3 r_2 (1-m_1)(1+ \beta x_1) x_1 x_2 - \frac{r_4 (1-m_2) x_2 x_3}{(1 + b x_2)} - d_1 x_2 \\
           &^{c}D_{t}^{\alpha}(x_3(t))=\frac{r_5 r_4 (1-m_2) x_2 x_3}{(1 + b x_2)} - d_2 x_3
       \end{split}
   \end{equation}
   with initial conditions: $x_1(0)= x_{1}^{0}>0$ , $x_2(0) = x_{2}^{0} >0$ , $x_3(0)=x_{3}^{0}>0$. 
In this model system (\ref{Frac eq}), when the value of $\alpha=1$, then it system becomes equivalent to the system (\ref{Final ode eq}). 
   
\section{Preliminaries}\label{p} In this section, we provide a comprehensive review of various definitions, theorems, and lemmas that are pertinent to the Caputo fractional derivative.
\begin{figure}[H]
         \centering
         \includegraphics[width=190mm]{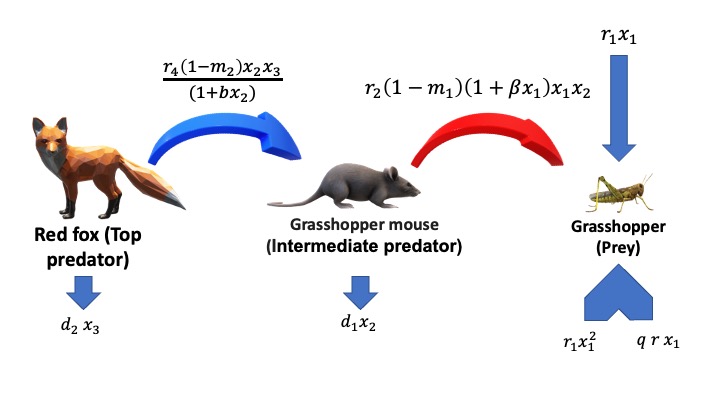}
         \caption{\emph{The graphical representation showcases the intricate relationships among the prey (grasshopper), intermediate predator (grasshopper mouse), and top predator (red fox).}}
         \label{redfoxgmousegrassflowchrt}
     \end{figure}
\begin{definition} \label{cf} \cite{pet}
    The Caputo fractional derivative with order $\alpha  \ge 0$ for the continuous function $f(t) \in C^{n} ([ t_0 , + \infty ) , R)$ is defined as
     $$^{c}_{t_0}D_{t}^{\alpha}(v(t)) = \frac{1}{\Gamma(n-\alpha)} \int_{t_{0}}^{t} \frac{f^{n}(r)}{(t-r)^{\alpha-n+1}} \, dr, $$

where $\Gamma(*)$ is the Gamma function and $t_{0}\le t$. Here, the symbol $\alpha$ represents the order of the fractional derivative and  $\alpha \in (n-1,n), n \in \mathbf{N}$ . For $\alpha \in (0,1)$ and considering the specific case of n = 1,
  $$^{c}_{t_0}D_{t}^{\alpha}(v(t)) = \frac{1}{\Gamma(1-\alpha)} \int_{t_0}^{t} \frac{f'(r)}{(t-r)^{\alpha}} \, dr. $$
\end{definition}

\begin{definition} \cite{li}
    A point $x_{0}$ is said to be an equilibrium point of the following Caputo fractional order dynamical system $^{c}_{t_0}D_{t}^{\alpha}(x(t))=\Pi(x), x(t_{0})>0$  iff $\Pi(x_{0})=0$.
\end{definition}

\begin{lemma}\cite{kilb} \label{Ml}
    Let us assume, f(t) be a continuous function on (0,T] and satisfies $^{c}_{t_0}D_{t}^{\alpha}(f(t))\le - b_1 f(t)+ b_2$, $f(0)=f_{0}>0$, $0<\alpha<1$, where $b_1 \ne 0$ and $b_1,b_2\in \mathbb{R}$. Then $$f(t)\le \left(f_{0}-\frac{b_2}{b_1} \right)E_{\alpha}(-b_{1}t^{\alpha})+\frac{b_2}{b_1}.$$
Where, $E_{\alpha}$ is the Mittag–Leffler function.
\end{lemma}

\begin{lemma}\cite{kli}\label{uniquelemma}
    We consider a system denoted as $^{c}_{t_0}D_{t}^{\alpha}(H(t))=\phi(t,H)$, where $t_{0}>0$ and $H(t_{0})=H_{t_{0}}$ is the initial condition. The parameter $\alpha$ belongs to the interval $(0,1]$, and the function $\phi: [t_{0}, \infty) \times \upsilon \to \mathbf{R}^{n\times n}$, where $\upsilon$ is a subset of $\mathbf{R}^{n}$. Now, when it satisfies the local Lipschitz condition with respect to $H\in \mathbf{R}^{n}$ \\
    $$||\phi(t,H_1)-\phi(t,H_2)||\le K ||H_{1}-H_{2}||$$
   then, the presence of a unique solution is confirmed on the interval $[t_{0}, \infty) \times \upsilon \to \mathbf{R}^{n\times n}$, where
$$||H_1(x_1,x_2,...,x_n)-H_2(y_1,y_2,...,y_n)||= \sum_{j=1}^{n} |x_{i}-y_{i}|, x_{i},y_{i}\in \mathbf{R}.$$
\end{lemma}

\begin{lemma} \cite{odib}\label{nd}
    Let us suppose that $w(t), ^{c}_{t_0}D_{t}^{\alpha}(w(t)) \in C[a,b]$ and $0 < \alpha \le 1$, then\\
    (\romannum{1}) w(t) is a non-increasing function, $\forall t \in [a,b]$, provided $^{c}_{t_0}D_{t}^{\alpha}(w(t))\le 0$.\\
    (\romannum{2}) w(t) is a non-decreasing function, $\forall t \in [a,b]$, provided $ ^{c}_{t_0}D_{t}^{\alpha}(w(t))\ge 0$.
\end{lemma}

\begin{lemma} \cite{pet,podlub} \label{l3}
    Let us consider, a Caputo fractional order dynamical system as follows:
    $$^{c}_{t_0}D_{t}^{\alpha}(g(t))=f(g(t)), g(0)=g_{0}\in R^{n}$$  
    where, $0<\alpha<1$, $g(t)=(g_1(t),...,g_n(t))^{T}\in R^{n}$ and $f : [f_1, f_2, . . . , f_n] :
R^{n}\to R^{n}$. Let, $f(g^{*})=0$, then $g^{*}$ is an equilibrium point of the abovementioned fractional system. Let us assunme $J(g^{*})=\frac{\delta(\varphi_1,\varphi_2,....,\varphi_n)}{\delta(g_{1},g_{2},....,g_{n})}$ be the Jacobian matrix of the above system at equilibrium point $g = g^{*}$ and  $\tau_{j}$, j=1,2,3,...,n are the eigenvalues of $J(g^{*})$. Then the equilibrium point $g^{*}$ is stable if and only if $|arg(\tau_{j})|\ge \frac{\alpha \pi}{2}$ and eigenvalues with $|arg(\tau_{j})|= \frac{\alpha \pi}{2}$ have the
same geometric multiplicity and algebraic multiplicity. On the other hand, the equilibrium point $g^{*}$ is locally asymptotically stable if and only if $|arg(\tau_{j})|>\frac{\alpha \pi}{2}$ and unstable if and only if there exists
eigenvalue $\tau_{j}$ of the Jacobian matrix $J(g^{*})$ such that $|arg(\tau_{j})|<\frac{\alpha \pi}{2}$.
\end{lemma}

\section{Well-posedness of the systems} \label{wp} Establishing the well-posedness of a predator-prey system is essential to guarantee that the model is mathematically sound, biologically meaningful, and practically applicable. This ensures that the model will operate predictably, accurately represent real-world interactions, and function as a reliable instrument for scientific research and practical applications in ecology and conservation. In this section, we begin by examining the well-posedness of the system (\ref{Final ode eq}) and then proceed to analyse the system (\ref{Frac eq}).
\subsection{Non-negativeness of solutions}
\begin{theorem}
    The solutions of the system (\ref{Frac eq}) starting from $(x_1(0), x_2(0),x_3(0)) \in R^{3}_{+}$ are non-negative.
\end{theorem}
\begin{proof}
    To commence, let us demonstrate the non-negativity of the solutions $x_1(t)$ of the system (\ref{Frac eq}) i.e., $x_1(t)\ge 0$, $\forall t \ge t_{0}$. To prove the non-negativity of the solutions of the system (\ref{Frac eq}), we follow the approach given in \cite{li}. Now, let us assume that the above inequality is not true, then $\exists$ a constant $t_1$, $t_{0}\le t < t_{1}$ such that \\
\begin{equation}\label{nn}
\begin{split}
        &x_1(t)>0, t_{0}\le t < t_1\\
        &x_1(t_1)=0,\\
        &x_1(t_1^{+})<0
\end{split}
\end{equation}
 Now, using the first equation of the system (\ref{Frac eq}), we have\\
    $$^{c}_{t_0}D_{t_1}^{\alpha}x_1(t_1)|_{x_1(t_1)=0}=0.$$
By the use of Lemma (\ref{nd}), we find that $$x_1(t_1^{+})=0.$$ This leads to a contradiction since $x_1(t_1^{+})<0$. Thus, $x_1(t)\ge 0$, $\forall t \in [t_{0}, \infty)$. In a similar manner, we can ensure  $x_2(t) \ge 0$, $x_3(t) \ge 0, \forall t \in [t_{0}, \infty)$. Hence, the theorem.
\end{proof}

\subsection{Boundedness of solutions}

\begin{theorem}
    All solutions of system (\ref{Frac eq}) are bounded.
\end{theorem}
\begin{proof}
     To show the boundedness of all the solutions of the system (\ref{Frac eq}), we shall contemplate a function, $W(t)=x_1(t)+\frac{x_2(t)}{r_3}+\frac{x_3(t)}{r_3 r_5}$. Thus, we have

\begin{equation*} \label{}
       \begin{split}
          ^{c}D_{t}^{\alpha}(W(t))&=^{c}D_{t}^{\alpha}(x_1(t))+^{c}D_{t}^{\alpha}(\frac{x_2(t)}{r_3}) +^{c}D_{t}^{\alpha}(\frac{x_3(t)}{r_3 r_5})\\
          &=r_1x_1-r_1 x_1^2 -qr x_1- \frac{d_1 x_2}{r_3}- \frac{d_2 x_3}{r_3r_5}
       \end{split}
\end{equation*}
Now, considering any real number $\Delta$ such that

\begin{equation*} \label{}
       \begin{split}
          ^{c}D_{t}^{\alpha}(W(t))+\Delta W(t) &=r_1x_1-r_1 x_1^2 -qr x_1- \frac{d_1 x_2}{r_3}- \frac{d_2 x_3}{r_3r_5}+ \Delta (x_1(t)+\frac{x_2(t)}{r_3}+\frac{x_3(t)}{r_3 r_5})\\
          &\le r_1x_1-r_1 x_{1}^2+\Delta x_1+\frac{x_2}{r_3}(\Delta-d_1)+\frac{x_3}{r_3r_5}(\Delta-d_2)
       \end{split}
\end{equation*}
Let us consider, $0<\Delta<min(d_1,d_2)$ then we have,
\begin{equation*} \label{}
       \begin{split}
        ^{c}D_{t}^{\alpha}(W(t))+\Delta W(t) \le \frac{(r_1+\Delta)^2}{4r_1}=\Omega
       \end{split}
\end{equation*}

\textbf{Case 1:} When $0<\alpha<1$\\

We utilise Lemma (\ref{Ml}) provided in Section 3, then we obtain

\begin{equation*} \label{}
       \begin{split}
          W(t)\le (W(0)-\frac{(r_1+\Delta)^2}{4r_1})E_{\alpha}(-dt^{\alpha})+\frac{(r_1+\Delta)^2}{4r_1} \to \frac{(r_1+\Delta)^2}{4r_1}, \text{as \emph{t}}  \to \infty
       \end{split}
\end{equation*}

where, $E_{\alpha}$ is the Mittag–Leffler function. Hence,  $$W(t)\le \frac{(r_1+\Delta)^2}{4r_1}, \text{as \emph{t}}  \to \infty.$$
Therefore, all the solutions of the system described in equation (\ref{Frac eq}) for $0<\alpha<1$ originating in region $R^{3}_{+}$ enters the region $\Upsilon$ defined by

$$\Upsilon=\Bigl\{(x_1,x_2,x_3)\in R^{3}_{+}: W \le \frac{(r_1+\Delta)^2}{4r_1}+\epsilon, \epsilon > 0\bigl\}.$$

\textbf{Case 2:} When $\alpha=1$\\

We apply the notions of differential inequality, we obtain
$0<W(t)\le\frac{\Omega}{\Delta}(1-e^{-\Delta t})+W(0)e^{-\Delta t}$. However, $0<W(t)\le \frac{\Omega}{\Delta}$, when $t\to \infty$. Therefore, all the solutions of the system (3) are restricted to the region $\zeta=  \left((x_1,x_2,x_3): 0 \le x_1+\frac{x_2}{r_3} +\frac{x_3}{r_3 r_5} \le \frac{\Omega}{\Delta} + \Theta, \forall \Theta > 0\right)$. The boundedness of the system (\ref{Frac eq}) is consequently established.

\end{proof}

\subsection{Existence and uniqueness}

\begin{theorem}
For nonnegative initial conditions, system (\ref{Frac eq}) always exhibits unique solutions.
\end{theorem}
\begin{proof}
We will utilise the methodology described in \cite{li1} in order to prove the existence of unique solutions of the system (\ref{Frac eq}).
Let us consider, the region
    $$F=\{ (x_1,x_2,x_3)\in R^{3}:max(|x_1|,|x_2|,|x_3|)\le N\}$$
    Let, $x=(x_1,x_2,x_3) \in F $ and $\Bar{x}=(\Bar{x_1},\Bar{x_2},\bar{x_3})\in F$. Now, let us consider, a function \\
    $$G(x)= (G_1(x), G_2(x), G_3(x))$$ where,
$$G_1(x)= r_1 x_1 (1-x_1) -r_2 (1-m_1) (1+ \beta x_1) x_1 x_2- qr x_1,$$ $$G_2(x)=r_3 r_2 (1-m_1)(1+ \beta x_1) x_1 x_2 - \frac{r_4 (1-m_2) x_2 x_3}{(1 + b x_2)} - d_1 x_2,$$
$$G_3(x)=\frac{r_5 r_4 (1-m_2) x_2 x_3}{(1 + b x_2)} - d_2 x_3$$
Now, we have
   \begin{equation} \label{10}
       \begin{split}
           &||G(x)-G(\bar{x})||=|G_1(x)-G_1(\bar{x})|+|G_2(x)-G_2(\bar{x})|+|G_3(x)-G_3(\bar{x})| 
       \end{split}
    \end{equation}
  we find,\\
   \begin{equation} \label{11}
       \begin{split}
           |G_1(x)-G_1(\bar{x})|&=|r_1 x_1 (1-x_1) -r_2 (1-m_1) (1+ \beta x_1) x_1 x_2- qr x_1-r_1 \bar{x_1} (1-\bar{x_1})\\
           &+r_2 (1-m_1) (1+ \beta \bar{x_1}) \bar{x_1} \bar{x_2}+ qr \bar{x_1}|\\
           &\le |r_1(x_1-\bar{x_1})|+|r_1(x_1^2-\bar{x_1}^2)|+|r_2(1-m_1)(\bar{x_1}\bar{x_2}-x_1 x_2)|\\
           &+|r_2(1-m_1)\beta (\bar{x_1}^2\bar{x_2}-x_1^2 x_2)|+|q r (x_1-\bar{x_1})| \\
           &\le r_1|x_1-\bar{x_1}|+r_1|x_1-\bar{x_1}||x_1+\bar{x_1}|+r_2(1-m_1)|\bar{x_2}(x_1-\bar{x_1})+x_1 (x_2-\bar{x_2})|\\
           &+|r_2(1-m_1)\beta (\bar{x_1}^2\bar{x_2}-x_1^2 x_2)|+|q r (x_1-\bar{x_1})|\\
           &\le r_1|x_1-\bar{x_1}|+r_1|x_1-\bar{x_1}||x_1+\bar{x_1}|+r_2(1-m_1)|\bar{x_2}(x_1-\bar{x_1})+x_1 (x_2-\bar{x_2})|\\
            &+r_2(1-m_1)\beta |(-x_1x_2(x_1-\bar{x_1})-\bar{x_2}^2(x_2-\bar{x_2})-\bar{x_1}x_2(x_1-\bar{x_1}))|+|q r (x_1-\bar{x_1})|\\
       \end{split}
\end{equation}
In a Similar way, we get
\begin{equation} \label{12}
       \begin{split}
           |G_2(x)-G_2(\bar{x})|&\le r_3r_2(1-m_1)|(\bar{x_1}\bar{x_2}-x_1 x_2)|+ r_3 r_2(1-m_1)\beta |(\bar{x_1}^2\bar{x_2}-x_1^2 x_2)||+d_1|x_2-\bar{x_2}|\\
           &+r_4(1-m_2)|\frac{x_2 x_3}{(1+b x_2)}-\frac{\bar{x_2} \bar{x_3}}{(1+b \bar{x_2})}|\\
           |G_3(x)-G_3(\bar{x})|&\le  r_5 r_4(1-m_2)|\frac{x_2 x_3}{(1+b x_2)}-\frac{\bar{x_2} \bar{x_3}}{(1+b \bar{x_2})}| +d_2|x_3-\bar{x_3}|
       \end{split}
       \end{equation}

From equations (\ref{10}), (\ref{11}), and (\ref{12}), we get

\begin{equation} \label{13}
       \begin{split}
           ||G(x)-G(\bar{x})|| &\le r_1|x_1-\bar{x_1}|+r_1|x_1-\bar{x_1}||x_1+\bar{x_1}| + r_2(1-m_1)(1+r_3)|\bar{x_2}(x_1-\bar{x_1})|\\
           &+r_2(1-m_1)(1+r_3) |x_1 (x_2-\bar{x_2})|+ r_2(1-m_1)\beta (1+r_3)|-x_1x_2(x_1-\bar{x_1})|\\
           &+ r_2(1-m_1)\beta (1+r_3)|\bar{x_2}^2(x_2-\bar{x_2})|+r_2(1-m_1)\beta (1+r_3)|-\bar{x_1}x_2(x_1-\bar{x_1})|\\
           &+r_4(1-m_2)(1+r_5)|\frac{\bar{x_3}(x_2-\bar{x_2})+(x_3-\bar{x_3})(1+b \bar{x_2})x_2}{(1+b x_2)(1+b \bar{x_2})}|+|q r (x_1-\bar{x_1})|\\
           &+d_1|x_2-\bar{x_2}|+d_2|x_3-\bar{x_3}|\\
           &\le (r_1+r_1 |x_1+\bar{x_1}|+r_2(1-m_1)(1+r_3)|\bar{x_2}|+r_2(1-m_1)\beta (1+r_3)|x_1x_2|\\
           &+r_2(1-m_1)\beta (1+r_3)|\bar{x_1}x_2|+ |q r|)|x_1-\bar{x_1}|+(r_2(1-m_1)(1+r_3) |x_1|\\
           &+ r_2(1-m_1)\beta (1+r_3)|\bar{x_2}^2|+r_4(1-m_2)(1+r_5)|\frac{\bar{x_3}(x_2-\bar{x_2})}{(1+b x_2)(1+b \bar{x_2})}|+\\
           &|d_1|)|x_2-\bar{x_2}|+(r_4(1-m_2)(1+r_5)|\frac{x_2}{(1+b x_2)}|+|d_2|)|x_3-\bar{x_3}|\\
           &\le Y|x-\bar{x}|
       \end{split}
       \end{equation}

where, $Y=max\{y_1,y_2,y_3\}$ and $y_1=r_1+r_1 |x_1+\bar{x_1}|+r_2(1-m_1)(1+r_3)|\bar{x_2}|+r_2(1-m_1)\beta (1+r_3)|x_1x_2|+r_2(1-m_1)\beta (1+r_3)|\bar{x_1}x_2|+ |q r|$, $y_2=r_2(1-m_1)(1+r_3) |x_1|+ r_2(1-m_1)\beta (1+r_3)|\bar{x_2}^2|+r_4(1-m_2)(1+r_5)|\frac{\bar{x_3}(x_2-\bar{x_2})}{(1+b x_2)(1+b \bar{x_2})}|+|d_1|$, and $y_3=r_4(1-m_2)(1+r_5)|\frac{x_2}{(1+b x_2)}|+|d_2|$.\\

Therefore, the function G(x) fulfils the Lipschitz condition. Therefore, based on Lemma (\ref{uniquelemma}), it may be inferred that the system described in equation (\ref{Frac eq}) has a unique solution inside the specified domain F.

\end{proof}

\begin{figure}[H]
         \centering
         \includegraphics[width=180mm]{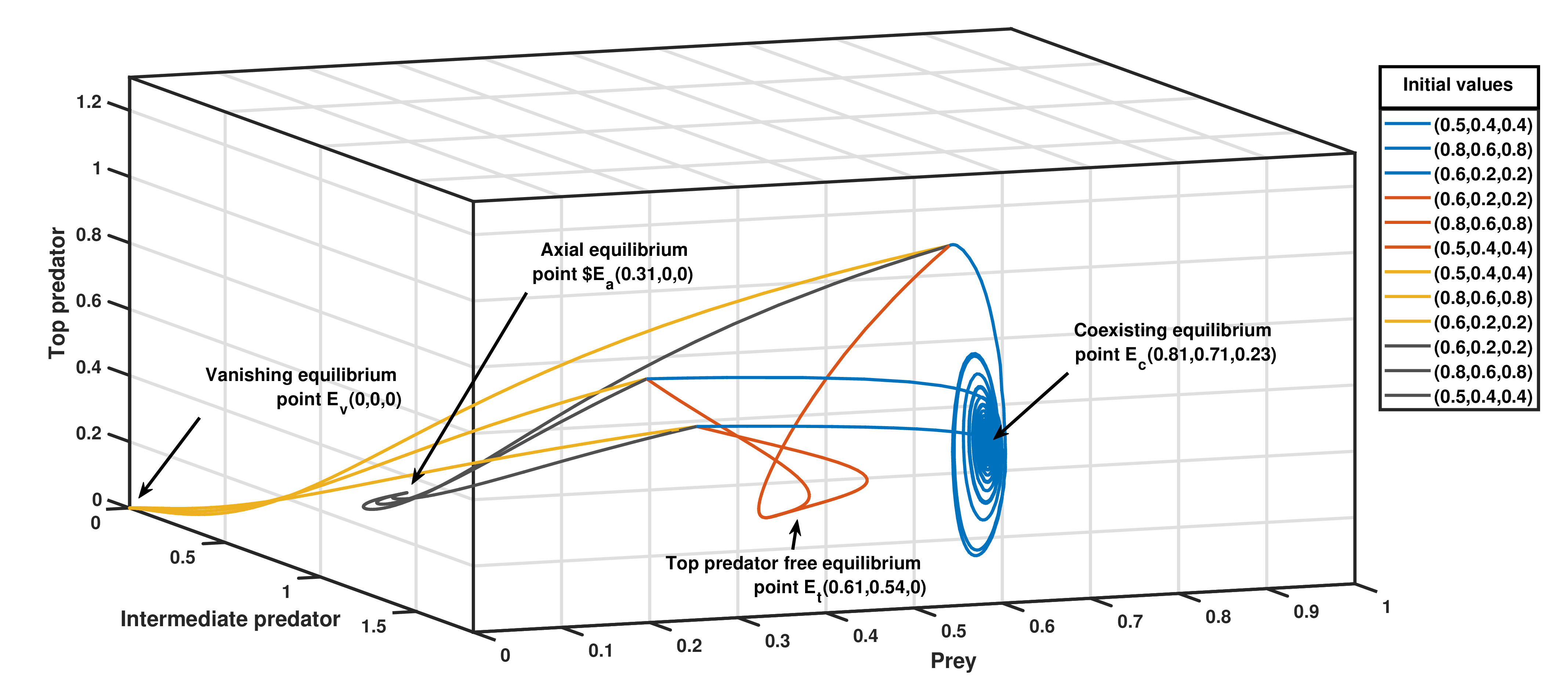}
         \caption{\emph{ This illustration demonstrates the local stability of all four ecologically feasible equilibrium points for different sets of parameter values.}}
         \label{lsode}
     \end{figure}

\section{Existence of equilibrium points and their stability analysis}\label{es}
In the context of a predator-prey dynamical system, equilibrium points refer to specific values of the population sizes at which the rates of change of the populations cease to exist. These equilibrium points play a crucial role in understanding the dynamics of predator-prey interactions. These points delineate a state of equilibrium in which populations remain constant and exhibit no changes over time. A point $(x_1^{eq},x_2^{eq},x_3^{eq})$ is said to be an equilibrium point of the system (\ref{Frac eq}) for $0<\alpha\le 1$, if they satisfy the following equations:\\
$$r_1 x_1 (1-x_1) =r_2 (1-m_1) (1+ \beta x_1) x_1 x_2+ qr x_1,$$
$$r_3 r_2 (1-m_1)(1+ \beta x_1) x_1 x_2 = \frac{r_4 (1-m_2) x_2 x_3}{(1 + b x_2)} + d_1 x_2,$$
$$\frac{r_5 r_4 (1-m_2) x_2 x_3}{(1 + b x_2)} = d_2 x_3.$$
Now, solving these equations we get four biologically feasible equilibrium points and these are given by 
\begin{enumerate}
\item  Vanishing equilibrium point $E_{v}(0,0,0)$,
\item Axial equilibrium point $E_{a}(1-\frac{q r}{r_1},0,0)$,
\item Top predator free equilibrium point $E_{t}(A,B,0)$, and
\item Coexisting equilibrium point $E_{c}(C,D,E)$.
\end{enumerate}
Here, vanishing equilibrium point $E_{v}(0,0,0)$ represents the situation in which all three species perish.  Axial equilibrium point $E_{a}(1-\frac{q r}{r_1},0,0)$  represents a situation in which only the prey persists and all the predators vanish. Top predator free equilibrium point $E_{t}(A,B,0)$ represents the situation in which only the top predator vanishes while coexisting equilibrium point $E_{c}(C,D,E)$ signifies the situation in which all three species coexist. Morever, $$A=\frac{-r_2 r_3 + m_1 r_2 r_3 + \omega_1}{
 2 \beta r_2 r_3 - 2 \beta m_1 r_2 r_3},$$ $$B=\frac{-r_1 ((-1 + m_1) r_2 r_3 + 
   \omega_1) + 
 \beta (2 d_1 r_1 + (r q - r_1) ((-1 + m_1) r_2 r_3 + \omega_1))}{2 \beta^2 d_1 (-1 + 
   m_1) r_2},$$ $$C=\frac{b d_2 (-r q + r_1) + d_2 (r_2 - m_1 r_2) - (-1 + m_2) (r q - r_1) r_4 r_5}{
b d_2 r_1 + \beta d_2 (-1 + m_1) r_2 + (-1 + m_2) r_1 r_4 r_5},$$ $$D=-\frac{d_2}{b d_2 + (-1 + m_2) r_4 r_5},$$ and \\$$E=\frac{r_5 (\beta^2 d_1 d_2^2 (-1 + m_1)^2 r_2^2 + 
    b^2 d_2^2 \omega_2 - (-1 + 
       m_2) r_1 r_4 r_5 \omega_3 + 
    \beta (-1 + m_1) (-1 + 
       m_2) r_2 r_4 r_5 \omega_4 + 
    b d_2 \omega_5)}{(b d_2 + (-1 + m_2) r_4 r_5) (b d_2 r_1 + 
    \beta d_2 (-1 + m_1) r_2 + (-1 + m_2) r_1 r_4 r_5)^2}.$$\\
   Here, $\omega_1=\sqrt{(-1 + m_1) r_2 r_3 (-4 \beta d_1 - r_2 r_3 + m_1 r_2 r_3)}$, $\omega_2=(d_1 r_1^2 + (-1 + m_1) (r q - r_1) (\beta r q - r_1 - 
          \beta r_1) r_2 r_3)$, $\omega_3=(d_2 (-1 + m_1)^2 r_2^2 r_3 - (-1 + 
          m_2) (d_1 r_1 - (-1 + m_1) (r q - r_1) r_2 r_3) r_4 r_5)$, $\omega_4=(2 d_1 d_2 r_1 + (r q - 
          r_1) r_3 (d_2 (-1 + m_1) r_2 + (-1 + m_2) (r q - r_1) r_4 r_5))$, $\omega_5=(r_1 (-d_2 (-1 + m_1)^2 r_2^2 r_3 + 
          2 (-1 + m_2) (d_1 r_1 - (-1 + m_1) (r q - r_1) r_2 r_3) r_4 r_5) + 
       \beta (-1 + 
          m_1) r_2 (2 d_1 d_2 r_1 + (r q - r_1) r_3 (d_2 (-1 + m_1) r_2 + 
             2 (-1 + m_2) (r q - r-1) r_4 r_5)))$.\\\\
\textbf{Conditions of existence:}
\begin{enumerate}
\item  Vanishing equilibrium point $E_{v}(0,0,0)$ always exists.
\item Axial equilibrium point $E_{a}(1-\frac{q r}{r_1},0,0)$ exists iff $r < \frac{r_1}{q}$.
\item Top predator free equilibrium point $E_{t}(A,B,0)$ exists iff $r < \frac{r_1}{q}$ and 
$r_2 > \frac{-d_1 r_1^2}{-\beta r^2 q^2 r_3 + \omega_6}$.
\item Coexisting equilibrium point $E_{c}(C,D,E)$ exists iff $r_1 > \frac{d_2 (-1 + m_1) r_2}{b d_2 + (-1 + m_2) r_4 r_5}$, $r_3 > -\frac{d_1 \omega_7^2}{(-1 + m_1) \omega_8 r_2 \omega_9}$, $r_4 > \frac{b d_2}{r_5 - m_2 r_5}$, and $r < \frac{b d_2 r_1 + d_2 (r_2 - m_1 r_2) + (-1 + m_2) r_1 r_4 r_5}{q (b d_2 + (-1 + m_2) r_4 r_5)}$.
\end{enumerate}

Here, $\omega_6= \beta r^2 m_1 q^2 r_3 + r q r_1 r_3 + 
   2 \beta r q r_1 r_3 - r m_1 q r_1 r_3 - 2 \beta r m_1 q r_1 r_3 - r_1^2 r_3 - 
   \beta r_1^2 r_3 + m_1 r_1^2 r_3 + \beta m_1 r_1^2 r_3$, $\omega_7=(b d_2 r_1 + 
     \beta d_2 (-1 + m_1) r_2 + (-1 + m_2) r_1 r_4 r_5)$, $\omega_8=(\beta r q - 
     r_1 - \beta r_1)$, $\omega_9=(b d_2 + (-1 + m_2) r_4 r_5) (b d_2 (r q - r_1) + 
     d_2 (-1 + m_1) r_2 + (-1 + m_2) (r q - r_1) r_4 r_5)$. \\\\
\textbf{Local stability:}
\begin{theorem} \label{lsva}
    The vanishing equilibrium point $E_{v}(0,0,0)$ is locally stable iff $r > \frac{r_1}{q}$ for all $\alpha \in (0,1]$.
\end{theorem}
\begin{proof} The diagonal elements of the Jacobian matrix $J_{va}$ are the eigenvalues of $J_{va}$. The entries $-r q + r_1,  -d_1,$ and $-d_2$ are the eigenvalues of $J_{va}$. Now, utilising Lemma (\ref{l3}), we get the vanishing equilibrium point $E_{v}(0,0,0)$ is locally stable iff $r > \frac{r_1}{q}$ for all $\alpha \in (0,1)$. Moreover, it is readily apparent that for $\alpha=1$, $E_{v}(0,0,0)$ is locally stable iff $r > \frac{r_1}{q}$.  Therefore, the result.

\end{proof}

\begin{figure}[H]
     \centering
     \begin{subfigure}{0.45\textwidth}
         \centering
         \includegraphics[width=\textwidth]{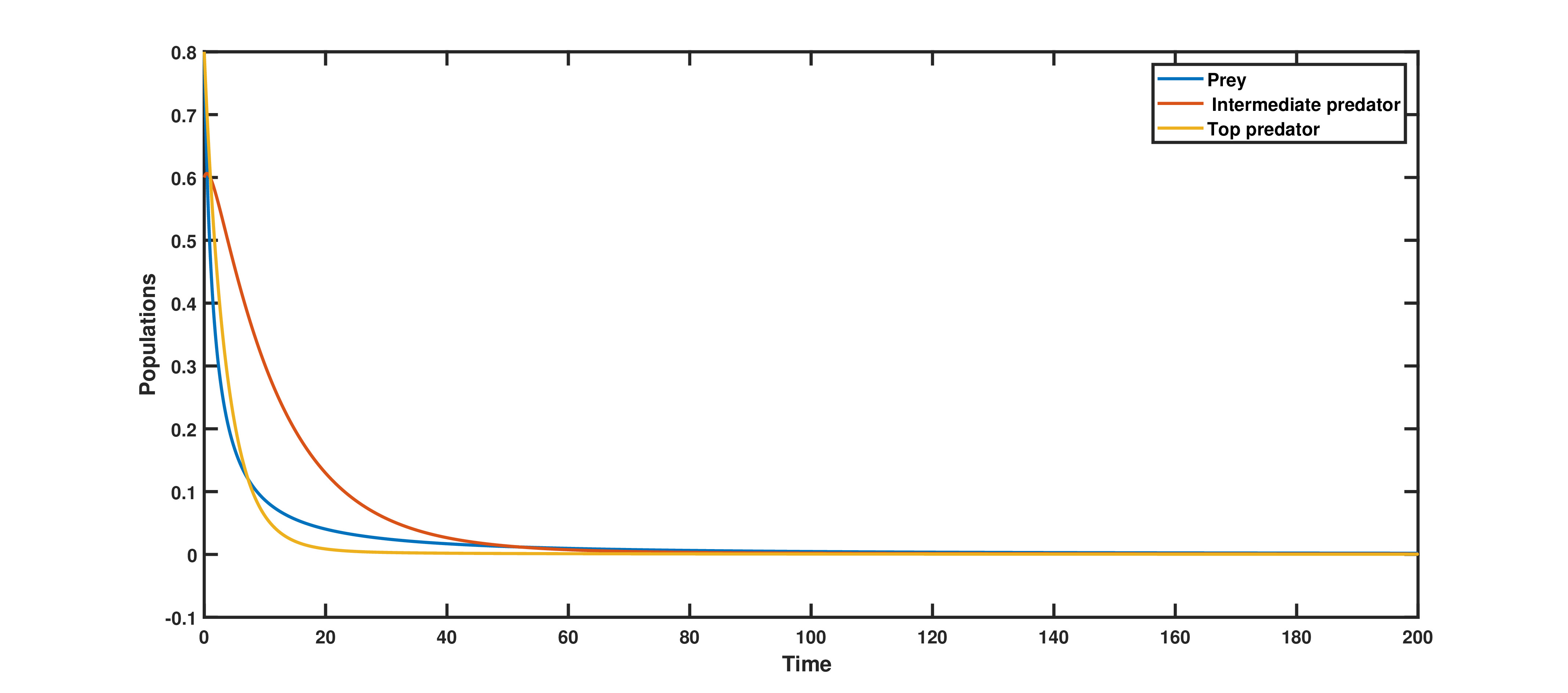}
         \caption{\emph{Local stability of $E_v$}}
         \label{lsvafig}
     \end{subfigure}
      \hfill
     \begin{subfigure}{0.45\textwidth}
         \centering
         \includegraphics[width=\textwidth]{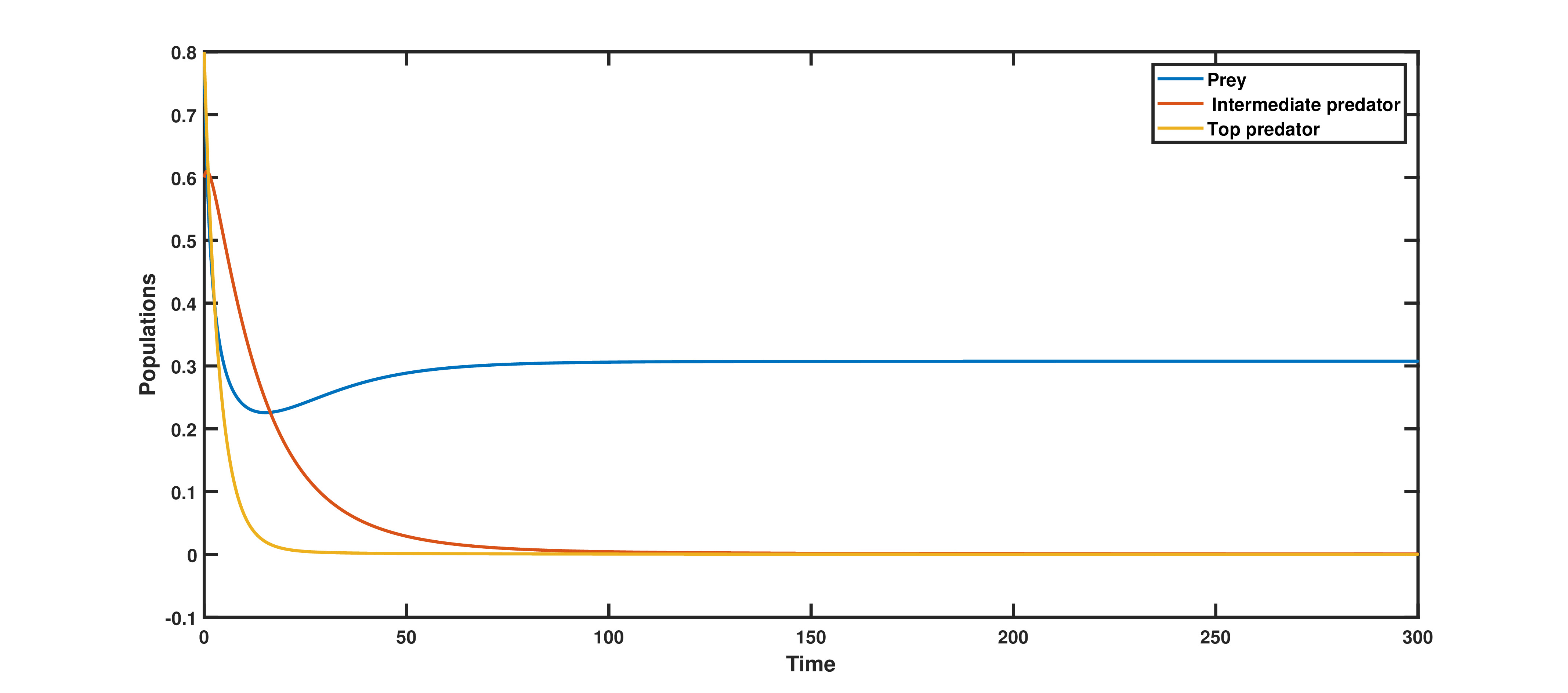}
         \caption{\emph{Local stability of $E_a$}}
         \label{lsaxfig}
     \end{subfigure}
     \hfill
     \begin{subfigure}{0.45\textwidth}
         \centering
         \includegraphics[width=\textwidth]{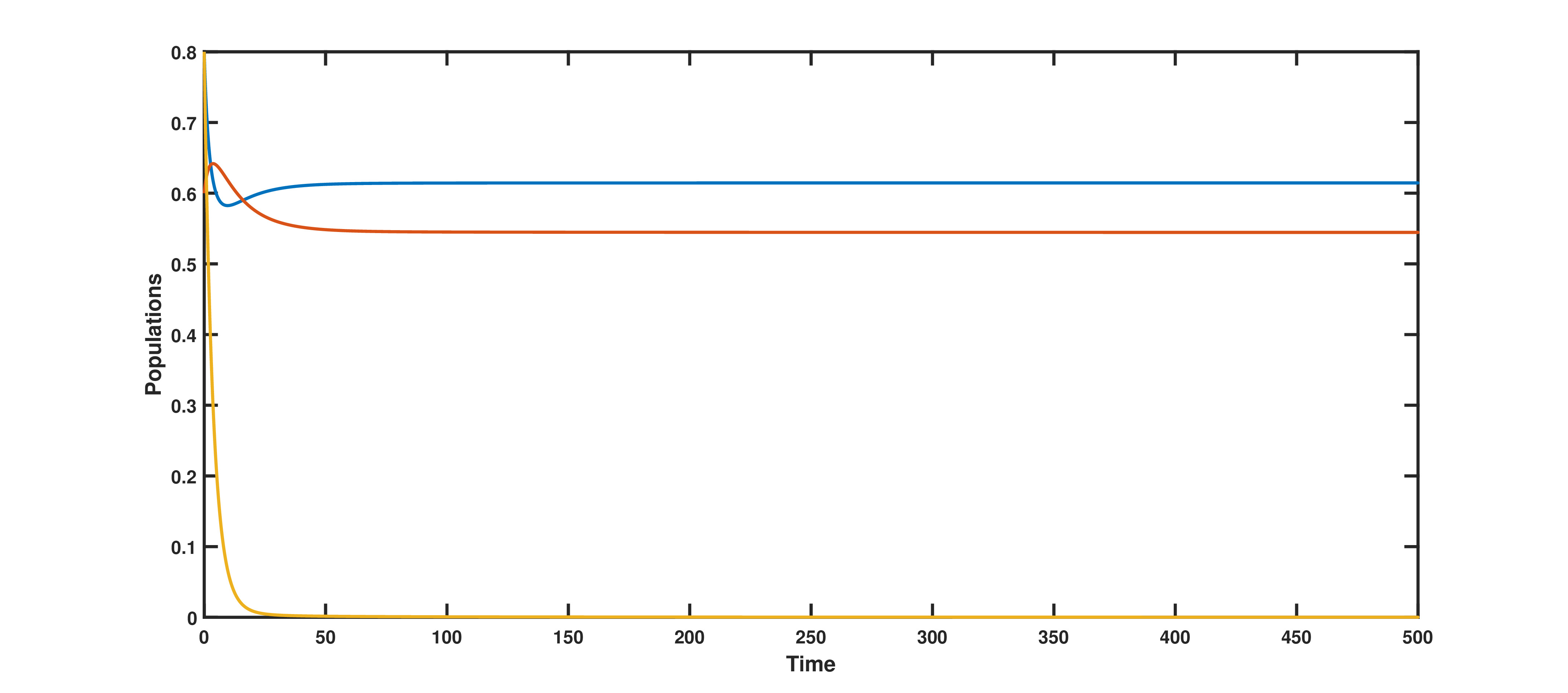}
         \caption{\emph{Local stability of $E_t$}}
         \label{lstofig}
     \end{subfigure}
     \hfill
     \begin{subfigure}{0.45\textwidth}
         \centering
         \includegraphics[width=\textwidth]{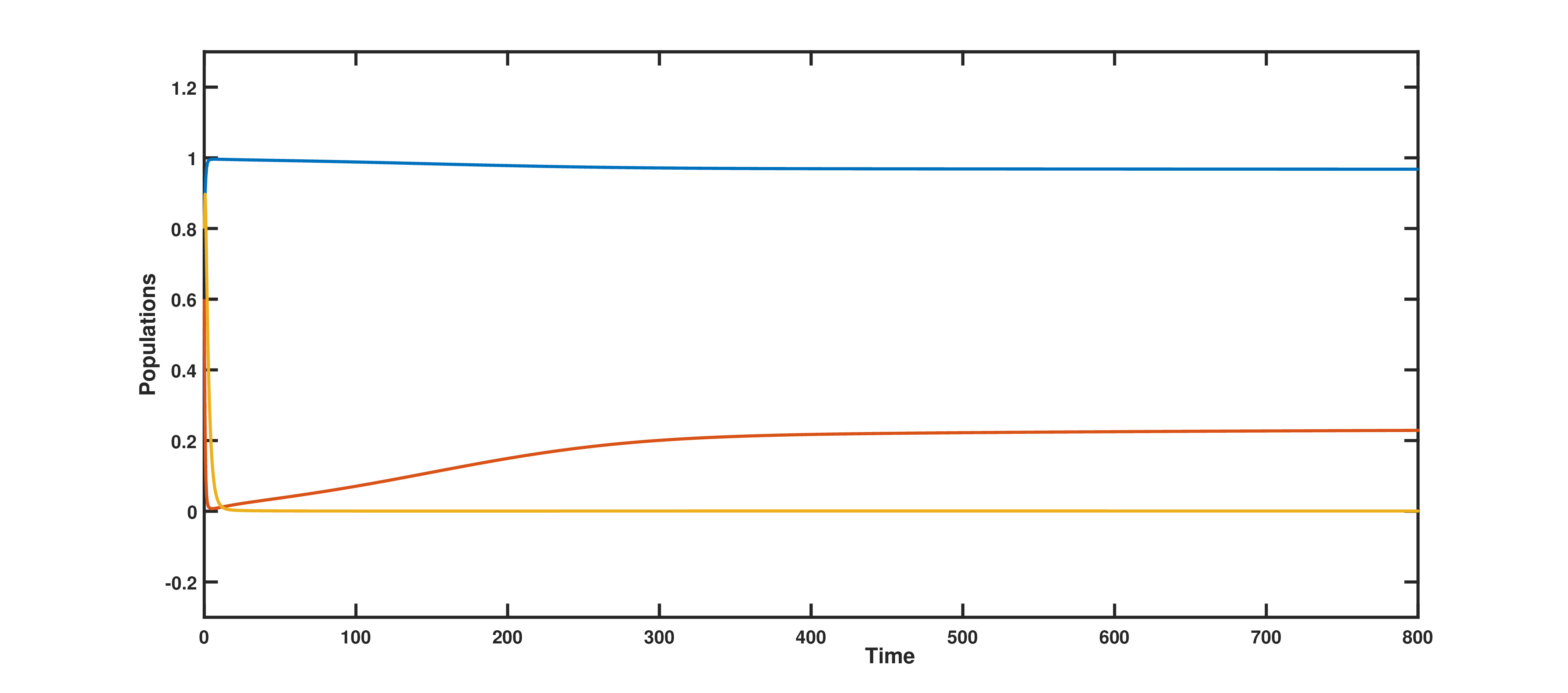}
         \caption{\emph{Local stability of $E_c$}}
         \label{lscofig}
     \end{subfigure}
       \caption{\emph{ This diagram illustrates the local stability of all the biologically viable fixed points of the system (\ref{Frac eq}) at a fractional order of $\alpha=0.98$.}}
        \label{lsfracfig}
\end{figure}

\begin{theorem}\label{lsax}
    The axial equilibrium point $E_{a}(1-\frac{q r}{r_1},0,0)$ is locally stable iff $r < \frac{r_1}{q}$ and $r_2 < \frac{-d_1 r_1^2}{-\beta r^2 q^2 r_3 + \omega_6}$ for all $\alpha \in (0,1]$.
\end{theorem}
\begin{proof} The Jacobian matrix of the system (\ref{Frac eq}) around $E_{a}(1-\frac{q r}{r_1},0,0)$ is as follows:

    $$J_{ax}=\begin{bmatrix}
r q - r_1 & \frac{(-1 + m1) (q r - r_1) (\beta q r - r_1 - \beta r_1) r_2}{r_1^2} & 0 \\
0 & \frac{-d_1 r_1^2 + \beta r^2 q^2 r_2 r_3 - r_2\omega_6}{r_1^2} & 0 \\
0 & 0 & -d_2
\end{bmatrix}
$$ 
Obviously, the diagonal elements $\xi_{a,1}=r q - r_1$, $\xi_{a,2}=\frac{-d_1 r_1^2 + \beta r^2 q^2 r_2 r_3 - r_2\omega_6}{r_1^2}$, and $\xi_{a,3}=-d_2$ are the eigenvalues of the Jacobian matrix $J_{va}$. Now, to study the local stability of the system (\ref{Frac eq}) around the equilibrium point $E_{a}$, we utilise Lemma (\ref{l3}). Here, $|arg(\xi_{a,1})|=\pi >\frac{\gamma \pi}{2}$ if $r < \frac{r_1}{q}$, $|arg(\xi_{a,2})|=\pi >\frac{\gamma \pi}{2}$ if $r_2 < \frac{-d_1 r_1^2}{-\beta r^2 q^2 r_3 + \omega_6}$, and $|arg\xi_{a,3})|=\pi >\frac{\gamma \pi}{2}$. Hence, the system (\ref{Frac eq}) is locally asymptotically stable around the axial equilibrium point $E_{a}$ for all $\alpha \in (0,1)$ if the conditions (\romannum{1}) $r < \frac{r_1}{q}$ and (\romannum{2}) $r_2 < \frac{-d_1 r_1^2}{-\beta r^2 q^2 r_3 + \omega_6}$ hold. Furthermore, $E_{a}(1-\frac{q r}{r_1},0,0)$ is locally stable for $\alpha=1$ iff all the eigenvalues of $J_{ax}$ are negative and all the eigenvalues of $J_{ax}$ are negative iff $r < \frac{r_1}{q}$ and $r_2 < \frac{-d_1 r_1^2}{-\beta r^2 q^2 r_3 + \omega_6}$. Hence, the proof.
\end{proof}
\begin{theorem}\label{lsto}
     The top predator free equilibrium point $E_{t}(A,B,0)$ is locally stable iff $M_1>0$, $M_2>0$, $M_3>0$, and $M_1 M_2-M_3>0$. The definitions of the symbols $M_i$ are provided within the proof.
\end{theorem}
\begin{proof}
\textbf{Case 1:} When $0<\alpha<1$:
The Jacobian matrix of the system (\ref{Frac eq}) around the top predator free equilibrium point $E_{t}(A,B,0)$ is $J_{to}$, already given in subsection (\ref{es}). The characteristic equation of $J_{to}$ can be written as $$(t_{33}-\nu)(\nu^2-C_{1}^{*}\nu+C_{2}^{*})=0,$$ where $C_1= t_{11}$ and $C_2=-t_{12} t_{21}$. Let us consider the three eigenvalues of $J_{to}$: $\xi_{t,1}$, $\xi_{t,2}$, and $\xi_{t,3}$. Let the equation $t_{33}-\nu$ give $\xi_{t,1}$, and from the equation $\nu^2-C_{1}^{*}\nu+C_{2}^{*}$, we get $\xi_{t,2}$ and $\xi_{t,3}$. Now, according to Lemma (\ref{l3}), the system (\ref{Frac eq}) shows local stability around the equilibrium point $E_t$ if $|arg(\xi_{t,1})|=\pi>\frac{\gamma \pi}{2}$, $|arg(\xi_{t,2})|=\pi>\frac{\gamma \pi}{2}$, and $|arg(\xi_{t,3})|=\pi>\frac{\gamma \pi}{2}$. If the conditions (\romannum{1}) $t_{33}<0$, (\romannum{2}) $t_{11}^2 + 4 t_{12} t_{21}<0$, and (\romannum{3}) $t_{11}<0$ satisfies, then $|arg(\xi_{t,1})|=\pi>\frac{\gamma \pi}{2}$, $|arg(\xi_{t,2})|=\pi>\frac{\gamma \pi}{2}$, and $|arg(\xi_{t,3})|=\pi>\frac{\gamma \pi}{2}$. Thus, the system (\ref{Frac eq}) shows local stability around the equilibrium point $E_t$ if the criteria (\romannum{1}) $t_{33}<0$, (\romannum{2}) $t_{11}^2 + 4 t_{12} t_{21}<0$, and (\romannum{3}) $t_{11}<0$ hold. Hence, the theorem.\\

\textbf{Case 2:} When $\alpha=1$: The Routh-Hurwitz criterion is employed to assess the local stability of $E_{t}(A,B,0)$. The Jacobian matrix of the system (\ref{Final ode eq}) around the top predator free equilibrium point $E_{t}(A,B,0)$ is given by

$$J_{to}=\begin{bmatrix}
t_{11} & t_{12} & 0 \\
t_{21} & 0 & t_{23} \\
0 & 0 & t_{33}
\end{bmatrix}
$$ \\
here,
$t_{11}=\frac{(\beta r q - r_1 - \beta r_1) (2 \beta d_1 + r_2 r_3 - m_1 r_2 r_3 - \omega_{10})}{(2 \beta^2 d_1)}$, $t_{12}=-\frac{d_1}{r_3}$, \\
   $t_{21}=-\frac{\omega_{10}(r_1 ((-1 + m_1) r_2 r_3 + 
      \omega_{10}) + 
   \beta (-2 d_1 r_1 - (r q - r_1) ((-1 + m_1) r_2 r_3 + \omega_{10})))}{2 \beta^2 d_1 (-1 + m_1) r_2}$, \\
   $t_{23}=\frac{(-1 + m_2) (-r_1 ((-1 + m_1) r_2 r_3 + \omega_{10}) + 
   \beta (2 d_1 r_1 + (r q - r_1) ((-1 + m_1) r_2 r_3 + \omega_{10}))) r_4}{2 \beta^2 d_1 (-1 + m_1) r_2 -b r_1 ((-1 + m_1) r_2 r_3 + 
    \omega_{10}) + \beta b (2 d_1 r_1 + (r q -r_1) ((-1 + m_1) r_2 r_3 + \omega_{10}))}$,\\
    $t_{33}=-d_2 + \frac{((-1 + m_2) (r_1 ((-1 + m_1) r_2 r_3 + \omega_{10}) + \beta (-2 d_1 r_1 - (r q - r_1) ((-1 + m_1) r_2 r_3 +\omega_{10}))) r_4 r_5)}{2 \beta^2 d_1 (-1 + m_1) r_2 - b r_1 ((-1 + m_1) r_2 r_3 + 
     \omega_{10}) + \beta b (2 d_1 r_1 + (r q - r_1) ((-1 + m_1) r_2 r_3 +\omega_{10}))}$, and \\
     $\omega_{10}=\sqrt{(-1 + 
    m_1) r_2 r_3 (-4 \beta d_1 + (-1 + m_1) r_2 r_3)}$.\\
    
Taking  $\sigma^3 + M_1 \sigma^2 + M_2\sigma + M_3 = 0$ as the characteristic equation of $J_{to}$, then $M_1=-(t_{11} + t_{33})$, $M_2=-(t_{12} t_{21} - t_{11} t_{33})$, and $M_3=t_{12} t_{21} t_{33}$. Now, $E_{t}(A,B,0)$ is locally asymptotically stable iff $M_{1} > 0$, $M_{2} > 0$, $M_{3} > 0$, and $M_1 M_2-M_3>0$. Hence, proved.

\end{proof}

\begin{theorem}\label{lscofrac}
If any of the criteria (\romannum{1}), (\romannum{2}), and (\romannum{3}) specified in the proof are satisfied, the system (\ref{Frac eq}) exhibits local stability near the coexisting equilibrium point $E_c(C,D,E)$ when $0<\alpha<1$.
\end{theorem}

\begin{proof}
In order to determine the local stability conditions of the system (\ref{Frac eq}) in the vicinity of the coexistence equilibrium point $E_c(C,D,E)$, it is necessary to calculate the Jacobian matrix of the system (\ref{Frac eq}) around the coexistence equilibrium point $E_{c}$. The Jacobian matrix of the system (\ref {Frac eq}) around the coexistence equilibrium point $E_{c}(C,D,E)$ is denoted as $J_{co}$, as previously provided in subsection (\ref{es}). Now, let us consider, $\varrho=\vartheta^3 + N_1 \vartheta^2 + N_2\vartheta + N_3$. The establishment of local asymptotic stability for the equilibrium point $E_{c}(C,D,E)$ in the system (\ref{Frac eq}) is contingent upon the satisfaction of any of the following conditions \cite{ahm}:\\
(\romannum{1}) If $\Delta (\varrho)>0$, $N_1>0$, $N_3>0$, and $N_1N_2-N_3>0$. \\
 (\romannum{2}) If $\Delta (\varrho)<0$, then $N_1\ge 0$, $N_2 \ge 0$, $N_3 > 0$, and $\gamma<\frac{2}{3}$.\\
 (\romannum{3}) If $\Delta (\varrho)<0$, $N_1>0$, $N_2>0$, $N_1 N_2 =N_3$, and $\gamma \in (0,1)$.\\
Here, $\Delta (\varrho)=18N_1N_2N_3+(N_1N_2)^{2}-4(N_1)^{2}N_3-4(N_2)^{2}-27(N_3)^{2}$.\\
In a subsequent section, the numerical verification of the local stability of $E_{c}(C,D,E)$ will be conducted.
\end{proof}
\begin{theorem}\label{lsco}
      The system (\ref{Frac eq}) shows local stability around the coexisting equilibrium point $E_{c}(C,D,E)$ when $\alpha=1$ if and only if $N_1$, $N_2$, and $N_3$ are all positive, $N_1 N_2-N_3$ is also positive. The symbols $N_i$ are defined directly in the proof.

\end{theorem}
\begin{proof}
    The Jacobian matrix of the system (\ref{Final ode eq}) evaluated at the coexisting equilibrium point $E_{c}(C,D,E)$ is provided as follows:
$$J_{co}=\begin{bmatrix}
u_{11} & u_{12} & 0 \\
u_{21} & u_{22} & u_{23} \\
0 & u_{32} & 0
\end{bmatrix}
$$ 
here, $u_{11}=\frac{b d_2 (q r - r_1) + d_2 (-1 + m_1) r_2 + (-1 + m_2) (q r - r_1) r_4 r_5}{b d_2 + (-1 + m_2) r_4 r_5}$, \\
$u_{12}=\frac{(-1 + m_1) (\beta q r - r_1 - 
   \beta r_1) r_2 (b d_2 + (-1 + m_2) r_4 r_5) (b d_2 (q r - r_1) + 
   d_2 (-1 + m_1) r_2 + (-1 + m_2) (q r - r_1) r_4 r_5)}{(b d_2 r_1 + \beta d_2 (-1 + m_1) r_2 + (-1 + m_2) r_1 r_4 r_5)^2}$,\\
   $u_{21}=-\frac{d_2 (-1 + m_1) r_2 r_3 (-r_1 (b d_2 + (-1 + m_2) r_4 r_5) + 
   \beta (2 b d_2 (q r - r_1) + d_2 (-1 + m_1) r_2 + 
      2 (-1 + m_2) (q r - r_1) r_4 r_5))}{(b d_2 + (-1 + m_2) r_4 r_5) (b d_2 r_1 + 
   \beta d_2 (-1 + m_1) r_2 + (-1 + m_2) r_1 r_4 r_5)}$,\\
   $u_{22}=\frac{b d_2 (\beta^2 d_1 d_2^2 (-1 + m_1)^2 r_2^2 + 
   b^2 d_2^2 \omega_{11} - (-1 + 
      m_2) r_1 r_4 r_5 \omega_{12} + 
   \beta (-1 + m_1) (-1 + 
      m_2) r_2 r_4 r_5 \omega_{13} + 
   b d_2 \omega_{14})}{((-1 + m_2) r_4 r_5 (b d_2 r_1 + 
    \beta d_2 (-1 + m_1) r_2 + (-1 + m_2) r_1 r_4 r_5)^2)}$, \\
    $u_{23}=-\frac{d_2}{r_5}$, $u_{32}=\frac{(1 - m_2) (b d_2 + (-1 + m_2) r_4 r_5) (\beta^2 d_1 d_2^2 (-1 + m_1)^2 r_2^2 + 
   b^2 d_2^2 \omega_{11} - (-1 + 
      m_2) r_1 r_4 r_5 \omega_{12} + 
   \beta (-1 + m_1) (-1 + 
      m_2) r_2 r_4 r_5 \omega_{13} + 
   b d_2 \omega_{14})}{(-1 + m_2)^2 r_4 (b d_2 r_1 + \beta d_2 (-1 + m_1) r_2 + (-1 + m_2) r_1 r_4 r_5)^2}$, and $\omega_{11}=(d_1 r_1^2 + (-1 + m_1) (q r - r_1) (\beta q r - r_1 - 
         \beta r_1) r_2 r_3)$, $\omega_{12}=(d_2 (-1 + m_1)^2 r_2^2 r_3 - (-1 + 
         m_2) (d_1 r_1 - (-1 + m_1) (q r - r_1) r_2 r_3) r_4 r_5)$, $\omega_{13}=(2 d_1 d_2 r_1 + (q r - 
         r_1) r_3 (d_2 (-1 + m_1) r_2 + (-1 + m_2) (q r - r_1) r_4 r_5))$, $\omega_{14}=(r_1 (-d_2 (-1 + m_1)^2 r_2^2 r_3 + 
         2 (-1 + m_2) (d_1 r_1 - (-1 + m_1) (q r - r_1) r_2 r_3) r_4 r_5) + 
      \beta (-1 + 
         m_1) r_2 (2 d_1 d_2 r_1 + (q r - r_1) r_3 (d_2 (-1 + m_1) r_2 + 
            2 (-1 + m_2) (q r - r_1) r_4 r_5)))$.\\
            
            Now, let us consider $\vartheta^3 + N_1 \vartheta^2 + N_2\vartheta + N_3 = 0$ be the characteristic equation of $J_{co}$. Then, $N_1=-(u_{11} + u_{22})$, $N_2=-u_{12} u_{21} + u_{11} u_{22} - u_{23} u_{32}$, and $N_3=u_{11} u_{23} u_{32}$. Thus, the coexisting equilibrium point $E_{c}(C,D,E)$ is locally stable if and only if $N_{i}>0$ and $N_1 N_2-N_3>0$ for i=1,2,3, according to the Routh-Hurwitz criterion. Therefore, the proof.
\end{proof}

\section{Bifurcations}\label{hb} Bifurcation in ecology denotes a substantial modification in the configuration or conduct of an ecological system resulting from variations in factors. The bifurcation point refers to the precise instant when a system undergoes a change from one stable state to another, possibly leading to different population densities. Bifurcations can have substantial consequences for the management and conservation of ecosystems. 
\subsection{Transcritical bifurcation} Transcritical bifurcation is a phenomenon in dynamical systems when the stability of two equilibrium points, one stable and one unstable, switch as a parameter is changed. This phenomena is of great significance when it comes to comprehending diverse ecological and biological systems, as it can elucidate changes in population dynamics, species interactions, and ecosystem states. Theorems pertaining to this can be found below.

\begin{theorem}\label{odetrnsthm}
The system (\ref{Final ode eq}) undergoes a transcritical bifurcation near the equilibrium point $E_a$ at the critical value $r_1$=$q r=r_1^{tbp}$.
\end{theorem}
\begin{proof}
The Jacobian matrix of the system (\ref{Final ode eq}) around $E_a$ is $J_{ax}$ as mentioned in the previous section. Using $r_1=q r=r_1^{tbp}$, we get

$$(J_{ax})_{r_1=r_1^{tbp}}=\begin{bmatrix}
0 & 0 & 0 \\
0 & -d_1 & 0 \\
0 & 0 & -d_2
\end{bmatrix}
$$ 
Now, we assume that the zero eigenvalues of $(J_{ax})_{r_1=r_1^{tbp}}$ and $(J_{ax})^{t}_{r_1=r_1^{tbp}}$ correspond to two eigenvectors, $U_1$ and $U_2$, respectively. After some computation, we get
$U_1=(u_{11},u_{12},u_{13})^t=(1, 0, 0)^t$ and $U_2=(u_{21},u_{22},u_{23})^t=(1,0,0)^t$. Now, the theorem put forward by Sotomayor \cite{perko} is used to demonstrate the occurrence of a transcritical bifurcation at $r_1=q r=r_1^{tbp}$ in the vicinity of $E_a$. Outlined below are the prerequisites for transcritical bifurcation, as stated in Sotomayor's theorem \cite{perko}.

$Z_{r_1}(E_a;r_1^{tbp})=\begin{bmatrix}
0\\
0 \\
0 
\end{bmatrix}
,  D(Z_{r_1}(E_1;r_1^{tbp}))U_1=\begin{bmatrix}
q r & 0 & 0\\
0 & 0 &0\\
0 & 0 & 0
\end{bmatrix}
\begin{bmatrix}
1\\
0\\
0
\end{bmatrix}=\begin{bmatrix}
q r \\
0 \\
0
\end{bmatrix}, $ and\\

$ D^2(Z_{r_1}(E_1;r_1^{tbp}))(U_1,U_1)=\begin{bmatrix}
-2 r_1 + 2 (-1 + m_1) r_2\\
-2 (-1 + m_1) r_2 r_3\\
0
\end{bmatrix}
$ .\\
Therefore,
\begin{equation*}
    \begin{split}
        U_2^T(Z_{r_1}(E_1;r_1^{tbp}))&=0\\
        U_2^T( D(Z_{r_1}(E_1;r_1^{tbp}))U_1)&=q r \ne 0,\\
        U_2^T(D^2(Z_{r_1}(E_1;r_1^{tbp}))(U_1,U_1))&=-2 r_1 + 2 (-1 + m_1) r_2 \ne 0
    \end{split}
\end{equation*}
Thus, the application of Sotomayor's theorem \cite{perko} proves the presence of a transcritical bifurcation around $E_a$ at $r_1$=$q r=r_1^{tbp}$. In addition, there are additional parameters that can be utilised as bifurcation parameters.

\end{proof}

\subsection{Hopf bifurcation} The Hopf bifurcation is a well-known and extensively studied phenomenon in the field of dynamical systems. It has been observed in various domains, including ecological models, where it plays a crucial role in understanding the qualitative changes in system behaviour as a parameter is systematically varied. The comprehension of Hopf bifurcation in the field of ecology holds significant importance for ecologists, as it enables them to comprehend the inherent capacity for intricate, cyclical patterns in the dynamics of populations. The following theorem establishes the conditions for the existence of Hopf bifurcation in the system (\ref{Final ode eq}) around the interior equilibrium point $E_{c}$.


\begin{theorem}\label{odehopfthm}
    The necessary and sufficient criteria for the system (\ref{Final ode eq}) to experience a Hopf bifurcation at $ m_1 = m_{1}^{hb}$ are as follows:\\
    (\romannum{1}) $N_{i}(m_{1}^{hb})>0$, i=1,2,3\\
    (\romannum{2}) $N_1(m_{1}^{hb})N_2(m_{1}^{hb})=N_3(m_{1}^{hb})$, \\
    (\romannum{3})$N_1(m_{1}^{hb})$ $N_2^{'}$($m_{1}^{hb}$) +$N_2(m_{1}^{hb})$ $N_1^{'}$($m_{1}^{hb}$) -$N_3^{'}$($m_{1}^{hb}$)   $ \neq 0$, j=1,2,3,\\
here, $N_1$, $N_2$, and $N_3$ are already defined in the theorem (\ref{lsco}). $\delta_1$, $\delta_2$, and $\delta_3$ are the roots of the characteristic equation of the Jacobian matrix $J_{co}$.
\begin{proof}
    Please refer to theorem 5 in \cite{bhat}.
\end{proof}

\begin{figure}[H]
     \centering
     \begin{subfigure}{0.45\textwidth}
         \centering
         \includegraphics[width=\textwidth]{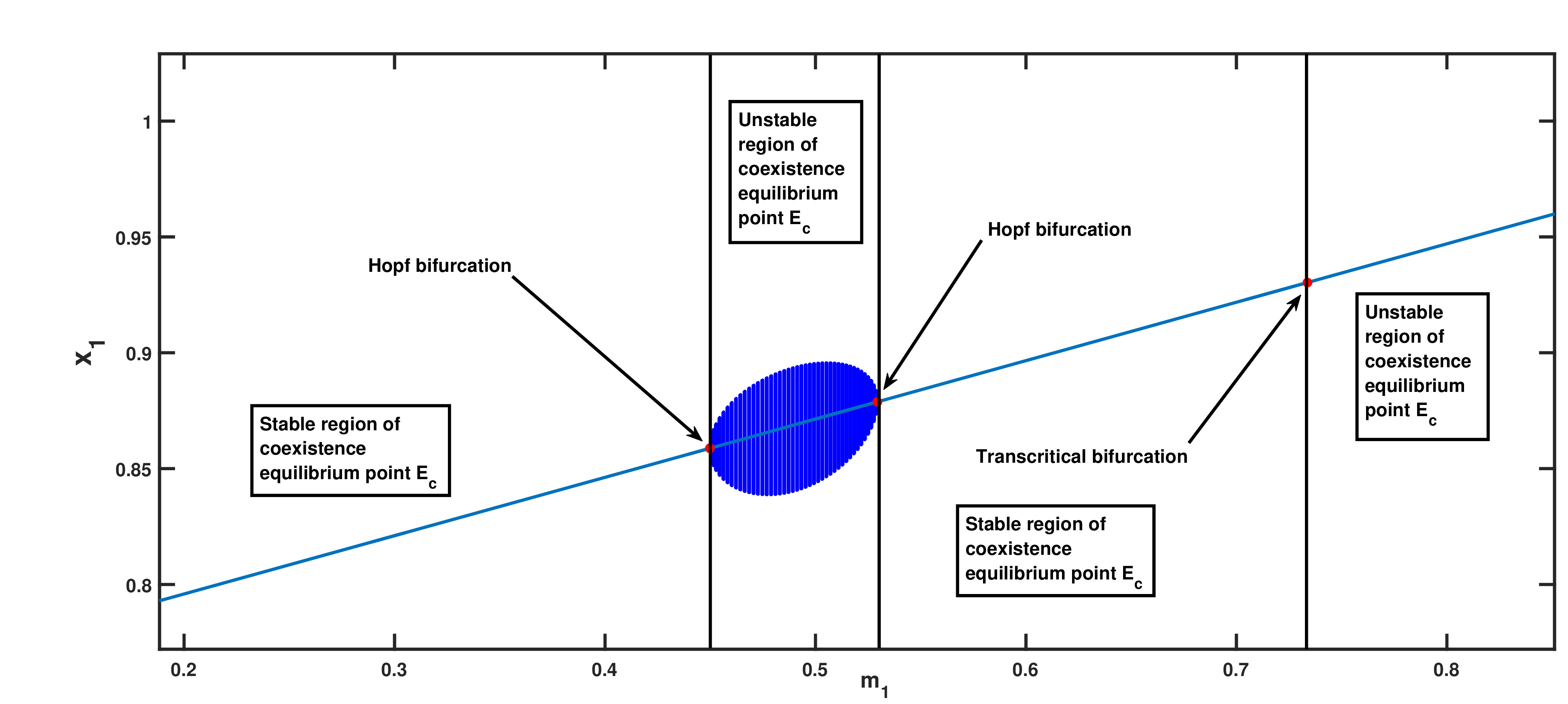}
         \caption{\emph{Two hopf bifurcations and one transcritical bifurcation occur when the parameter $m_1$ changes.}}
         \label{m1equilibrium curve}
     \end{subfigure}
      \hfill
     \begin{subfigure}{0.45\textwidth}
         \centering
         \includegraphics[width=\textwidth]{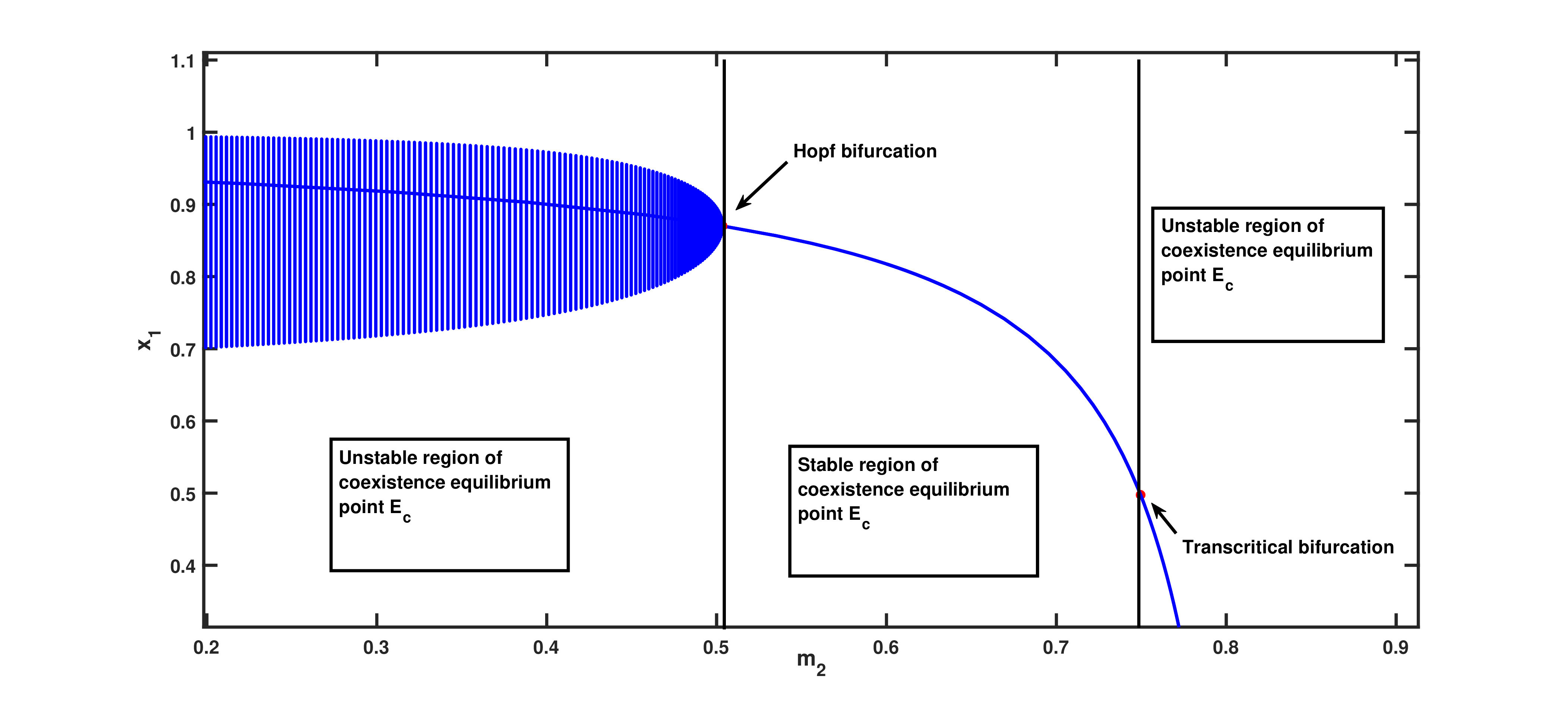}
         \caption{\emph{As the parameter $m_2$ is modified, a transcritical bifurcation and a Hopf bifurcation are observed.}}
         \label{m2equilibrium curve}
     \end{subfigure}
     \hfill
     \begin{subfigure}{0.45\textwidth}
         \centering
         \includegraphics[width=\textwidth]{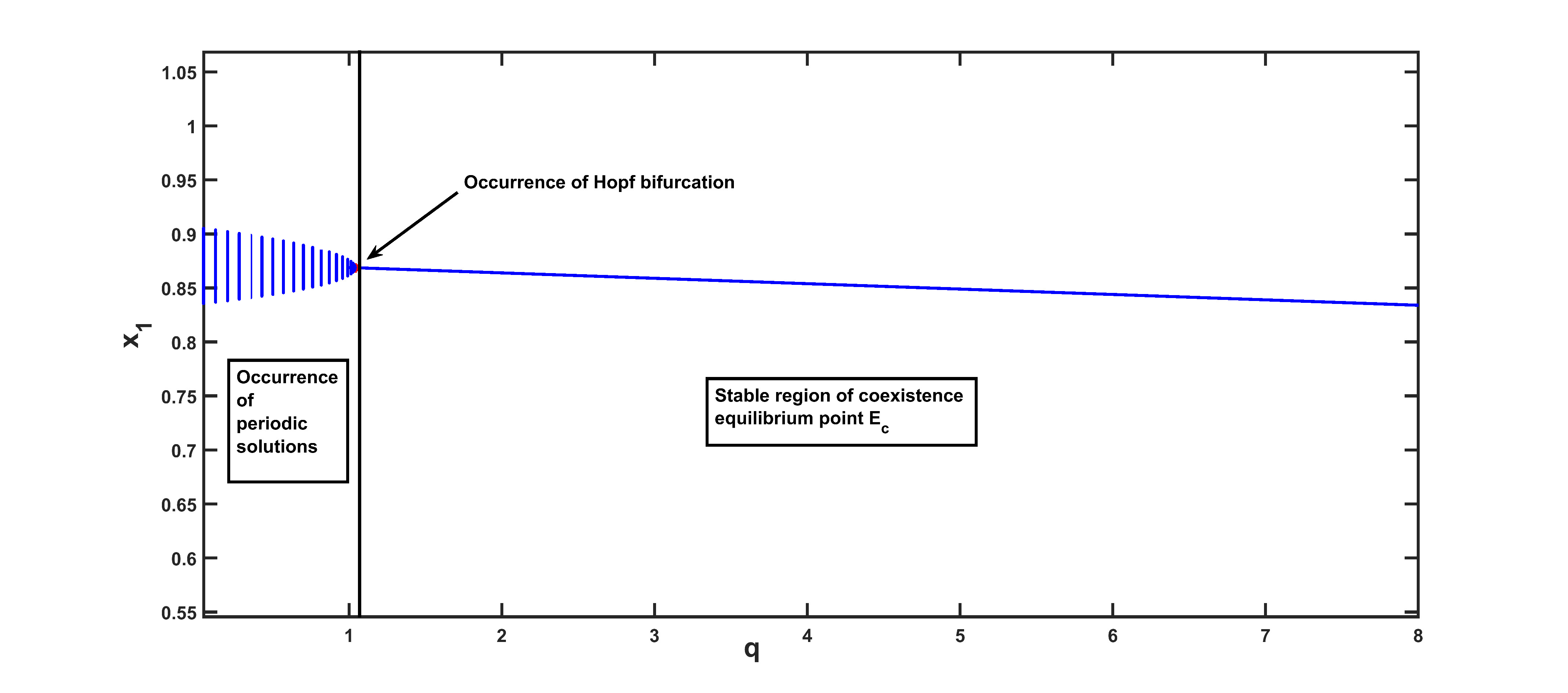}
         \caption{\emph{fluctuations in the population of the prey species are detected as a result of a Hopf bifurcation when the parameter $q$ is altered.}}
         \label{qequilibrium curve}
     \end{subfigure}
     \hfill
     \begin{subfigure}{0.45\textwidth}
         \centering
         \includegraphics[width=\textwidth]{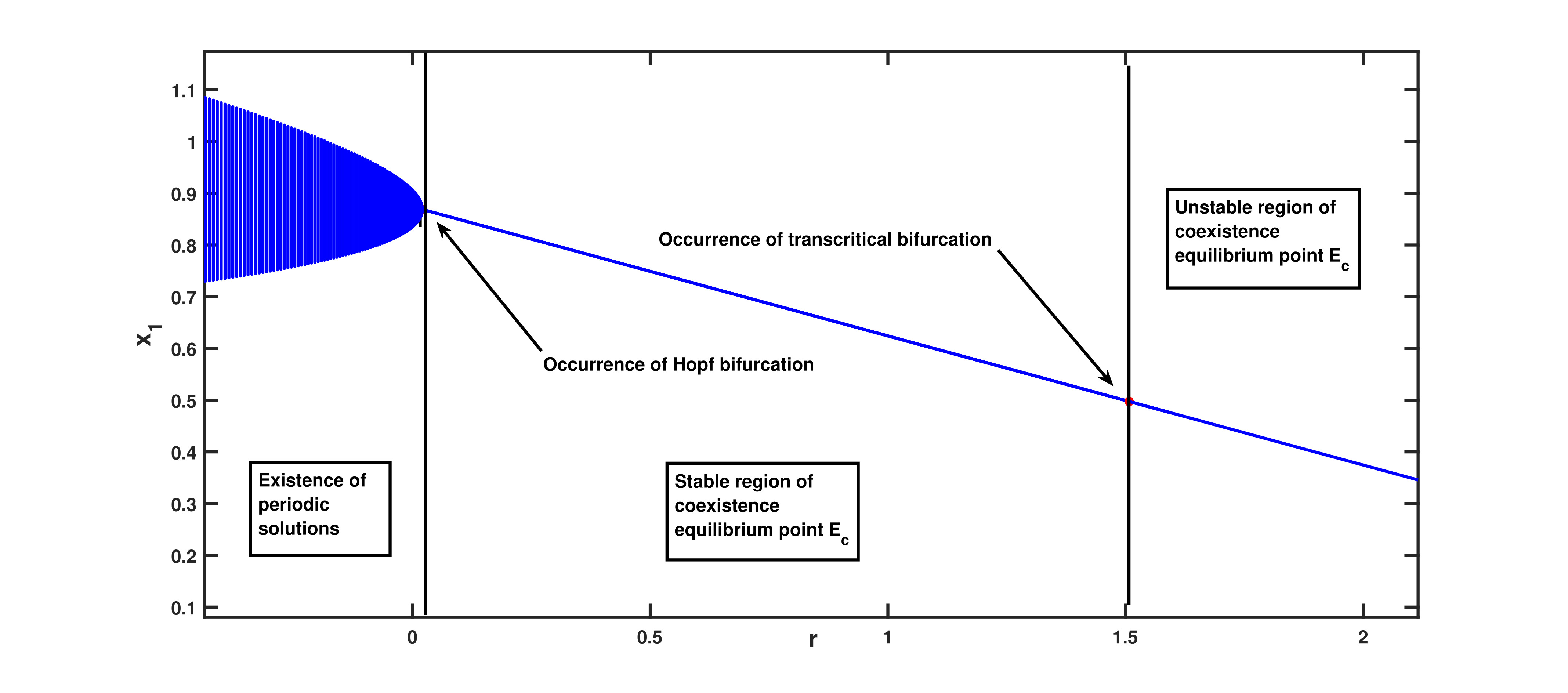}
         \caption{\emph{The variation of the parameter $r$ demonstrate the presence of a hopf bifurcation and a transcritical bifurcation in the system. }}
         \label{requilibrium curve}
     \end{subfigure}
     \hfill
     \begin{subfigure}{0.45\textwidth}
         \centering
         \includegraphics[width=\textwidth]{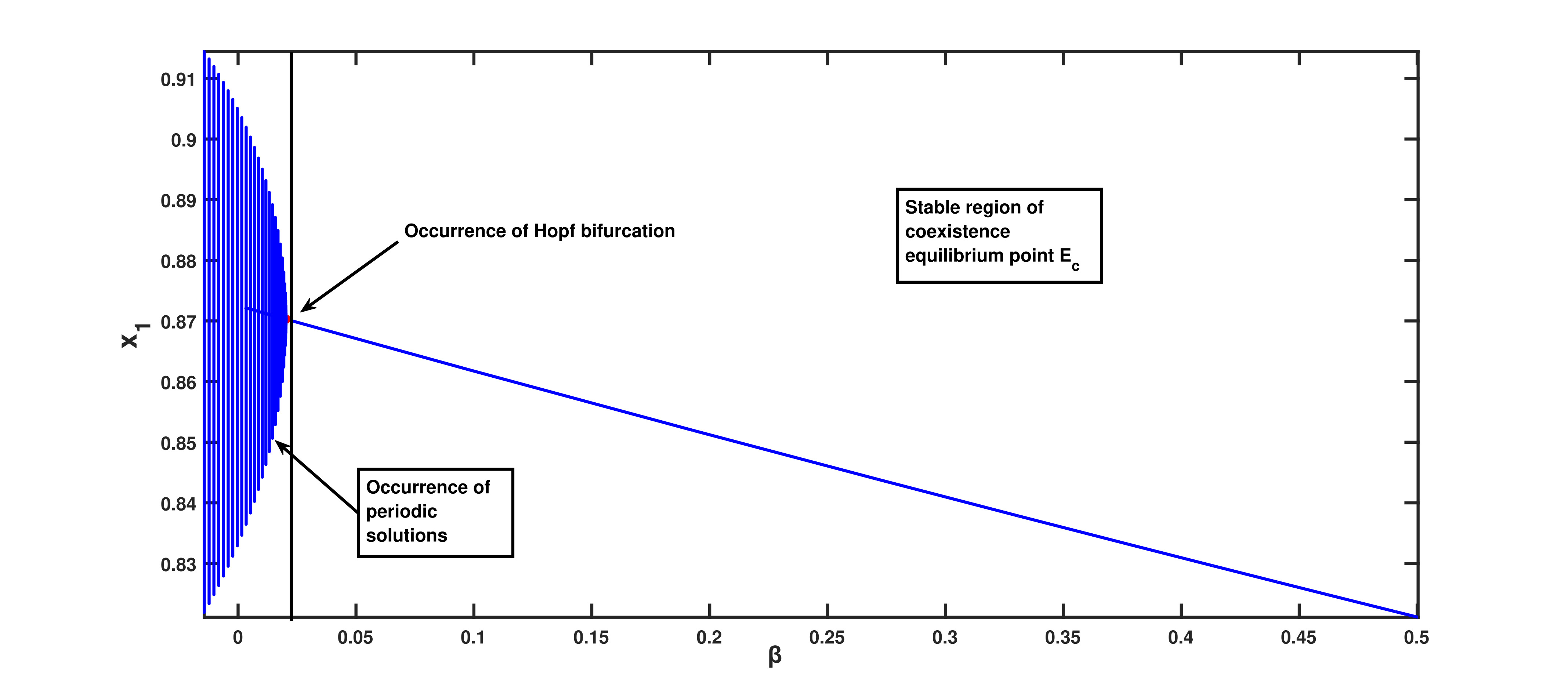}
         \caption{\emph{Occurrence of a hopf bifurcation is observed when the parameter $\beta$ is altered.}}
         \label{betaequilibrium curve}
     \end{subfigure}
       \caption{\emph{This diagram illustrates the existence of various bifurcations in relation to different parameters in the system (\ref{Final ode eq}). The set of parameter values which are used to generate these figures are as follows: $r_1 = 2$, $r_5 = 1$, $\beta = 0.01$, $m_1 = 0.5$, $m_2 = 0.5$ $d_1 = 0.25$, $r_2 = 1$, $d_2 = 0.5$, $r_4 = 3$, $b = 1$, $r_3 = 1$, $q = 0.5$, $r = 0.01$, and $\alpha=1$. }}
        \label{equicurve}
\end{figure}

\end{theorem}

\section{Numerical simulation}\label{ns} In the previous sections, we already established some analytical conclusions. In this section, we conduct a numerical investigation into the dynamics of the prey-predator systems (\ref{Final ode eq}) and (\ref{Frac eq}) and substantiate our analytical conclusions by utilizing some hypothetical parameter values. The analytical results obtained in previous sections are further supported by the utilisation of numerous figures generated by the use of the Mathematica and MATLAB software packages.

\begin{figure}[H]
     \centering
     \begin{subfigure}{0.45\textwidth}
         \centering
         \includegraphics[width=\textwidth]{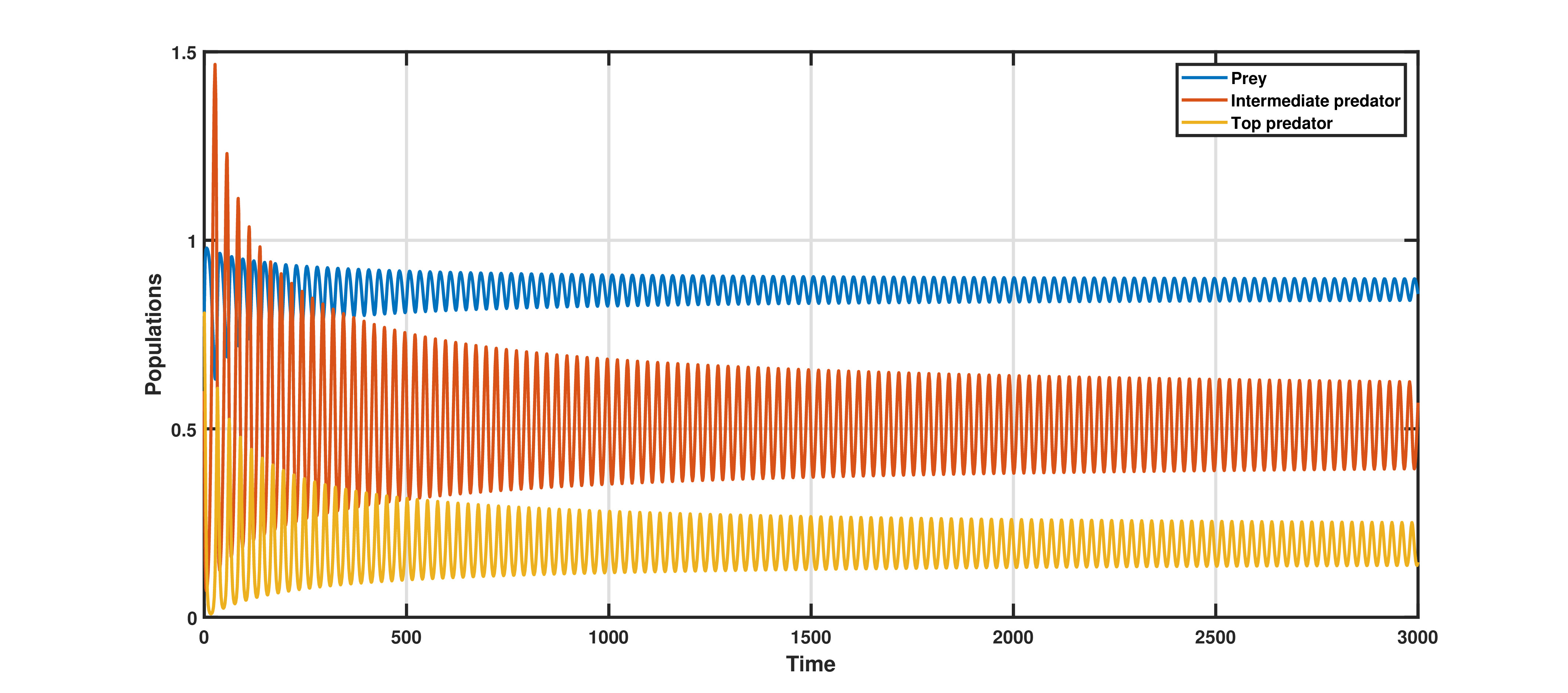}
         \caption{\emph{ Case (\romannum{1}): When $\alpha=1$}}
         \label{m10.5odets}
     \end{subfigure}
      \hfill
     \begin{subfigure}{0.45\textwidth}
         \centering
         \includegraphics[width=\textwidth]{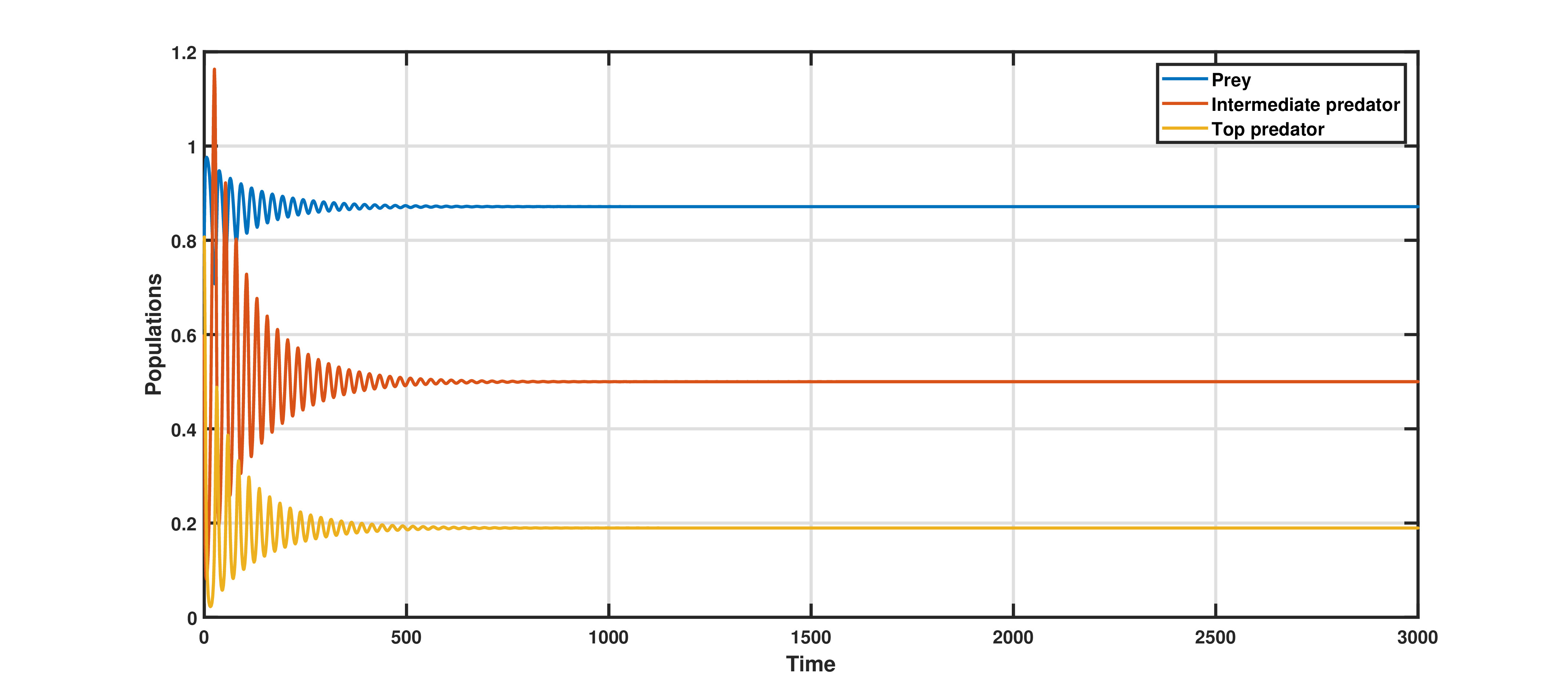}
         \caption{\emph{ Case (\romannum{2}): When $\alpha=0.98$}}
         \label{m10.5fde0.98ts}
     \end{subfigure}
     \hfill
     \begin{subfigure}{0.45\textwidth}
         \centering
         \includegraphics[width=\textwidth]{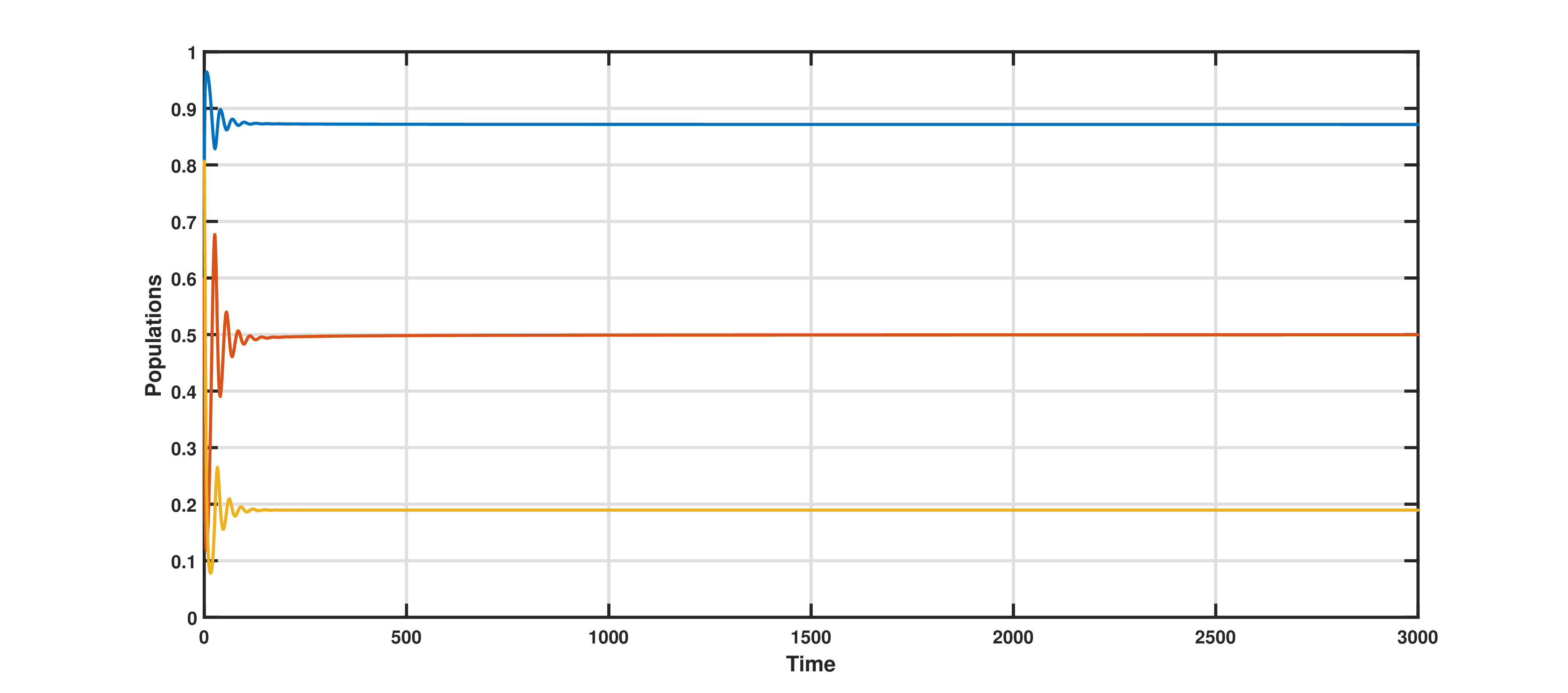}
         \caption{\emph{ Case (\romannum{3}): When $\alpha=0.90$}}
         \label{m10.5fde0.9ts}
     \end{subfigure}
     \hfill
     \begin{subfigure}{0.45 \textwidth}
         \centering
         \includegraphics[width=\textwidth]{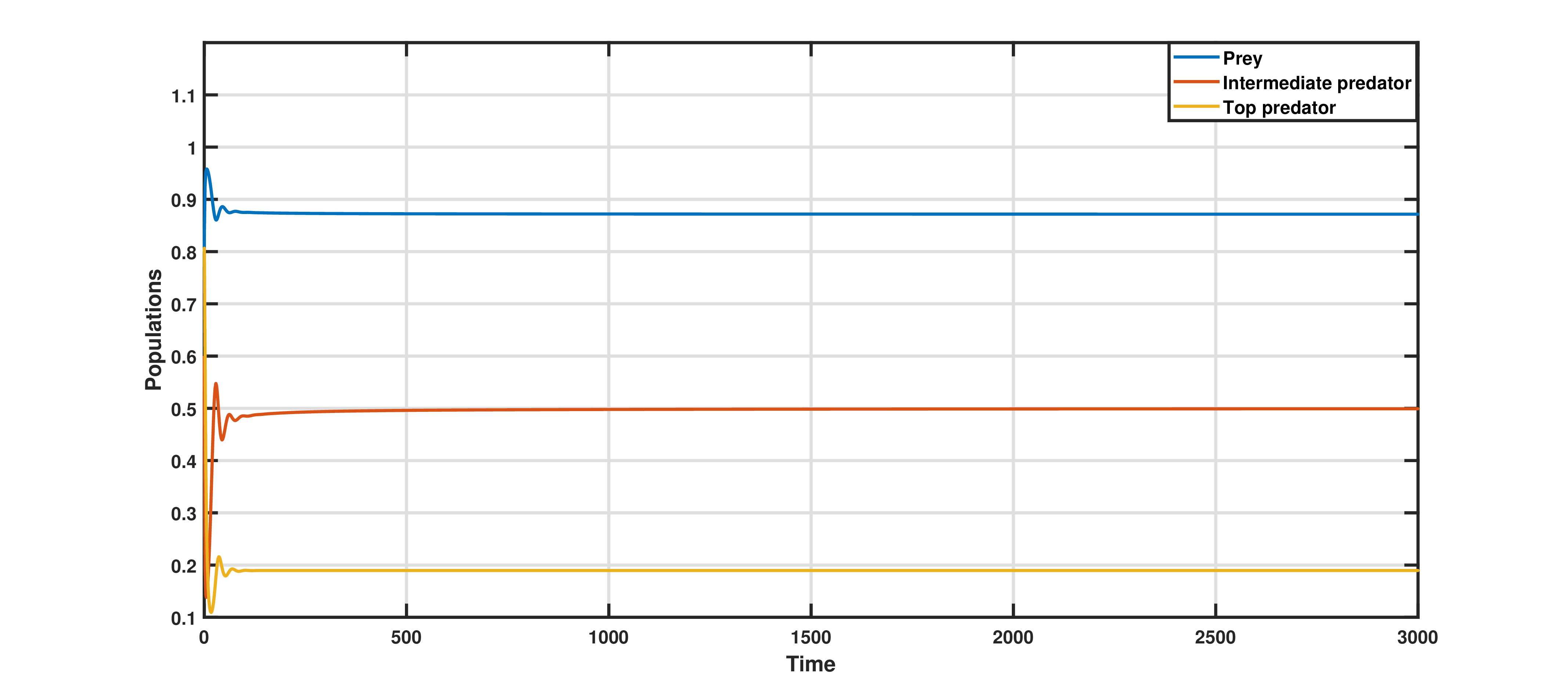}
         \caption{\emph{ Case (\romannum{4}): When $\alpha=0.85$}}
         \label{m10.5fde0.85ts}
     \end{subfigure}
       \caption{\emph{The impact of memory on the system is elegantly demonstrated. The following parameter values were used to generate these figures: $r_1=2$, $r_4 =3$,  $\beta = 0.01$, $m_1= 0.5$, $r_2 = 1$, $m_2 = 0.5$, $r_3 =1$, $d_1 = 0.25$, $d_2 = 0.5$, $b = 1$, $q=0.5$, $r_5 = 1$, and $r = 0.01$.}}
        \label{m1effectofmemory}
\end{figure}

\subsection{Numerical verification of the local stability of the equilibrium points} In this section, we employ numerical methods to confirm the local stability of the equilibrium points of the integer order system (\ref{Final ode eq}) and the fractional order system (\ref{Frac eq}) discussed in the preceding sections.


In order to validate the theorem (\ref{lsva}) numerically, we will examine the parameter values: $r_1 = 0.46$, $r_2 = 0.32$, $r_3 = 0.5$, $r_4 = 0.002$, $r_5 = 0.38$, $\beta = 4.047$, $m_1= 0.72$, $m_2 = 0.17$, $d_1 = 0.096$, $d_2 = 0.279$, $b = 3.33$, $q=1.35$, and $r = 0.38$. Given the parameter values, $r > 0.34=\frac{r_1}{q}$, i.e., the condition outlined in theorem (\ref{lsva}) is satisfied, resulting in the local stability of $E_{v}(0,0,0)$ for $0<\alpha\le1$. Figures (\ref{lsvafig}) and (\ref{lsode}) demonstrate this for $\alpha=0.98$ and $\alpha=1$, respectively. 
From an ecological standpoint, it has been observed that when the rate at which a prey species reproduces naturally becomes lower than the rate at which it is harvested, the prey species is at risk of extinction. Consequently, both the predator species that depend on the prey also face extinction. This leads to the establishment of a stable equilibrium point where all three species vanish. In addition, in this scenario, the presence of the axial equilibrium point becomes unattainable, mirroring the natural occurrence where the intrinsic growth rate of the prey species falls below the harvesting rate, rendering survival impossible for the prey. Therefore, in order for the prey to survive, its growth rate must exceed its harvesting rate. Therefore, in the previously mentioned bio-systems (\ref{Final ode eq}) and (\ref{Frac eq}), it is possible for all three species to become extinct under certain conditions.


By reducing the value of the parameter $q$ i.e., considering the following parameter values: $r_1=0.46$, $r_2 = 0.32$, $r_3 = 0.5$, $r_4 = 0.002$, $r_5 = 0.38$, $\beta = 4.047$, $m_1= 0.72$, $m_2 = 0.17$, $d_1 = 0.096$, $d_2 = 0.279$, $b = 3.33$, $q=0.838$, and $r = 0.38$, we can establish the local stability of the axial equilibrium point $E_{a}$. The eigenvalues of the Jacobian matrix $J_{ax}$ of the system (\ref{Final ode eq}) around the axial equilibrium point $E_{a}$ are -0.279, -0.141341, and -0.0651173. Since all of the eigenvalues are negative, the local stability of the axial equilibrium point is established. Furthermore, it meets all the numerical requirements stated in theorem (\ref{lsax}). Figures (\ref{lsaxfig}) and (\ref{lsode}) provides visual evidence of the local stability of the axial equilibrium point $E_{a}$ for $\alpha=0.98$ and $\alpha=1$, respectively. In ecological terms, this indicates that under specific circumstances, it is also possible for both the top predator and the intermediate predator species to face extinction while prey species continue to flourish.

On further reduction of the parameter $q$ to $q=0.019$ and keeping the other parameter values the same as previously mentioned, all the necessary conditions detailed in theorem (\ref{lsto}) are fulfilled.
Based on the theorem (\ref{lsto}), it can be inferred that the top predator free equilibrium point $E_{t}$ is locally stable. The precise depiction can be found in figures (\ref{lstofig}) and (\ref{lsode}) for $\alpha=0.98$ and $\alpha=1$, respectively. From an ecological point of view, this indicates that solely the apex predator that confronts extinction, enabling both the intermediate predator and the prey to endure within the ecosystem, irrespective of the memory effect.

Now, in order to establish the local stability of the coexisting equilibrium point $E_{c}$, we examine the parameter values: $r_1=2$, $r_3 =1$, $\beta = 0.01$, $m_1= 0.5$, $r_2 = 1$, $m_2 = 0.6$, $d_1 = 0.25$, $d_2 = 0.5$, $r_4 =3$, $b = 1$, $r_5 = 1$, $q=0.5$, $r = 0.01$, and $\alpha=1$. For these parameter values, all the conditions outlined in theorem (\ref{lsco}) gets fulfilled and hence the coexisting equilibrium point $E_{c}$ is locally stable for $\alpha=1$. In a similar way, for certain parameter values: $r_1=2$, $r_3 =1$, $\beta = 4.047$, $r_2 = 1$, $m_1= 0.9471$, $m_2 = 0.17$, $d_1 = 0.25$, $r_4 =3$, $d_2 = 0.5$, $b = 1$, $r_5 = 1$, $q=0.5$, $r = 0.01$ and $\alpha=0.98$, we calculate the following values: $\Delta (\varrho)=0.0024>0$, $N_1=1.98>0$, $N_3=0.00002>0$, and $N_1N_2-N_3=0.05>0$. Therefore, all the conditions specified by theorem (\ref{lscofrac}) are met. Thus, the local asymptotic stability of the fixed point $E_{c}(C,D,E)$ in the system (\ref{Frac eq}) has been verified.
Figures (\ref{lsfracfig}) and (\ref{lsode}) perfectly illustrates the local stability of all the ecologically feasible equilibrium points, including the coexisting equilibrium point $E_{c}$ for $0<\alpha=0.98<1$ and $\alpha=1$, respectively. This implies that, given some parametric settings, the coexistence of all three species within the system is conceivable with or without memory effect.
\subsection{Effect of memory on the population dynamics (Impact of the parameter $\alpha$)} Memory in predator-prey systems pertains to the capacity of organisms to recollect previous encounters and adapt their behaviour or strategy in accordance with those recollections. In the past, ecological models have primarily concentrated on immediate interactions wherein predators react to the present availability of prey and vice versa. However, the inclusion of memory in predator-prey models offers an additional level of richness and realism. Following that, we have constructed the model (\ref{Frac eq}) to investigate the impact of memory on the bio-system (\ref{Final ode eq}). The fractional order $\alpha$ in the model (\ref{Frac eq}) represents the degree of memory that is influencing the system under investigation. As the fractional order $\alpha \to 0$, the system demonstrates a greater level of memory. On the other hand, as the value of $\alpha \to 1$, the system becomes increasingly devoid of memory \cite{mat,saeed}. In order to have a deeper comprehension of the impact of memory on the system (\ref{Frac eq}), we examine several scenarios by adjusting different parameter values. Figure (\ref{m1effectofmemory}) displays four distinct instances. In the first case, we discuss the situation in which the system (\ref{Frac eq}) has a value of $\alpha=1$ , which is equivalent to the system (\ref{Final ode eq}). The second instance involves considering the system (\ref{Frac eq}) with the value of $\alpha=0.98$. In the third scenario, the system (\ref{Frac eq}) is analysed with a value of $\alpha=0.9$, whereas in the fourth scenario, the system (\ref{Frac eq}) is studied with a value of $\alpha=0.85$. Figure (\ref{m1effectofmemory}) clearly demonstrates that when the order of the fractional derivative decreases, i.e., $\alpha \to 0$, the stability of the system under discussion near the coexisting equilibrium point $E_{c}$ improves.

\begin{figure}[H]
     \centering
     \begin{subfigure}{0.45\textwidth}
         \centering
         \includegraphics[width=\textwidth]{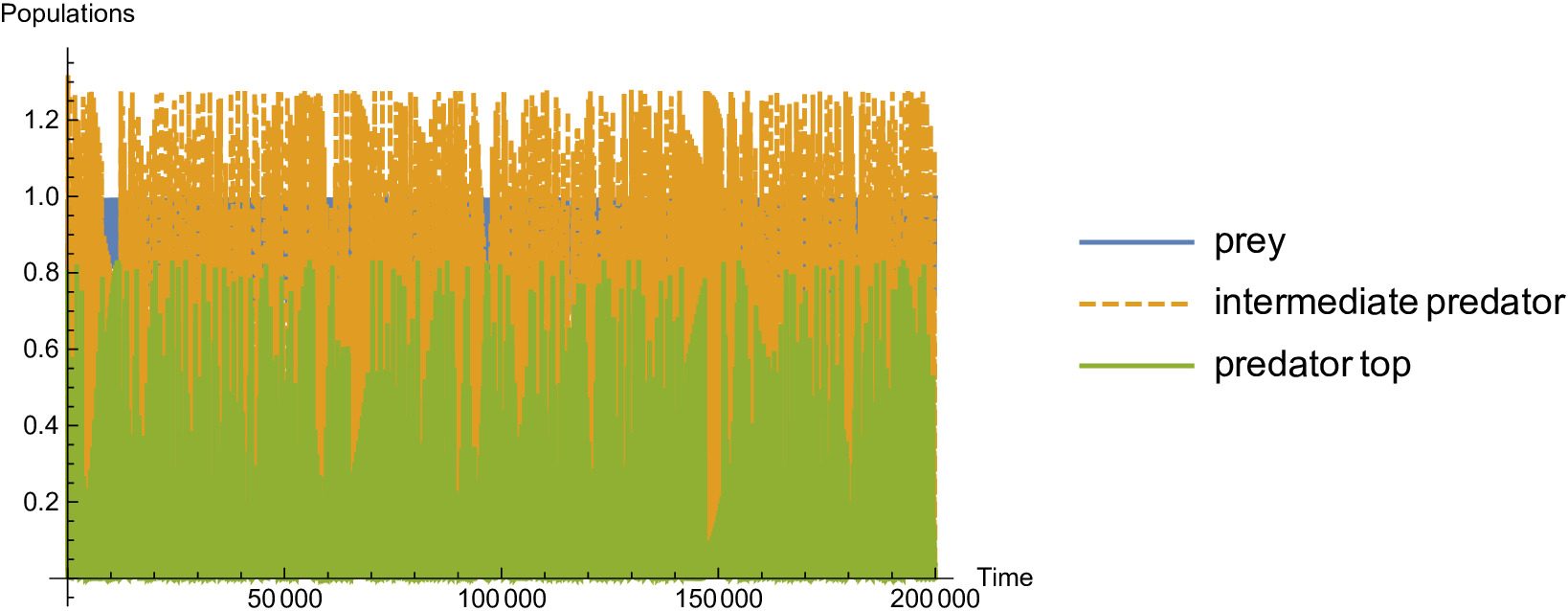}
         \caption{\emph{ Case (\romannum{1}): When $\alpha=1$}}
         \label{m20ode}
     \end{subfigure}
      \hfill
     \begin{subfigure}{0.45\textwidth}
         \centering
         \includegraphics[width=\textwidth]{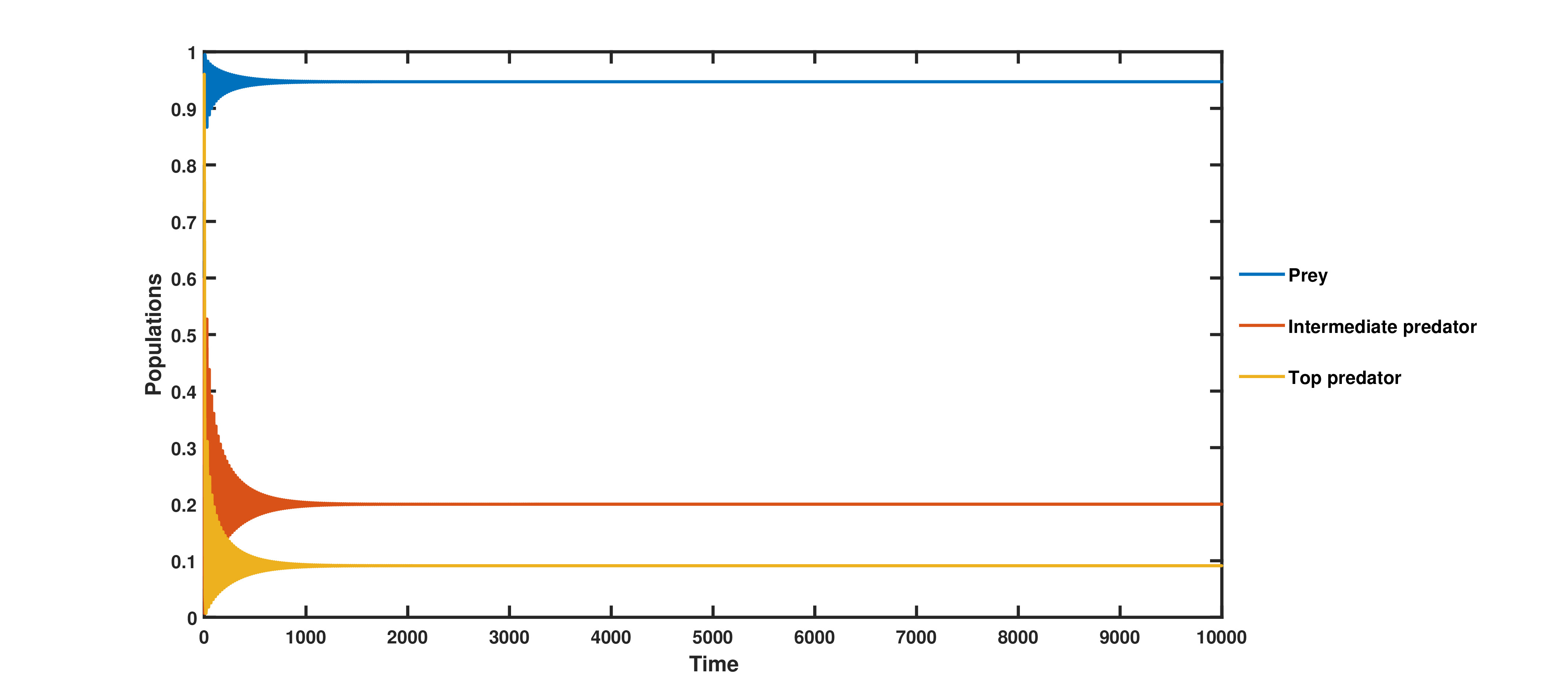}
         \caption{\emph{ Case (\romannum{2}): When $\alpha=0.98$}}
         \label{m20fde0.98}
     \end{subfigure}
     \hfill
     \begin{subfigure}{0.45\textwidth}
         \centering
         \includegraphics[width=\textwidth]{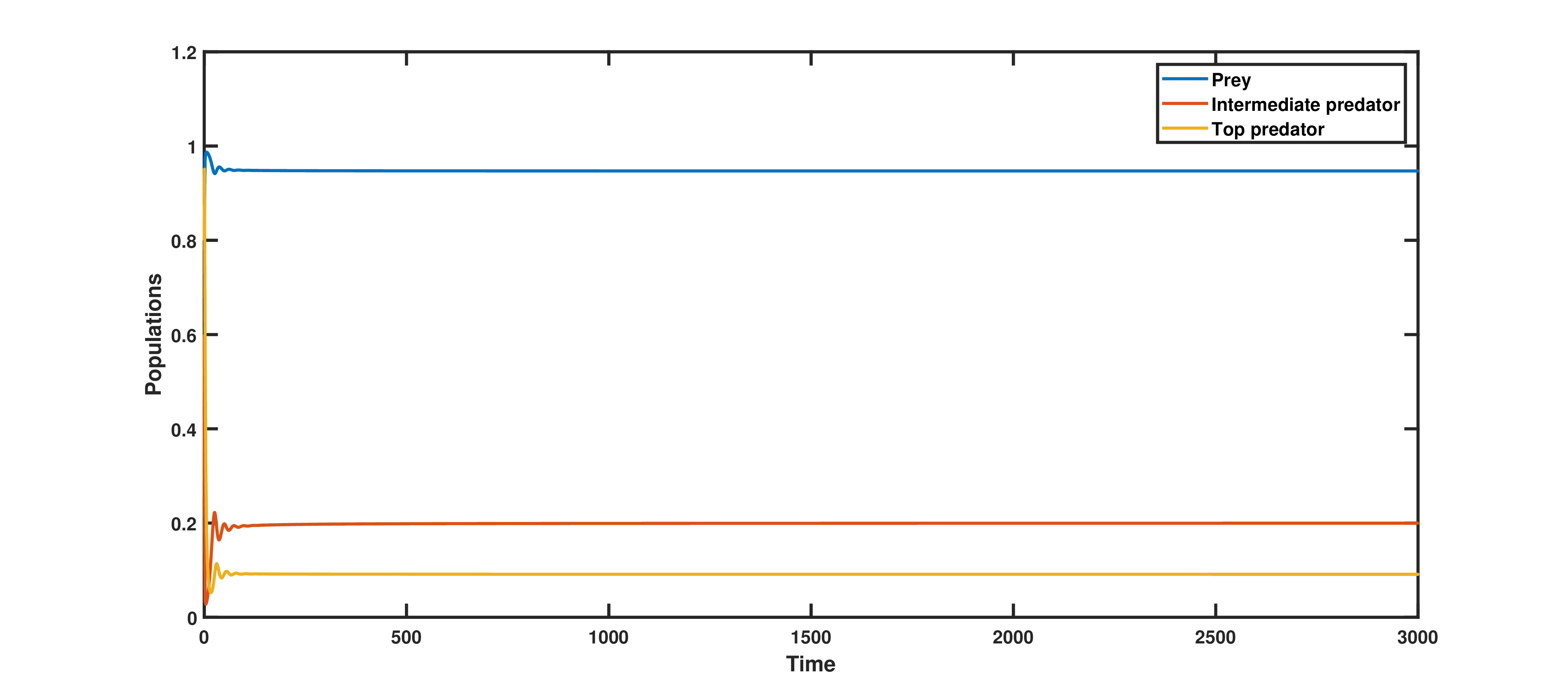}
         \caption{\emph{ Case (\romannum{3}): When $\alpha=0.90$}}
         \label{m20fde0.9}
     \end{subfigure}
     \hfill
     \begin{subfigure}{0.45 \textwidth}
         \centering
         \includegraphics[width=\textwidth]{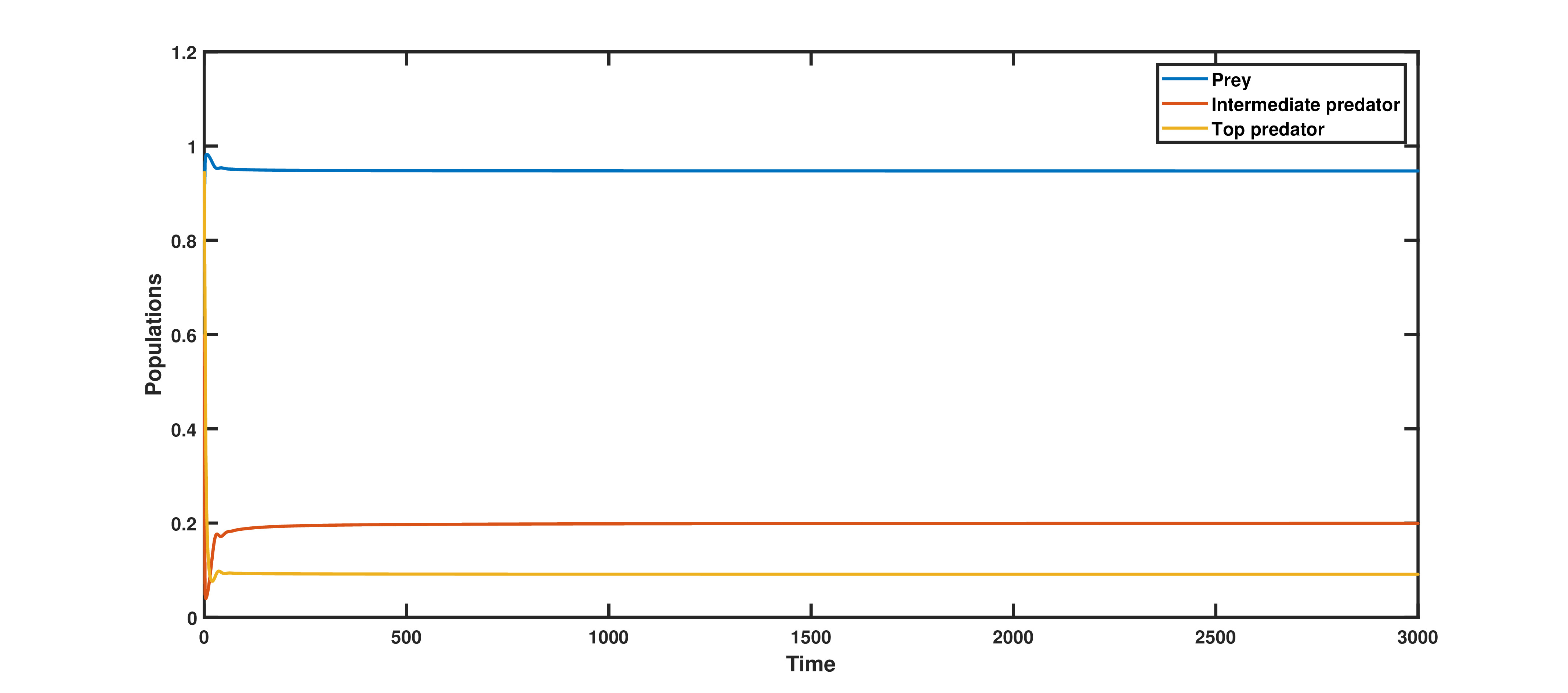}
         \caption{\emph{ Case (\romannum{4}): When $\alpha=0.85$}}
         \label{m20fde0.85}
     \end{subfigure}
       \caption{\emph{The impact of memory in the system is readily apparent in these figures. The parameters used here are: $r_1=2$, $r_5 = 1$, $\beta = 0.01$, $r_3 =1$, $m_1= 0.5$, $r_2 = 1$, $m_2 = 0$, $d_1 = 0.25$, $r_4 =3$, $d_2 = 0.5$, $b = 1$, $q=0.5$, and $r = 0.01$}}
        \label{m2=0effectofmemory}
\end{figure}

\begin{figure}[H]
     \centering
     \begin{subfigure}{0.45\textwidth}
         \centering
         \includegraphics[width=\textwidth]{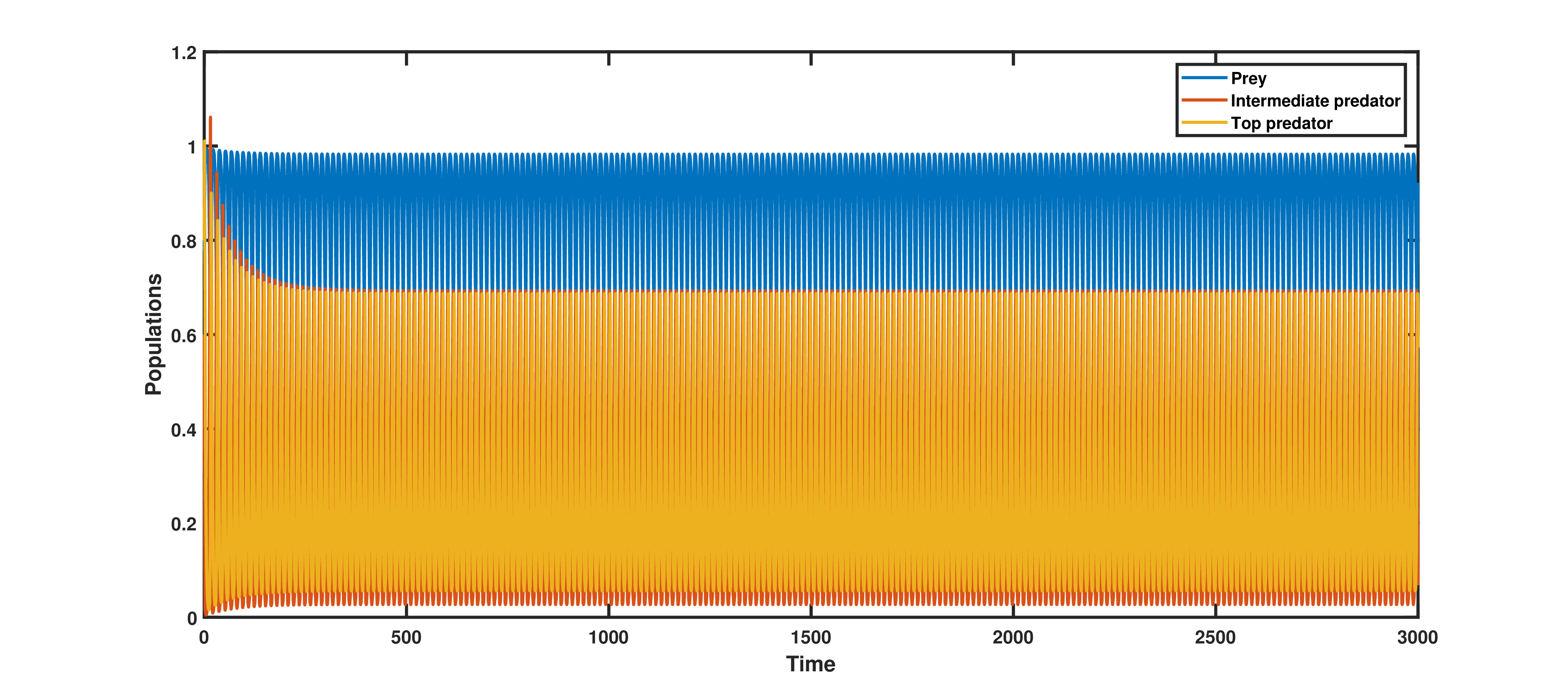}
         \caption{\emph{ Case (\romannum{1}): When $\alpha=1$}}
         \label{m1=0=m2ODE}
     \end{subfigure}
      \hfill
     \begin{subfigure}{0.45\textwidth}
         \centering
         \includegraphics[width=\textwidth]{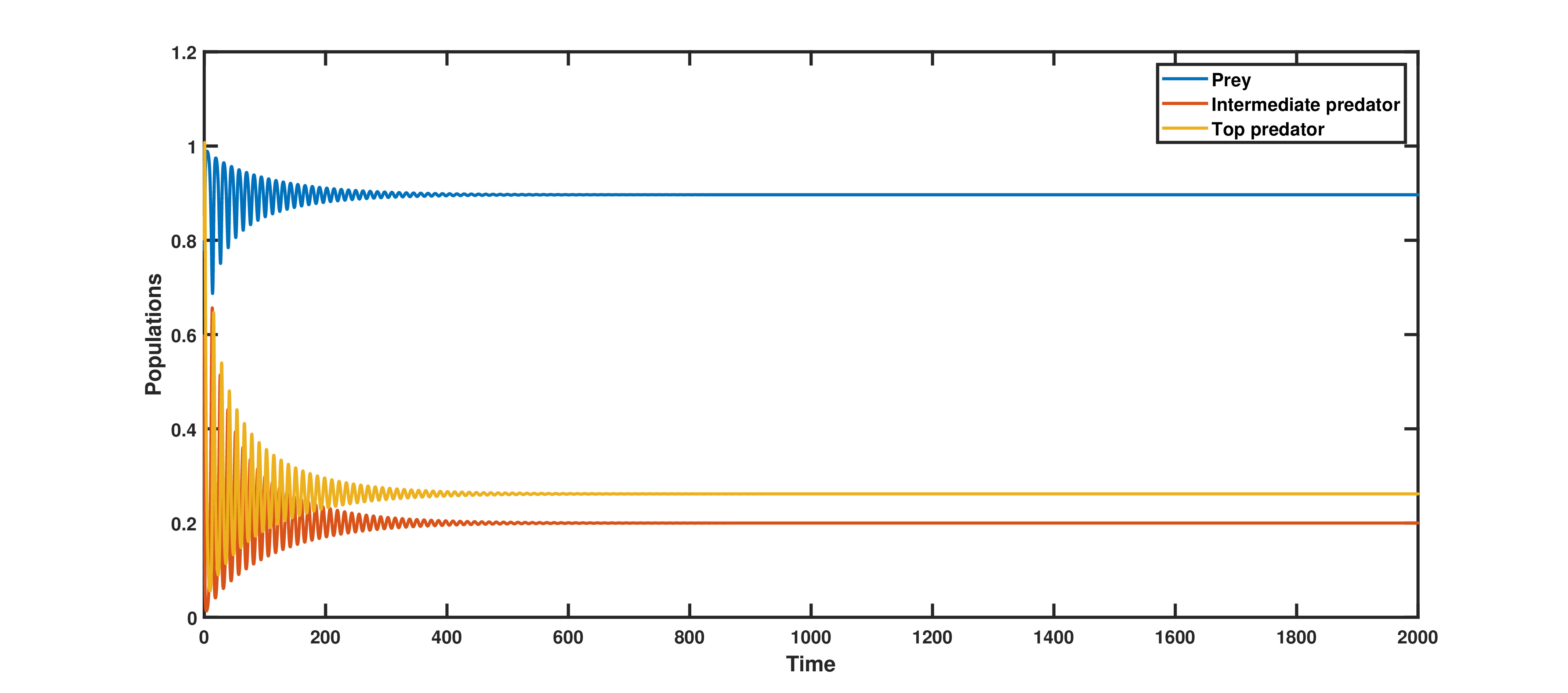}
         \caption{\emph{Case (\romannum{2}): When $\alpha=0.98$}}
         \label{m1=0=m2FDE0.98}
     \end{subfigure}
     \hfill
     \begin{subfigure}{0.45\textwidth}
         \centering
         \includegraphics[width=\textwidth]{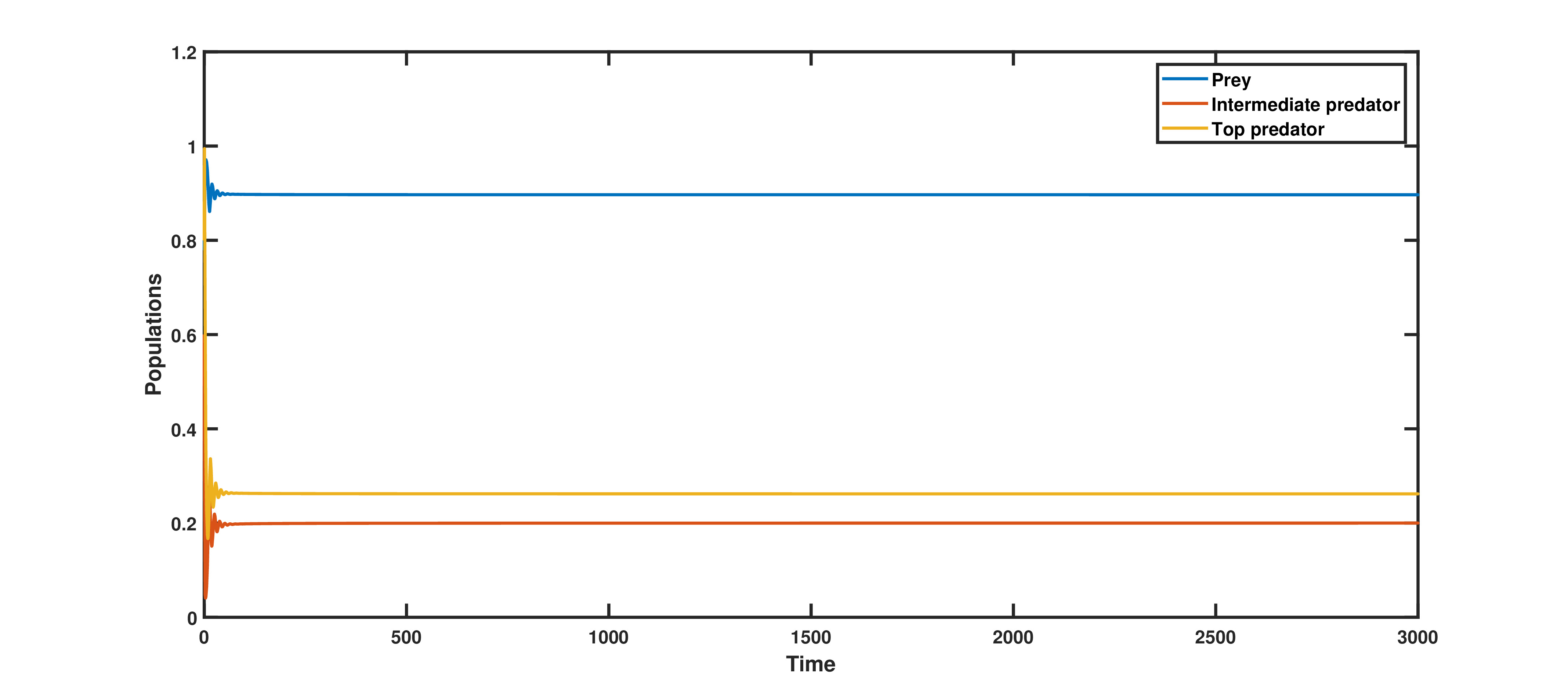}
         \caption{\emph{Case (\romannum{3}): When $\alpha=0.90$}}
         \label{m1=0=m2FDE0.9}
     \end{subfigure}
     \hfill
     \begin{subfigure}{0.45 \textwidth}
         \centering
         \includegraphics[width=\textwidth]{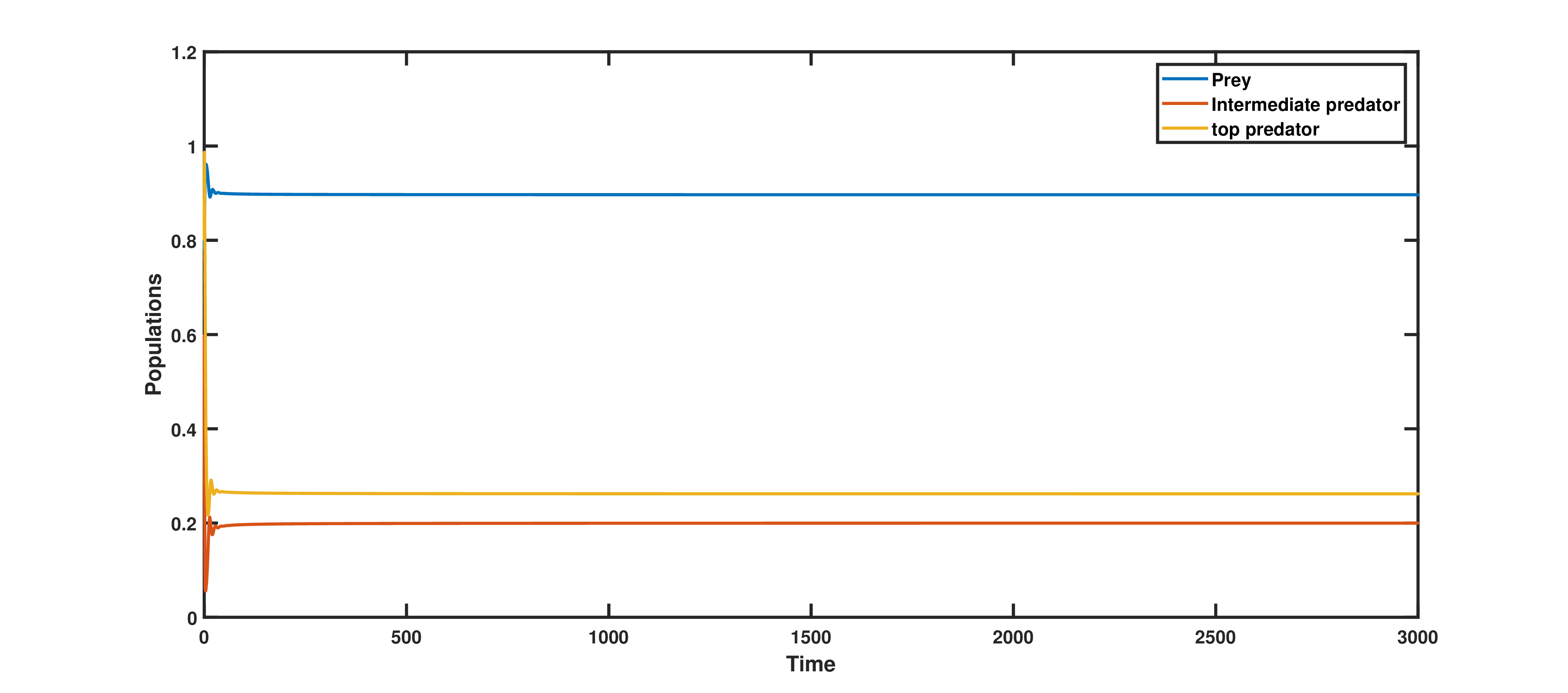}
         \caption{\emph{Case (\romannum{4}): When $\alpha=0.85$}}
         \label{m1=0=m2FDE0.85}
     \end{subfigure}
       \caption{\emph{The figures depict the impact of memory on the system in the absence of refuge behaviour. The parameter values used are as follows: $r_1=2$, $r_5 = 1$, $\beta = 0.01$, $r_2 = 1$, $m_1= 0$, $m_2 = 0$, $d_1 = 0.25$, $r_4 =3$, $d_2 = 0.5$, $b = 1$, $r_3 =1$, $q=0.5$, and $r = 0.01$.}}
        \label{m1=0=m2effectofmemory}
\end{figure}

\begin{figure}[H]
     \centering
     \begin{subfigure}{0.45\textwidth}
         \centering
         \includegraphics[width=\textwidth]{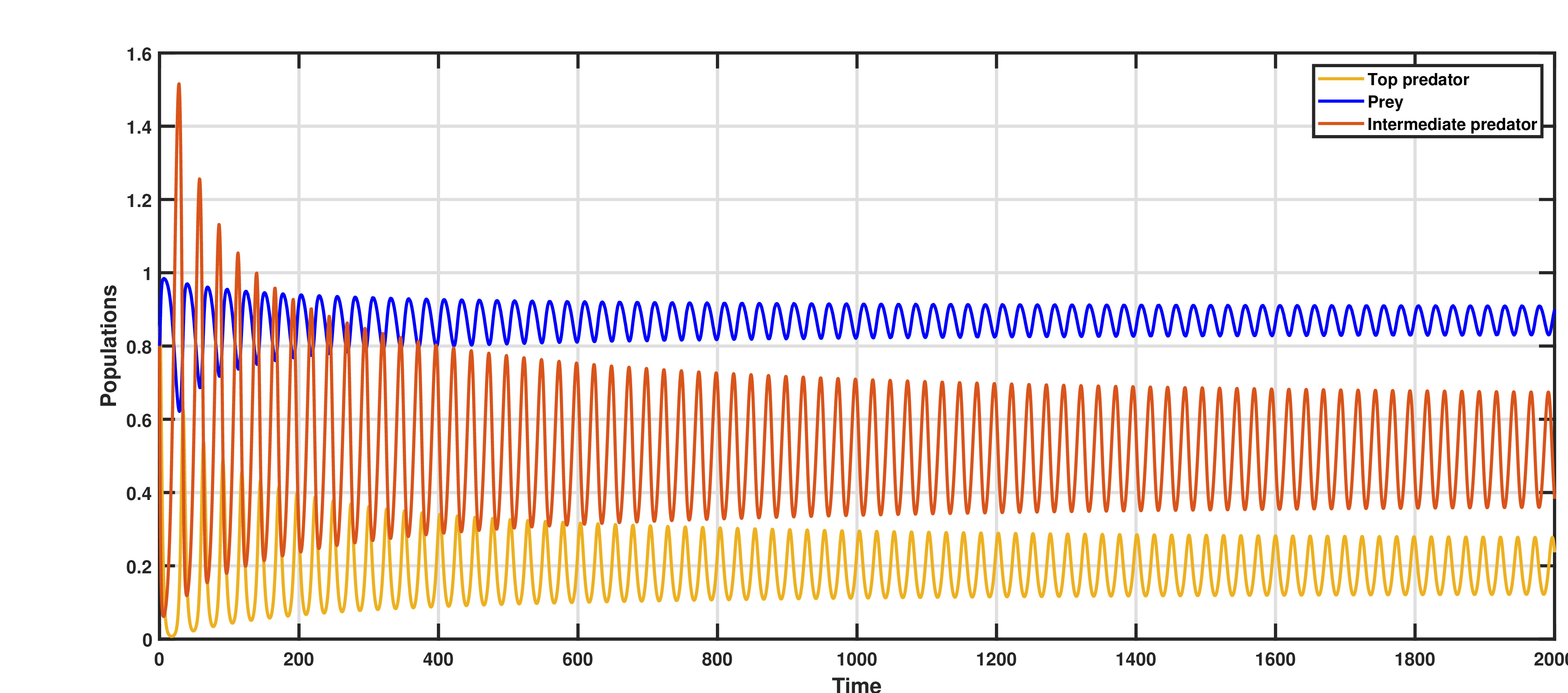}
         \caption{\emph{Case (\romannum{1}): When $\alpha=1$}}
         \label{q=r=0ode}
     \end{subfigure}
      \hfill
     \begin{subfigure}{0.45\textwidth}
         \centering
         \includegraphics[width=\textwidth]{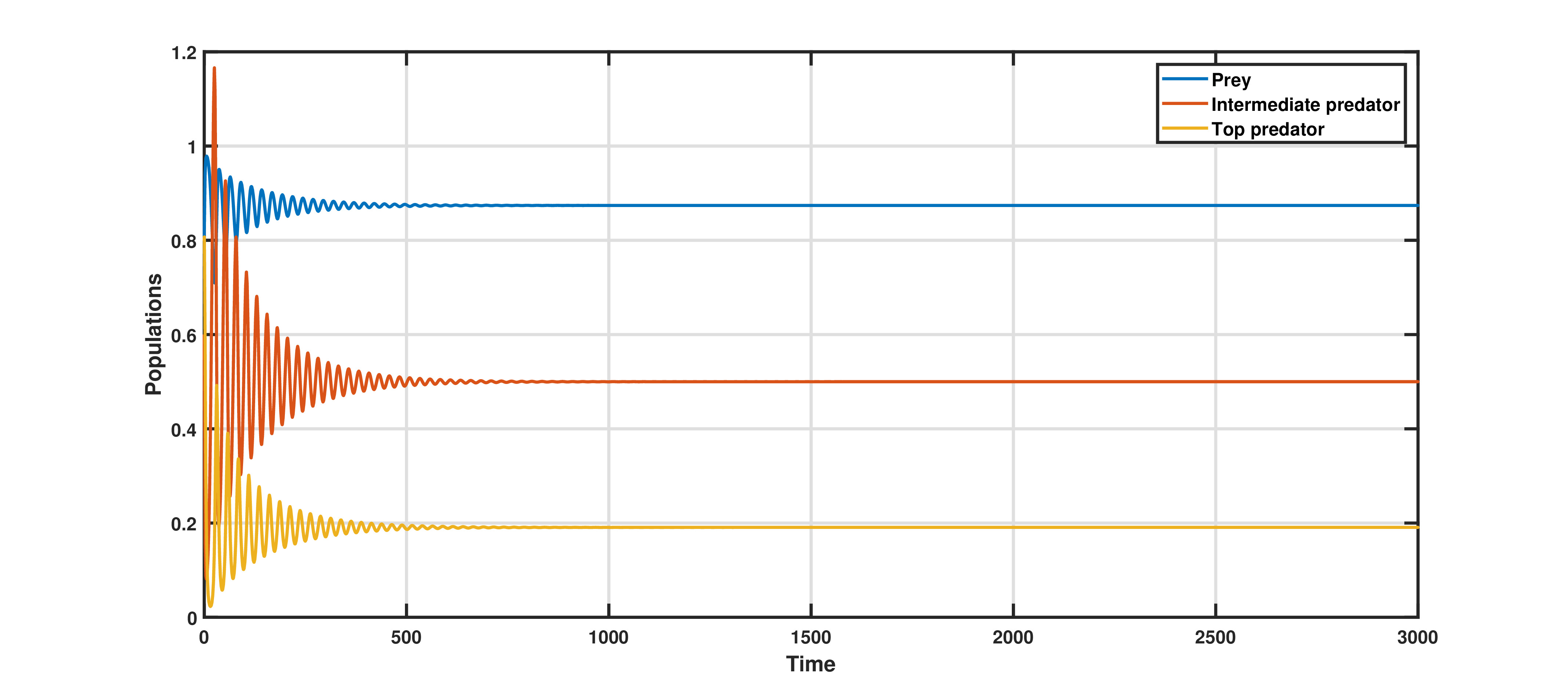}
         \caption{\emph{Case (\romannum{2}): When $\alpha=0.98$}}
         \label{q=r=0fde0.98}
     \end{subfigure}
     \hfill
     \begin{subfigure}{0.45\textwidth}
         \centering
         \includegraphics[width=\textwidth]{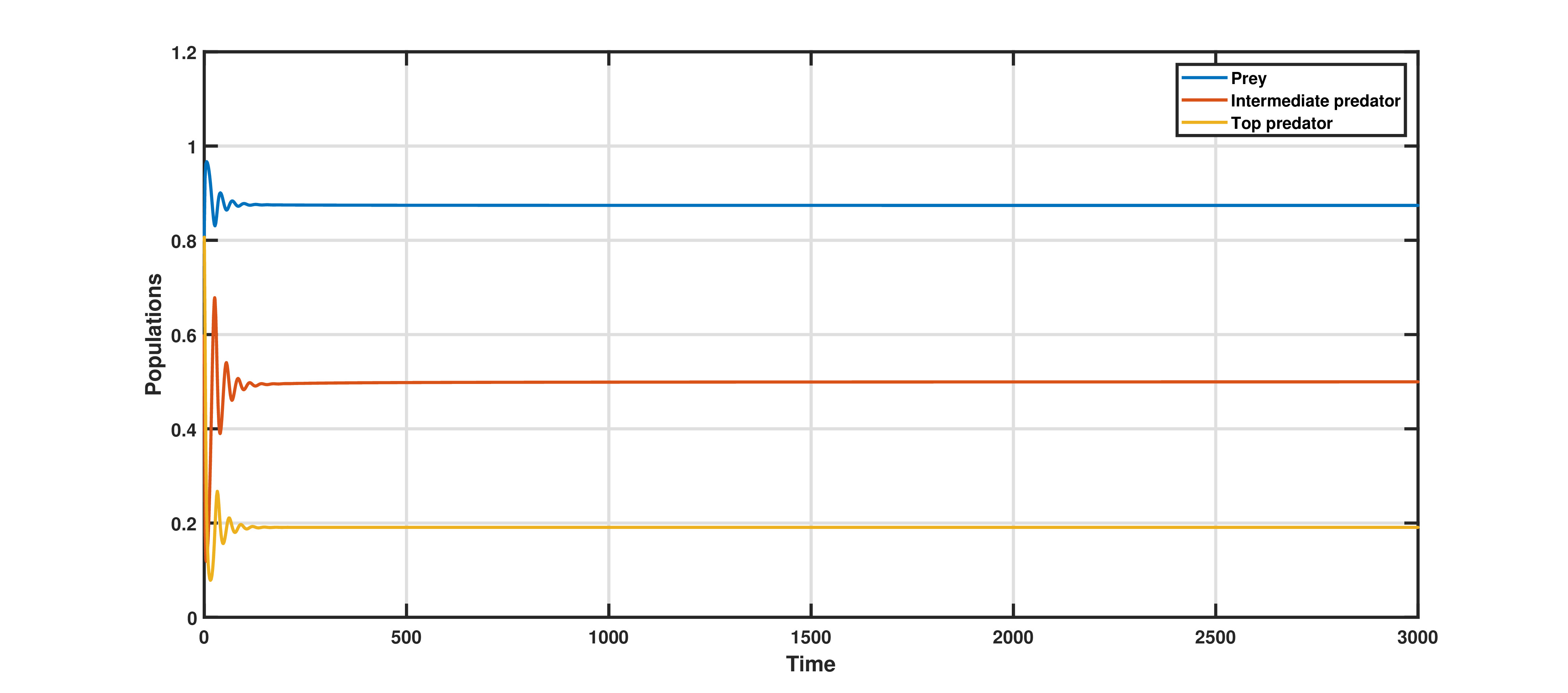}
         \caption{\emph{ Case (\romannum{3}): When $\alpha=0.90$}}
         \label{q=r=0fde0.9}
     \end{subfigure}
     \hfill
     \begin{subfigure}{0.45 \textwidth}
         \centering
         \includegraphics[width=\textwidth]{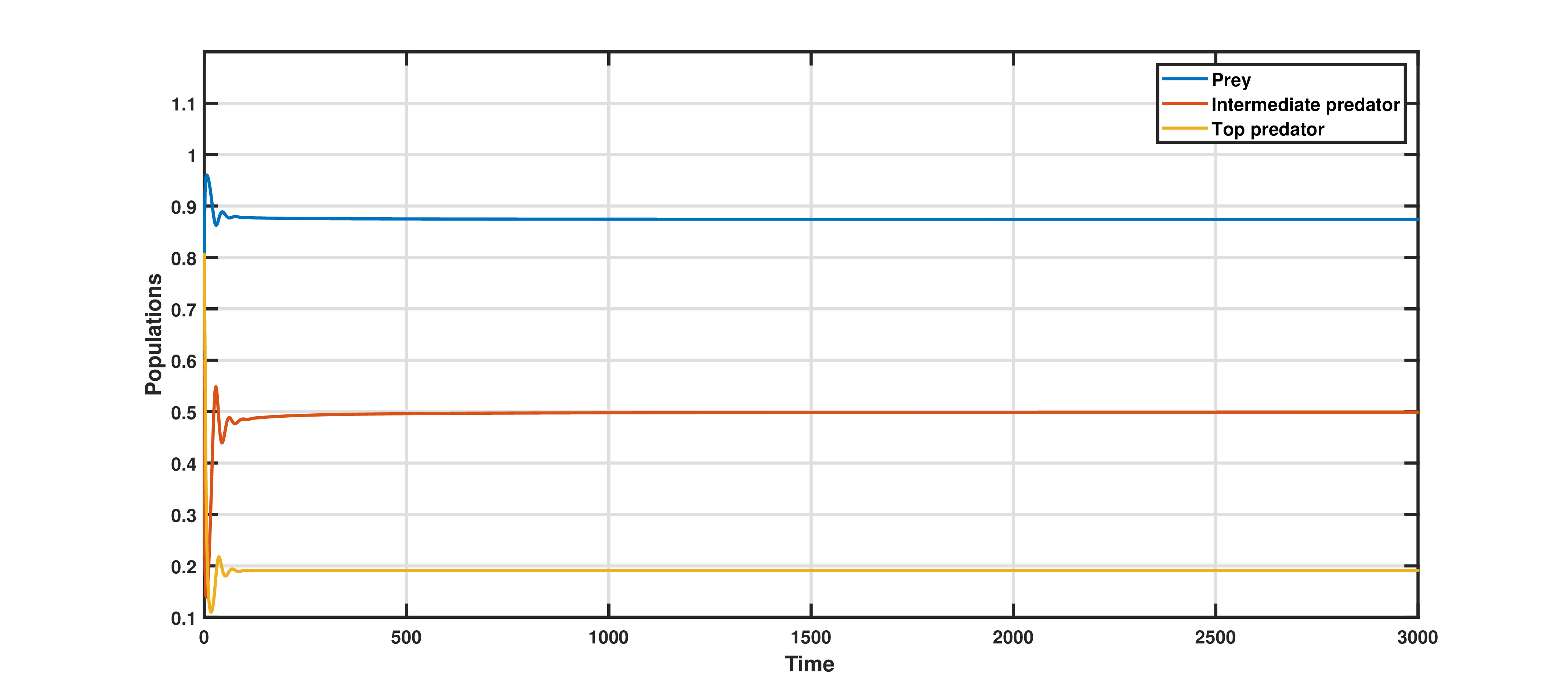}
         \caption{\emph{ Case (\romannum{4}): When $\alpha=0.85$}}
         \label{q=r=0fde0.85}
     \end{subfigure}
       \caption{\emph{In order to generate these figures, we have utilised specific parameter values. These values include: $r_1=2$, $r_3 =1$, $r_5 = 1$, $\beta = 0.01$, $m_1= 0.5$, $r_2 = 1$, $m_2 = 0.5$, $d_1 = 0.25$,  $r_4 =3$, $d_2 = 0.5$, $b = 1$, $q=0$, and $r = 0$. The long-term influence of memory on the system without harvesting is noticed.}}
        \label{q=r=0memoryeffect}
\end{figure}

\begin{figure}[H]
     \centering
     \begin{subfigure}{0.45\textwidth}
         \centering
         \includegraphics[width=\textwidth]{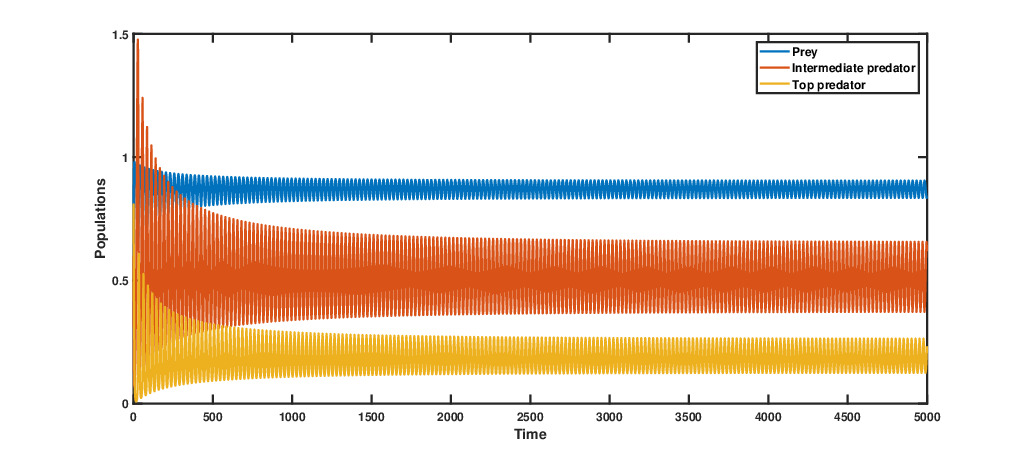}
         \caption{\emph{ Case (\romannum{1}): When $\alpha=1$}}
         \label{b2=0ode}
     \end{subfigure}
      \hfill
     \begin{subfigure}{0.45\textwidth}
         \centering
         \includegraphics[width=\textwidth]{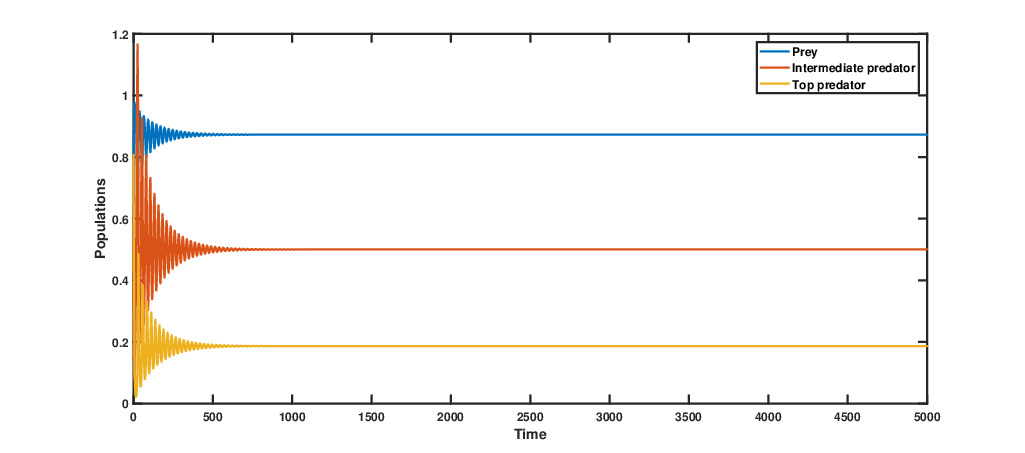}
         \caption{\emph{ Case (\romannum{2}): When $\alpha=0.98$}}
         \label{b2=0fde0.98}
     \end{subfigure}
     \hfill
     \begin{subfigure}{0.45\textwidth}
         \centering
         \includegraphics[width=\textwidth]{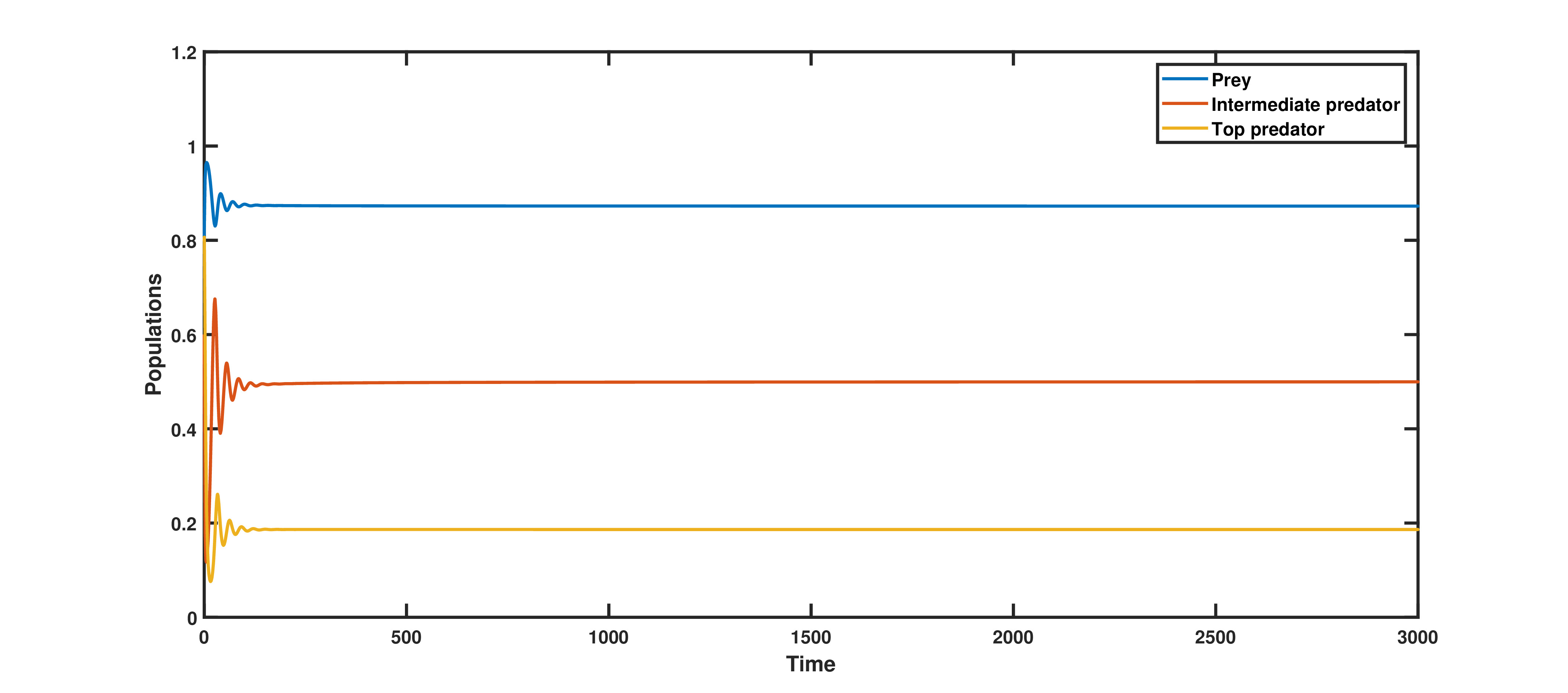}
         \caption{\emph{ Case (\romannum{3}): When $\alpha=0.90$}}
         \label{b2=0fde0.9}
     \end{subfigure}
     \hfill
     \begin{subfigure}{0.45 \textwidth}
         \centering
         \includegraphics[width=\textwidth]{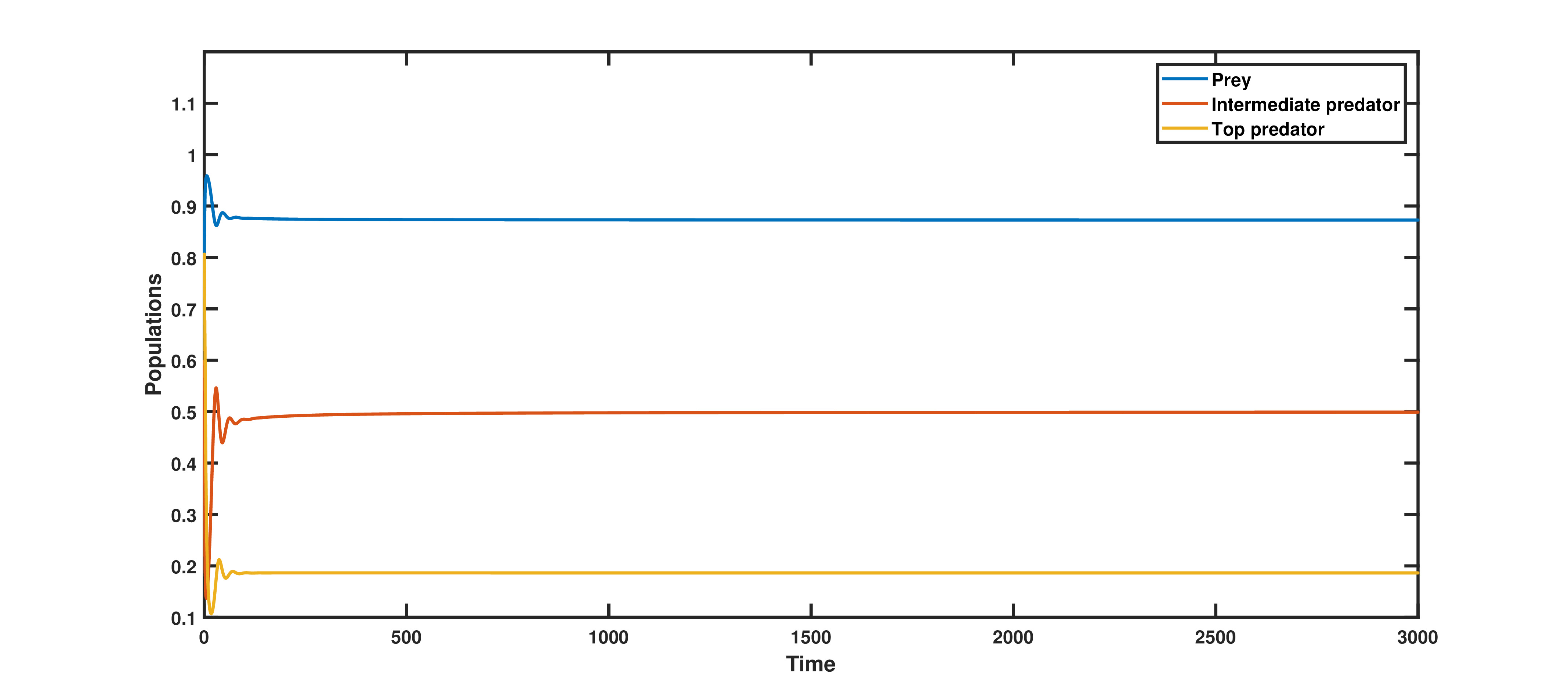}
         \caption{\emph{ Case (\romannum{4}): When $\alpha=0.85$}}
         \label{b2=0fde0.85}
     \end{subfigure}
       \caption{\emph{The long term effect of memory on the system without the prey odour effect is clearly visible from these figures. The parameter values used to generate these figures are: $r_1=2$, $r_5 = 1$, $\beta = 0$, $m_1= 0.5$, $r_2 = 1$, $m_2 = 0.5$, $d_1 = 0.25$, $d_2 = 0.5$, $r_4 =3$, $b = 1$, $q=0.5$, $r_3 =1$, and $r = 0.01$}}
        \label{b2=0memoryeffect}
\end{figure}

\begin{figure}[H]
     \centering
     \begin{subfigure}{0.45\textwidth}
         \centering
         \includegraphics[width=\textwidth]{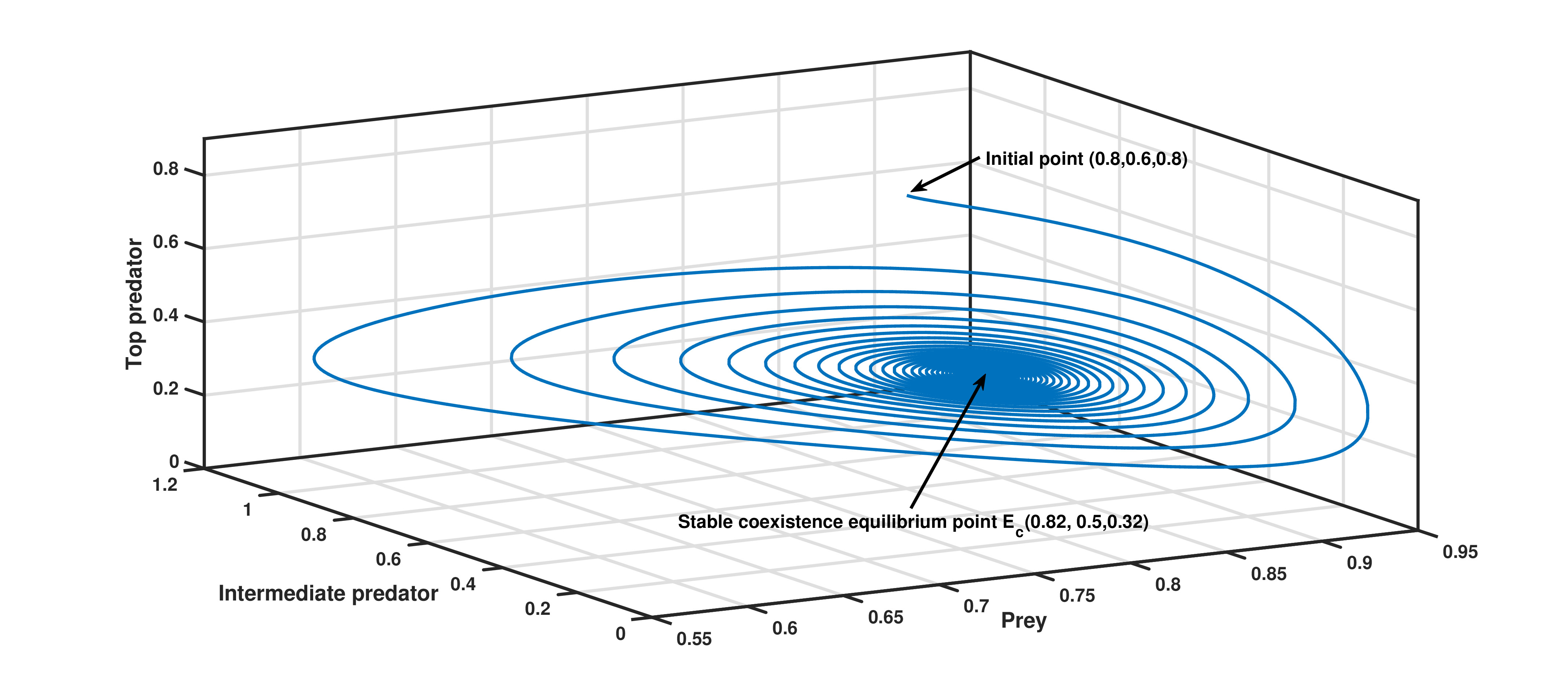}
         \caption{\emph{ This diagram illustrates the state immediately prior to the onset of a Hopf bifurcation at $m_1=0.4498$. This figure was generated using a parameter value of $m_1=0.3$, along with other mentioned parameter values. }}
         \label{m1 0.3}
     \end{subfigure}
      \hfill
     \begin{subfigure}{0.45\textwidth}
         \centering
         \includegraphics[width=\textwidth]{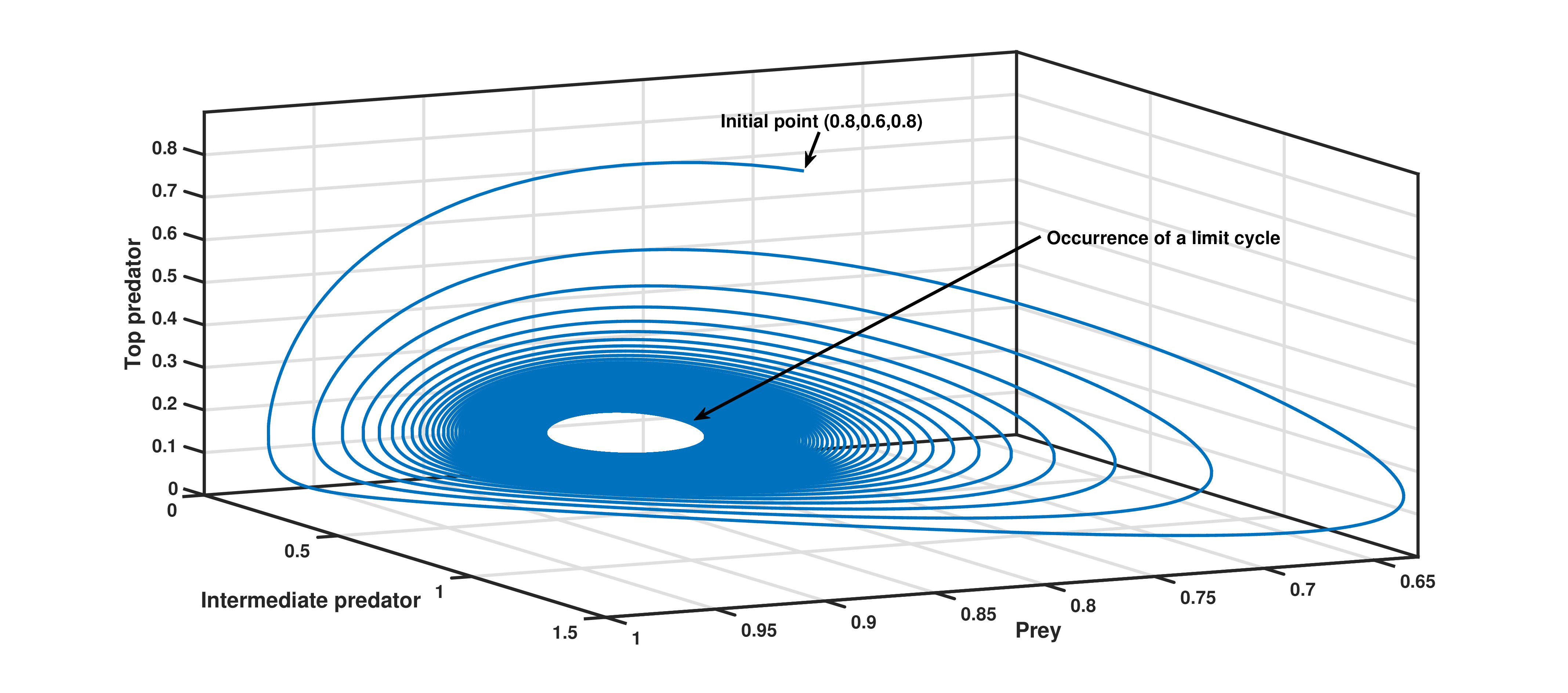}
         \caption{\emph{ The illustration depicts the situation right after the onset of a Hopf bifurcation at $m_1=0.4498$. A parameter value of $m_1=0.5$, in addition to the other parameter values mentioned below was employed to generate this figure.}}
         \label{m1 0.5}
     \end{subfigure}
     \hfill
     \begin{subfigure}{0.45\textwidth}
         \centering
         \includegraphics[width=\textwidth]{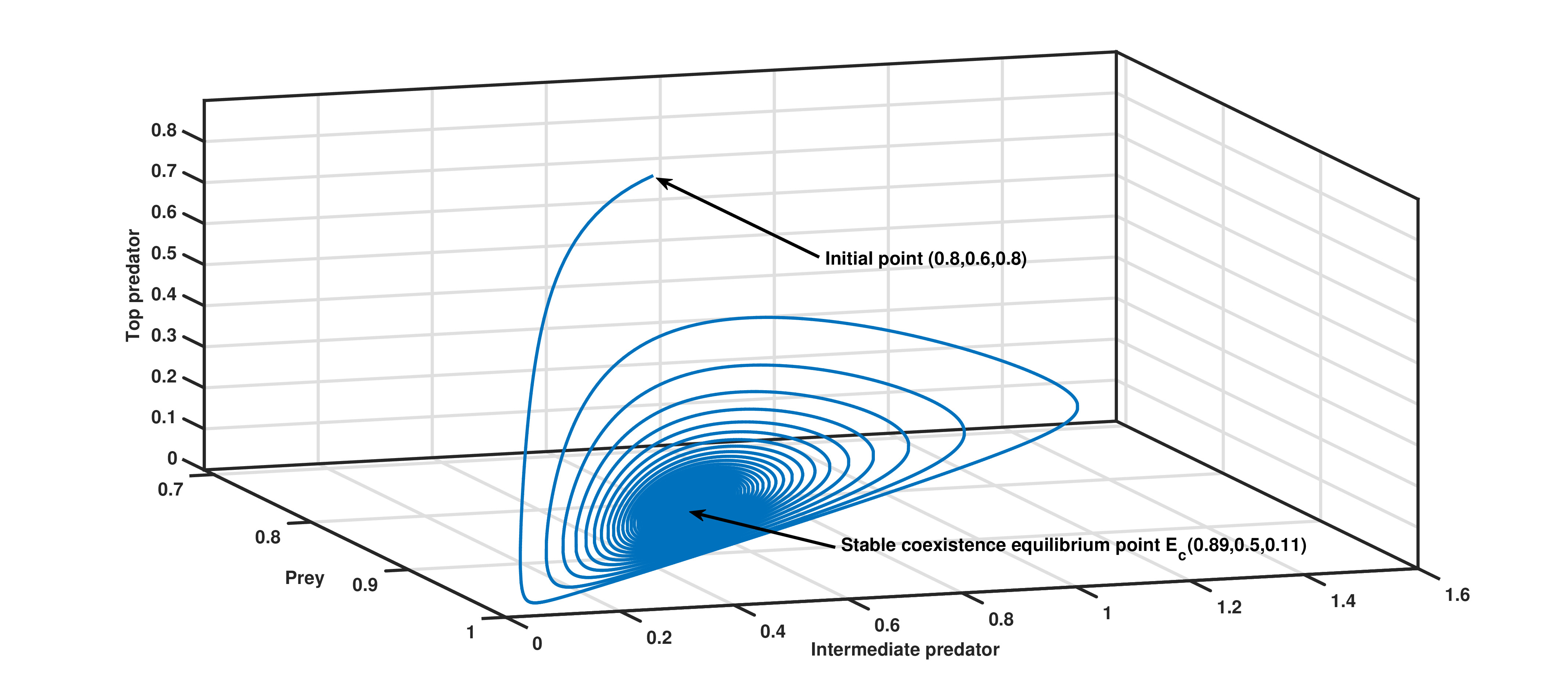}
         \caption{\emph{ This figure portrays the scenario immediately following the emergence of the second Hopf bifurcation at $m_1=0.5295$. This figure was generated using a parameter value of $m_1=0.6$, along with the other below mentioned parameter values.}}
         \label{m1 0.6}
     \end{subfigure}
     \hfill
     \begin{subfigure}{0.4 \textwidth}
         \centering
         \includegraphics[width=\textwidth]{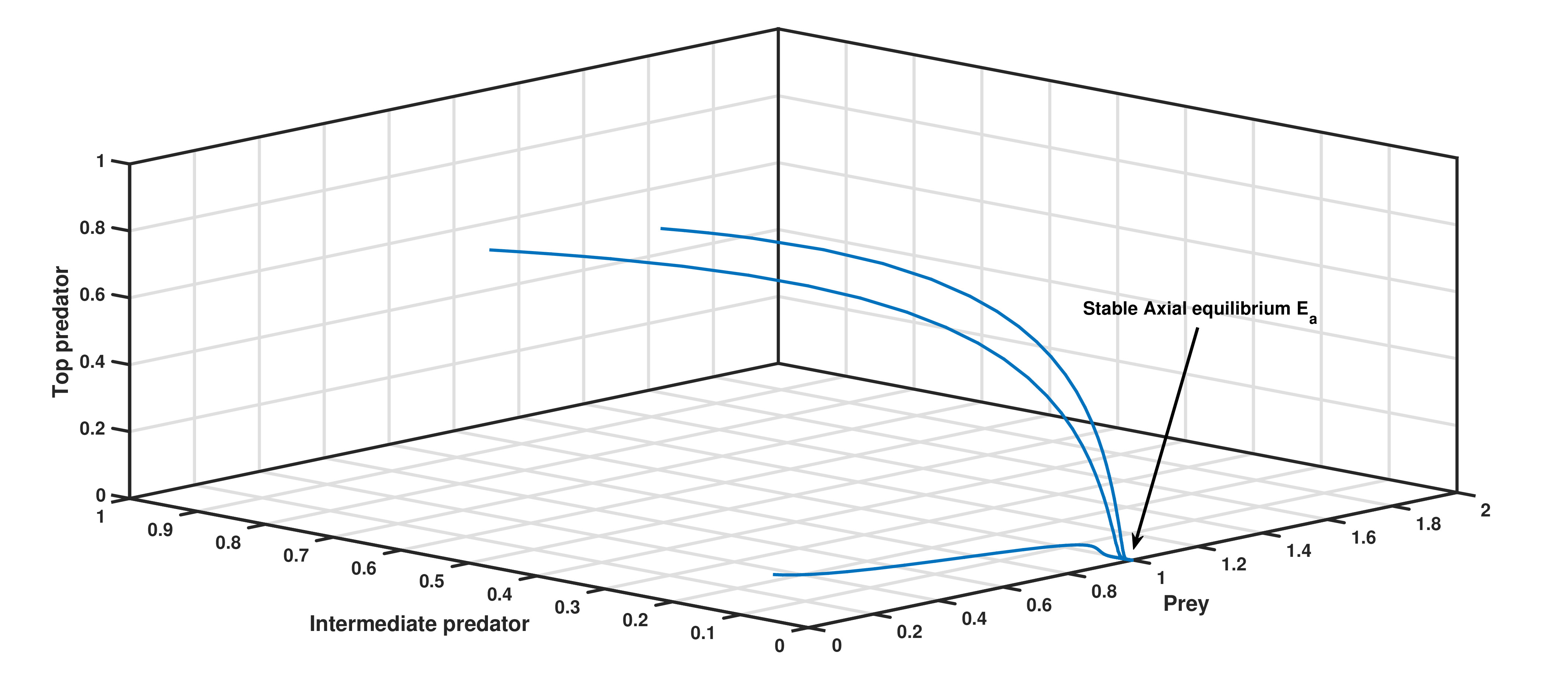}
         \caption{\emph{This figure demonstrates the situation that occurs right after the occurrence of a transcritical bifurcation at $m_1=0.733751855$. In addition to the other parameter values displayed below, the value of $m_1=0.8$ was used to generate this figure.}}
         \label{m1=0.8}
     \end{subfigure}
       \caption{\emph{ These figures display the various scenarios that occur in the system (\ref{Final ode eq}) when the parameter $m_1$ is changed, while keeping the other parameter values constant at $r_1=2$, $r_5 = 1$, $\beta = 0.01$, $m_2 = 0.5$, $r_2 = 1$, $d_1 = 0.25$, $r_3 =1$, $d_2 = 0.5$, $r_4 =3$, $b = 1$, $q=0.5$, $r = 0.01$ and $\alpha=1$.}}
        \label{m1effect}
\end{figure}
In figure (\ref{m10.5odets}), we can clearly see highly concentrated fluctuations in the populations of the three species, whereas in figure (\ref{m10.5fde0.98ts}), the fluctuations gradually disappear after a certain amount of time. Furthermore, it can be observed from figures (\ref{m10.5fde0.9ts}) and (\ref{m10.5fde0.85ts}) that the scale of fluctuations diminishes as the fractional order reduces, inducing more stability within the system. Similarly, figure (\ref{m2=0effectofmemory}) illustrates that as the value of $\alpha$ decreases from 1 to 0, the fluctuations within the system become stabilised, demonstrating the impact of memory on the system. In a similar way figure (\ref{m1=0=m2effectofmemory}) illustrates that in the system (\ref{Frac eq}) where the prey species and intermediate predators do not exhibit refuge behaviour against their predators, the fluctuations within the system increases as the species' memory diminishes. The same conclusion may be deduced from figures (\ref{q=r=0memoryeffect}) and (\ref{b2=0memoryeffect}). Figure (\ref{q=r=0memoryeffect}) illustrates the impact of fading memory on the system (\ref{Frac eq}) in the absence of any harvesting. Figure (\ref{b2=0memoryeffect}) illustrates the impact of memory on the system (\ref{Frac eq}) in the absence of prey odour effect. Hence, based on the aforementioned figures, it has been noted that individuals who suffer from memory loss or a deterioration in their ability to recall prior events can have a negative impact on the stability of the corresponding system.

\begin{figure}[H]
     \centering
     \begin{subfigure}{0.45\textwidth}
         \centering
         \includegraphics[width=\textwidth]{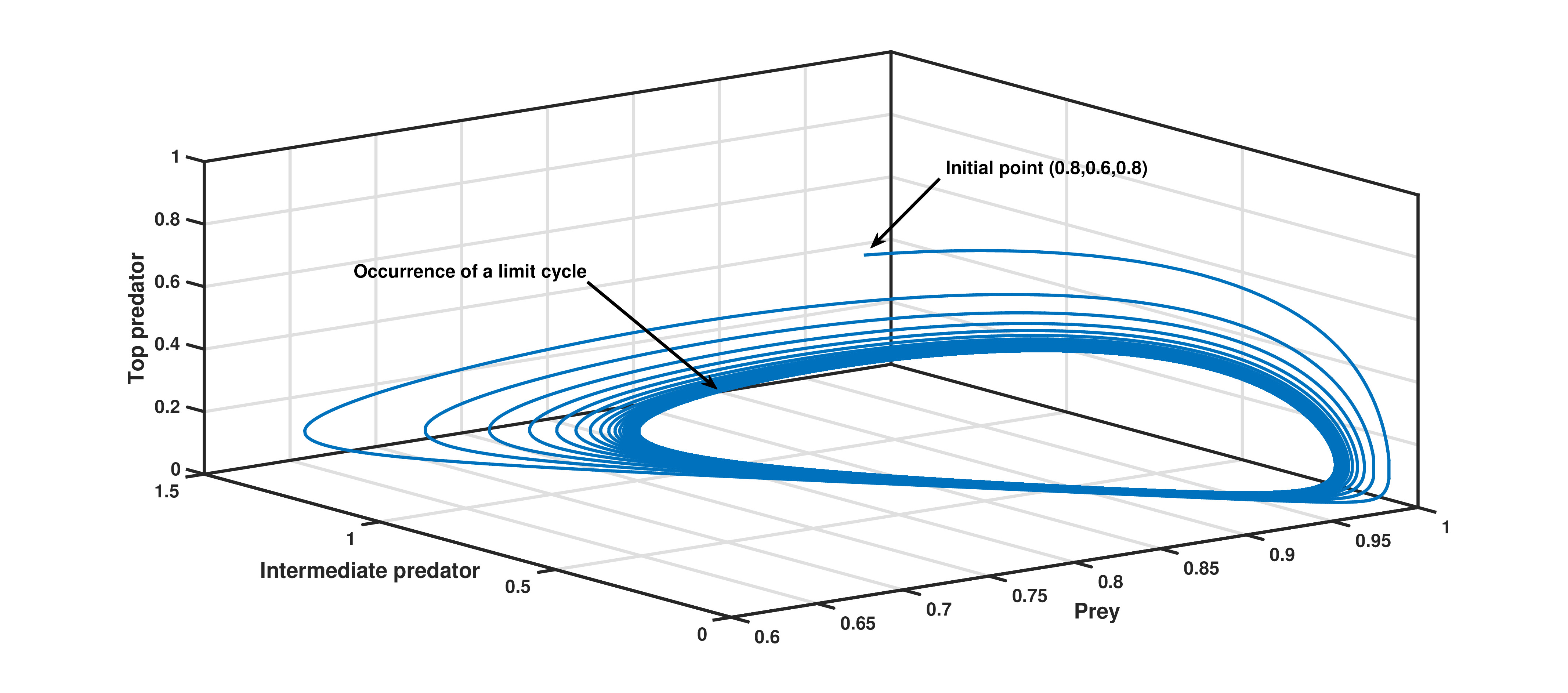}
         \caption{\emph{ This diagram exhibits the pre-Hopf bifurcation state. The figure is generated using a value of $m_2=0.4$, in addition to the other given parameter values.}}
         \label{m2 0.4}
     \end{subfigure}
      \hfill
     \begin{subfigure}{0.45\textwidth}
         \centering
         \includegraphics[width=\textwidth]{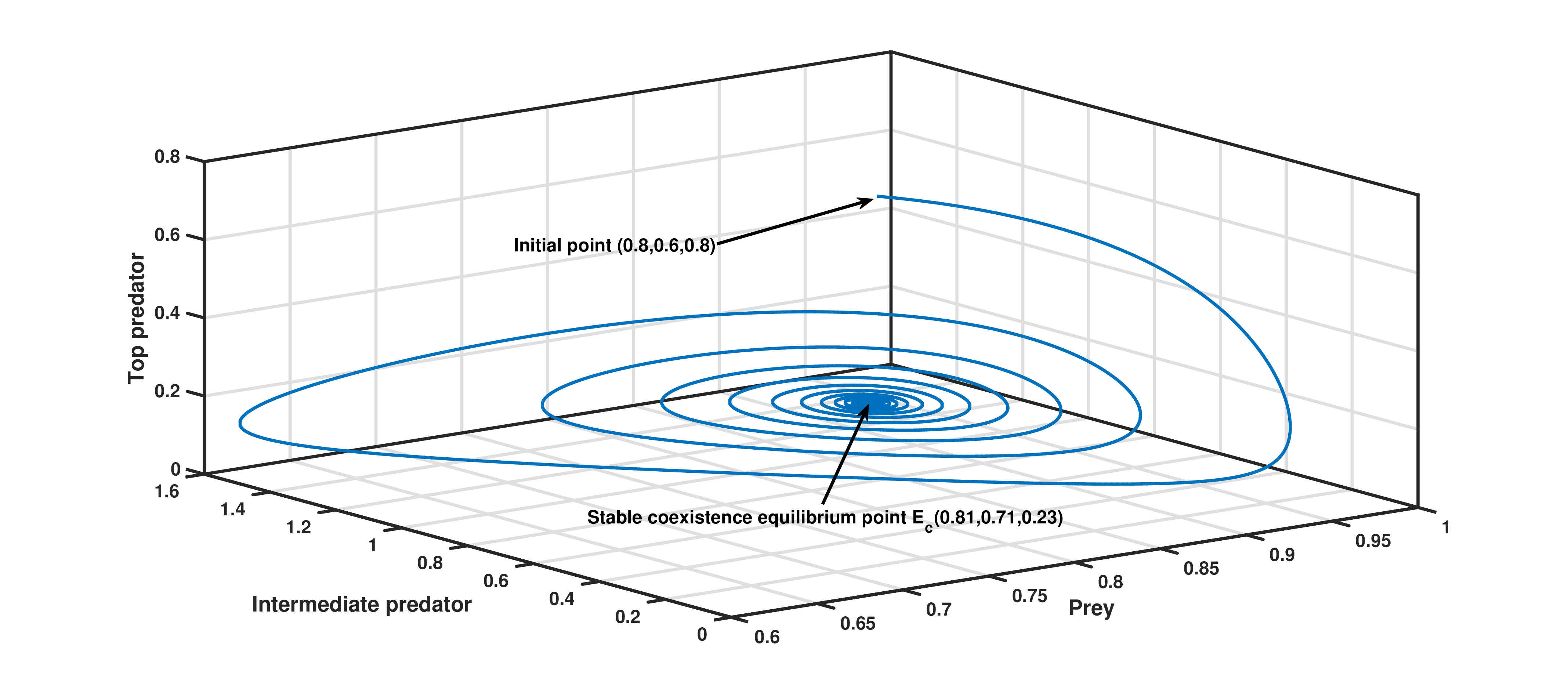}
         \caption{\emph{This image illustrates the immediate aftermath of a Hopf bifurcation at $m_2=0.503528$. This figure was produced using the value of $m_2=0.6$ in addition to the other parameter values shown below.}}
         \label{m2=0.6}
     \end{subfigure}
     \hfill
     \begin{subfigure}{0.45\textwidth}
         \centering
         \includegraphics[width=\textwidth]{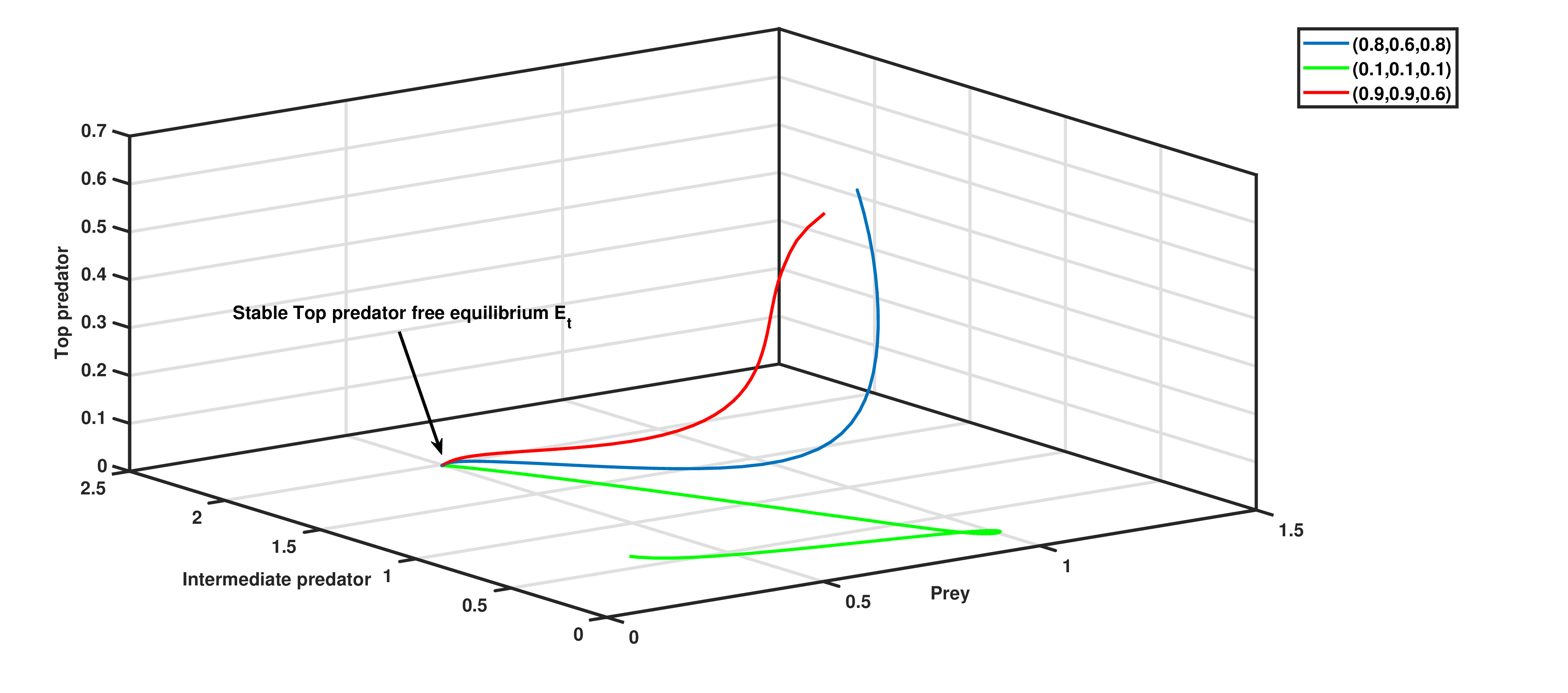}
         \caption{\emph{This diagram illustrates the immediate aftereffects of a transcritical bifurcation at $m_2=0.74958126$. Moreover this figure is generated using a value of $m_2=0.8$, in addition to the other parameter values mentioned below.}}
         \label{m2=0.8}
     \end{subfigure}
       \caption{\emph{These illustrations show the various scenarios that arise in the system (\ref{Final ode eq}) when the parameter $m_2$ is altered, while keeping the other parameter values constant at $r_1=2$, $r_2 = 1$ $\beta = 0.01$, $m_1 = 0.5$, $r_3 =1$, $d_1 = 0.25$, $r_4 =3$, $d_2 = 0.5$, $b = 1$, $r_5 = 1$, $q=0.5$, $r = 0.01$ and $\alpha=1$ .}}
        \label{m2effect}
\end{figure}

\subsection{Impact of the parameters $m_1$, $m_2$ on the population dynamics of the system (\ref{Frac eq}) for all $\alpha \in (0,1]$} In this part, we analyse the impact of the parameters related to the refuge triggered by the predator odour in the aforementioned system. At first, we analyse the effects of parameter $m_1$ on the system (\ref{Frac eq}) for $\alpha=1$, i.e., the system (\ref{Final ode eq}). The figure (\ref{m1equilibrium curve}) provide a clear and visually appealing demonstration of the several bifurcations that occur when the parameter $m_1$ is altered. We accomplish this by assuming other parameter values like $r_1=2$, $r_3 =1$, $\beta = 0.01$, $r_2 = 1$, $m_2 = 0.5$, $d_1 = 0.25$, $d_2 = 0.5$, $r_4 =3$, $b = 1$, $q=0.5$, $r_5 = 1$, $r = 0.01$, only altering the parameter $m_1$. If the value of parameter $m_1$ is below a threshold value of $m_1=0.4498=m_1^{hb1}$, then the system (\ref{Final ode eq}) exhibits stability around the fixed point $E_c$. Given the following set of parameter values, it can be readily confirmed numerically: $r_1=2$, $\beta = 0.01$, $m_1= 0.3$, $r_2 = 1$, $m_2 = 0.5$, $r_3 =1$, $d_1 = 0.25$, $r_4 =3$, $d_2 = 0.5$, $b = 1$, $r_5 = 1$, $q=0.5$, and $r = 0.01$. For this particular set of parameter values, we obtain 
that all the requirements needed to achieve stability given in theorem (\ref{lsco}) are satisfied. This confirms the stability of $E_c$. This is shown in figure (\ref{m1 0.3}). When the value of $m_1$ reaches a certain threshold, known as the Hopf bifurcation point, the system described by equation (\ref{Final ode eq}) undergoes a significant shift in stability. Numerical verification of the existence of a Hopf bifurcation can also be accomplished by applying the theorem (\ref{odehopfthm}). Given the threshold value of $m_1=m_1^{hb1}$ and keeping other parameter values unchanged, we find that all the conditions specified in theorem (\ref{odehopfthm}) are fulfilled.
This confirms the existence of a Hopf bifurcation at $m_1=m_1^{hb1}$. Furthermore, if the value of $m_1$ is slightly elevated, the system (\ref{Final ode eq}) demonstrates instability. Assuming the value of $m_1$ is $0.5$, while keeping all other parameter values the same, we find that $N_1N_2-N_3=-0.001<0$. This result corroborates the instability of the system near $E_c$. This situation is shown visually in figure (\ref{m1 0.5}). Another Hopf bifurcation occurs when the parameter $m_1$ is increased to $m_1=0.5295=m_1^{hb2}$, resulting in a change in the stability of the system (\ref{Final ode eq}). When $m_1=0.5295=m_1^{hb2}$ and all other parameter values remain the same, all the conditions of theorem (\ref{odehopfthm}) are satisfied. 
This validates the presence of a second Hopf bifurcation at $m_1=0.5295=m_1^{hb2}$. If the values of $m_1$ exceed $m_1^{hb2}$, the system (\ref{Final ode eq}) demonstrates stability in the vicinity of the fixed point $E_c$. Numerical confirmation can be obtained by evaluating the parameter values. For example, when $m_1=0.6>m_1^{hb2}$ and all other parameters are unaltered, we find that $N_1=1.75>0$, $N_2=0.04>0$, $N_3=0.06>0$, and $N_1N_2-N_3=0.01>0$. Figure (\ref{m1 0.6}) depicts this same scenario. On the other hand, if we slightly increase the value of the parameter $m_1$, an interesting phenomenon occurs. At a certain point, denoted as $m_1=0.73375185=m_1^{tbp}$, a transcritical bifurcation occurs. This results in the system (\ref{Final ode eq}) being unstable around $E_c$, whereas the fixed point $E_a$ becomes stable.  Numerically, when $m_1$ is equal to 0.8 and all other parameters remain unchanged, we obtain $r<4= \frac{r_1}{q}$ and $r_2=1 < \frac{-d_1 r_1^2}{-\beta r^2 q^2 r_3 + \omega_6}$. Additionally, we have $N_3=-0.03<0$. Therefore, it is evident that $E_a$ is stable while $E_c$ is unstable. The precise depiction of the particular scenario can be observed in figure (\ref{m1=0.8}). Based on the preceding discussion, it is evident that the parameter $m_1$ plays a significant role in the system (\ref{Final ode eq}). By looking at the situation from an an ecological standpoint, it has been observed that when the parameter $m_1$ is set to a very low value, the prey population tends to remain small. On the other hand, when the amount of prey refuge increases, it has a direct positive impact on the population of the prey species, resulting in increase in the prey population as well. As the parameter $m_1$ increases, an interesting phenomenon occurs. At a certain value, fluctuations in the population of the three species are observed. These fluctuations are subsequently stabilised by the occurrence of another Hopf bifurcation at a slightly elevated value of the parameter $m_1$. By increasing the parameter $m_1$, a remarkable phenomenon known as a transcritical bifurcation occurs. This leads to an unstable coexistence equilibrium, making it impossible for all three populations to cohabit.

Furthermore, it has been noted that for the range of $\alpha$, $0<\alpha<1$ in the system (\ref{Frac eq}), there is no alteration in the stability of the coexistence equilibrium point $E_{c}$ within the interval $[0,m_1^{tbp})$ of the parameter $m_1$. Within the range $(m_1^{tbp},1)$, the system exhibits behavior consistent with that observed in the system described by equation (\ref{Final ode eq}).

Now, we will examine the significance of the parameter $m_2$ in the system (\ref{Final ode eq}). To do this, we set other parameters to values like $r_1 = 2$, $r_5 = 1$, $\beta = 0.01$, $m_1 = 0.5$, $d_1 = 0.25$, $r_2 = 1$, $d_2 = 0.5$, $r_4 = 3$, $b = 1$, $r_3 = 1$, $q = 0.5$, $r = 0.01$, and change only $m_2$. Based on the information provided in figure (\ref{m2equilibrium curve}), it is clear that the parameter $m_2$ plays a significant role in the system (\ref{Final ode eq}). For all values of the parameter $m_2$ that are less than a threshold value $m_2=0.503528=m_2^{hbp}$, and with all other parameter values remaining the same, the system (\ref{Final ode eq}) exhibits instability around the fixed point $E_c$. More precisely, the system (\ref{Final ode eq}) exhibits oscillations in populations. For instance, when we set the parameter values as follows: $r_1 = 2$, $r_4 = 3$, $\beta = 0.01$, $m_1 = 0.5$, $r_2 = 1$, $m_2 = 0.4$, $r_3 = 1$, $d_1 = 0.25$, $d_2 = 0.5$, $b = 1$, $q = 0.5$, $r_5 = 1$, and $r = 0.01$, the expression $N_1 N_2-N_3$ evaluates to -0.02, which is less than 0. This confirms the instability of $E_c$. It is shown in figure (\ref{m2 0.4}). The value $m_2=0.503528=m_2^{hbp}$ represents a Hopf bifurcation point, as it fulfils all the criteria stated in theorem (\ref{odehopfthm}). 
Therefore, it is numerically confirmed that a Hopf bifurcation occurs at $m_2=0.503528=m_2^{hbp}$. This Hopf bifurcation stabilizes the system until a transcritical bifurcation occurs, which alters its stability. This transcritical bifurcation occurs at $m_2=0.74958126=m_2^{tbp}$, along with other parameter values that remain unchanged. For example, when the parameter values are set as follows: $r_1 = 2$, $r_4 = 3$, $\beta = 0.01$, $m_1 = 0.5$, $m_2 = 0.6$, $r_2 = 1$, $d_1 = 0.25$, $r_3 = 1$, $d_2 = 0.5$, $r_5 = 1$, $b = 1$, $q = 0.5$, and $r = 0.01$, the resulting values are $N_1=1.57>0$, $N_2=0.086>0$, $N_3=0.077>0$, and $N_1N_2-N_3=0.058>0$. This confirms the stability of $E_c$. This particular situation is illustrated in figure (\ref{m2=0.6}). For a different set of parameter values, specifically $r_1 = 2$, $r_5 = 1$, $\beta = 0.01$, $m_1 = 0.5$, $r_4 = 3$, $m_2 = 0.8$, $d_1 = 0.25$, $d_2 = 0.5$,  $r_3 = 1$, $b = 1$, $r_2 = 1$, $q = 0.5$, and $r = 0.01$, it has been determined that $N_1 N_2-N_3$ is equal to -9.39, which is less than 0. Moreover, all the conditions outlined in theorem (\ref{lsto}) are also met, confirming the instability of $E_c$ and the stability of $E_t$. The precise depiction of the particular circumstance can be observed in figure (\ref{m2=0.8}). When the intermediate predator seeks shelter from the top predator's predation pressure, the system (\ref{Final ode eq}) becomes more intricate and intriguing. It is evident that when the parameter $m_2$ is extremely low, oscillations in the population of all three species within the system occur. However, as this parameter increases, the fluctuations gradually decrease. Eventually, after reaching a certain value, these fluctuations stabilise due to the occurrence of a Hopf bifurcation. As the parameter $m_2$ increases, the population of the intermediate predator similarly increases. However, the population of prey continues to decline, a trend that can be rationalised from an ecological standpoint. When the parameter $m_2$ reaches a higher value, a transcritical bifurcation takes place, resulting in the inability of all three species to coexist within the system (\ref{Final ode eq}). This phenomenon can be readily understood from an ecological standpoint. When the intermediate predator strategically enhances its protection against predation by the top predator in response to the top predator's odour, thereby reducing its vulnerability to being hunted. Consequently, the predation pressure on the prey by the intermediate predator increases, leading to a substantial decline in the prey population. As a result, the coexistence of all three species within the system becomes extremely difficult.

When the parameter $m_2$ falls within the range $[0,m_2^{tbp})$, the coexisting equilibrium point demonstrates stability in the system (\ref{Frac eq}) for $0<\alpha<1$. This is in stark contrast to the fluctuations noted in the system (\ref{Final ode eq}) for the range $[0,m_2^{hbp})$, where $m_2^{hbp}< m_2^{tbp}$. When the parameter $m_2$ falls within the interval $(m_2^{tbp},1)$, the system outlined in (\ref{Frac eq}) for $0<\alpha<1$ exhibits behavior analogous to that of the system described in (\ref{Final ode eq}).\\\\

\begin{figure}[H]
     \centering
     \begin{subfigure}{0.45\textwidth}
         \centering
         \includegraphics[width=\textwidth]{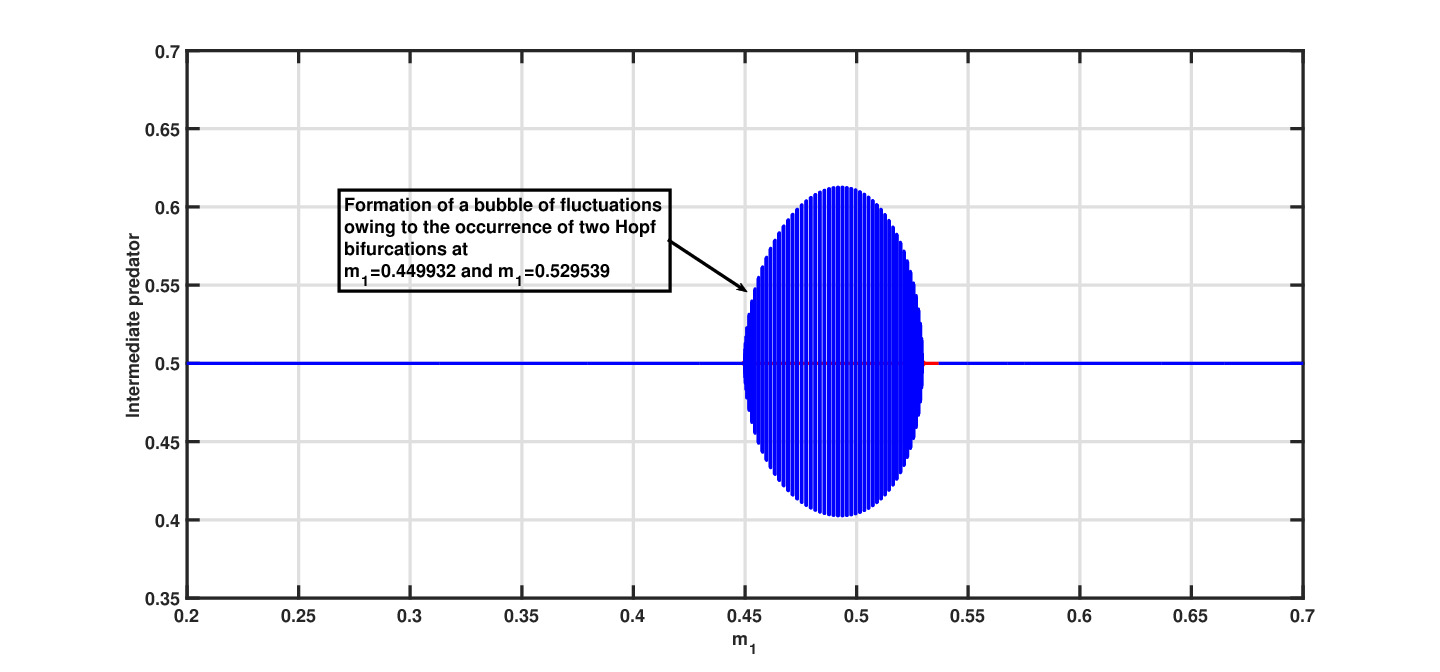}
         \caption{\emph{This figure displays the occurrence of a bubble of oscillations in $x_2-m_1$ plane. }}
         \label{x2m1}
     \end{subfigure}
      \hfill
      \begin{subfigure}{0.45\textwidth}
         \centering
         \includegraphics[width=\textwidth]{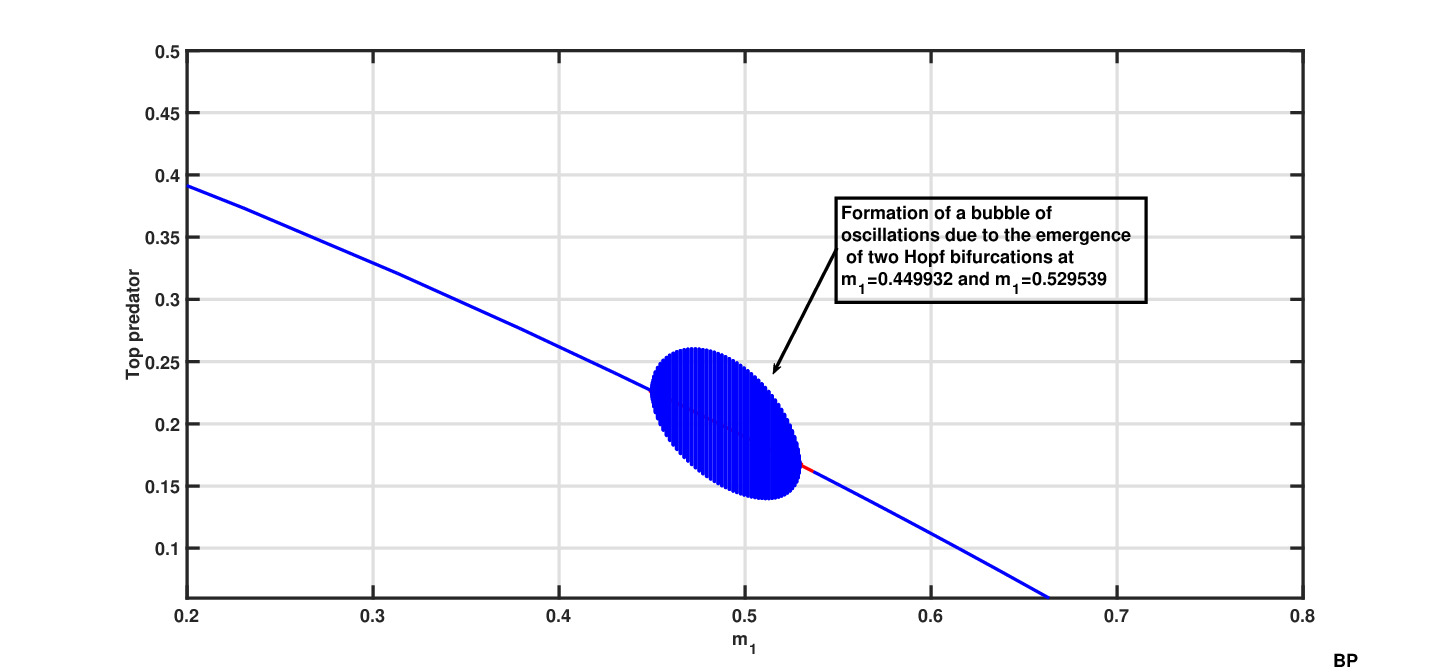}
         \caption{\emph{This picture illustrates the presence of a bubble of fluctuations in the $x_3-m_1$ plane.}}
         \label{x3m1}
     \end{subfigure}
       \caption{\emph{ The figures are produced utilising the following parameter values: $r_1=2$, $r_3=1$, $\beta=0.01$, $r_2=1$, $m_2=0.5$, $d_1=0.25$, $d_2=0.5$, $r_4=3$, $b=1$, $q=0.5$, $r_5=1$, $r=0.01$, with the sole modification of the parameter $m_1$. These figures effectively illustrate the existence of the bubbling phenomenon within the system (\ref{Final ode eq}). }}
        \label{bubblingphenomenon}
\end{figure}
\subsection{Presence of the bubbling phenomenon in the system (\ref{Final ode eq})}
Within this section, we will examine the occurrence of the bubbling phenomenon within the system (\ref{Final ode eq}). Occasionally, a bifurcation diagram may exhibit a closed-loop configuration resembling a bubble like structure due to the emergence and vanishing of oscillations originating from two Hopf bifurcation points. This occurrence is commonly known as the bubbling phenomenon. Many researchers have acknowledged the occurrence of bubbling phenomena in predator-prey models\cite{liz,sas,man1,man2}. The system (\ref{Final ode eq}) exhibits the bubbling phenomenon when certain parameter values are fixed and only the parameter $m_1$ is altered. The fixed parameter values are: $r_1=2$, $r_3 =1$, $\beta = 0.01$, $r_2 = 1$, $m_2 = 0.5$, $d_1 = 0.25$, $d_2 = 0.5$, $r_4 =3$, $b = 1$, $q=0.5$, $r_5 = 1$, $r = 0.01$. The emergence of a bubble of oscillations between two hopf bifurcation points, $m_1=0.4498=m_1^{hb1}$ and $m_1=0.5295=m_1^{hb2}$, respectively, is properly illustrated in figures (\ref{m1equilibrium curve}) and (\ref{bubblingphenomenon}). Following the initial Hopf bifurcation at $m_1=m_1^{hb1}$, there is a noticeable presence of oscillations in the populations of the three species. As the parameter $m_1$ continues to rise, the amplitude of these oscillations steadily increases. However, after reaching a certain threshold, the amplitude begins to diminish and eventually stabilises at the second Hopf bifurcation at $m_1=m_1^{hb2}$. Thereby, the significance of the parameter $m_1$ within the system (\ref{Final ode eq}) is further emphasised. The figure (\ref{m1equilibrium curve}) clearly illustrates the bubbling phenomenon in the $x_1-m_1$ plane. On the other hand, figures (\ref{x2m1}) and (\ref{x3m1}) illustrate the presence of a fluctuation bubble resulting from the appearance of two Hopf bifurcations at $m_1=m_1^{hb1}$ and $m_1=m_1^{hb2}$ on the $x_2-m_1$ and $x_3-m_1$ planes, respectively.

\begin{figure}[H]
     \centering
     \begin{subfigure}{0.45\textwidth}
         \centering
         \includegraphics[width=\textwidth]{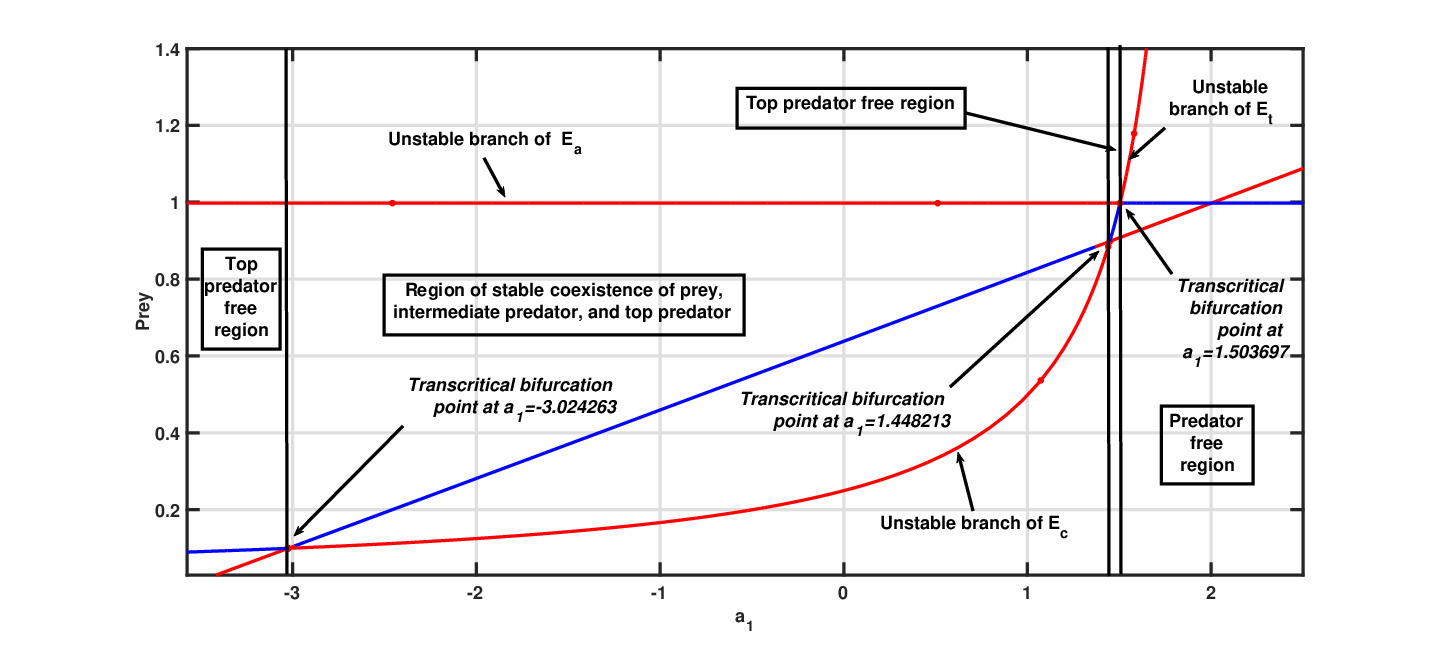}
         \caption{\emph{ This diagram illustrates equilibrium curve of $E_{c}$. Three transcritical bifurcation events are observed at $a_1=-3.024$, $a_1=1.448$, and $a_1=1.503$.}}
         \label{a1 equi}
     \end{subfigure}
      \hfill
     \begin{subfigure}{0.45\textwidth}
         \centering
         \includegraphics[width=\textwidth]{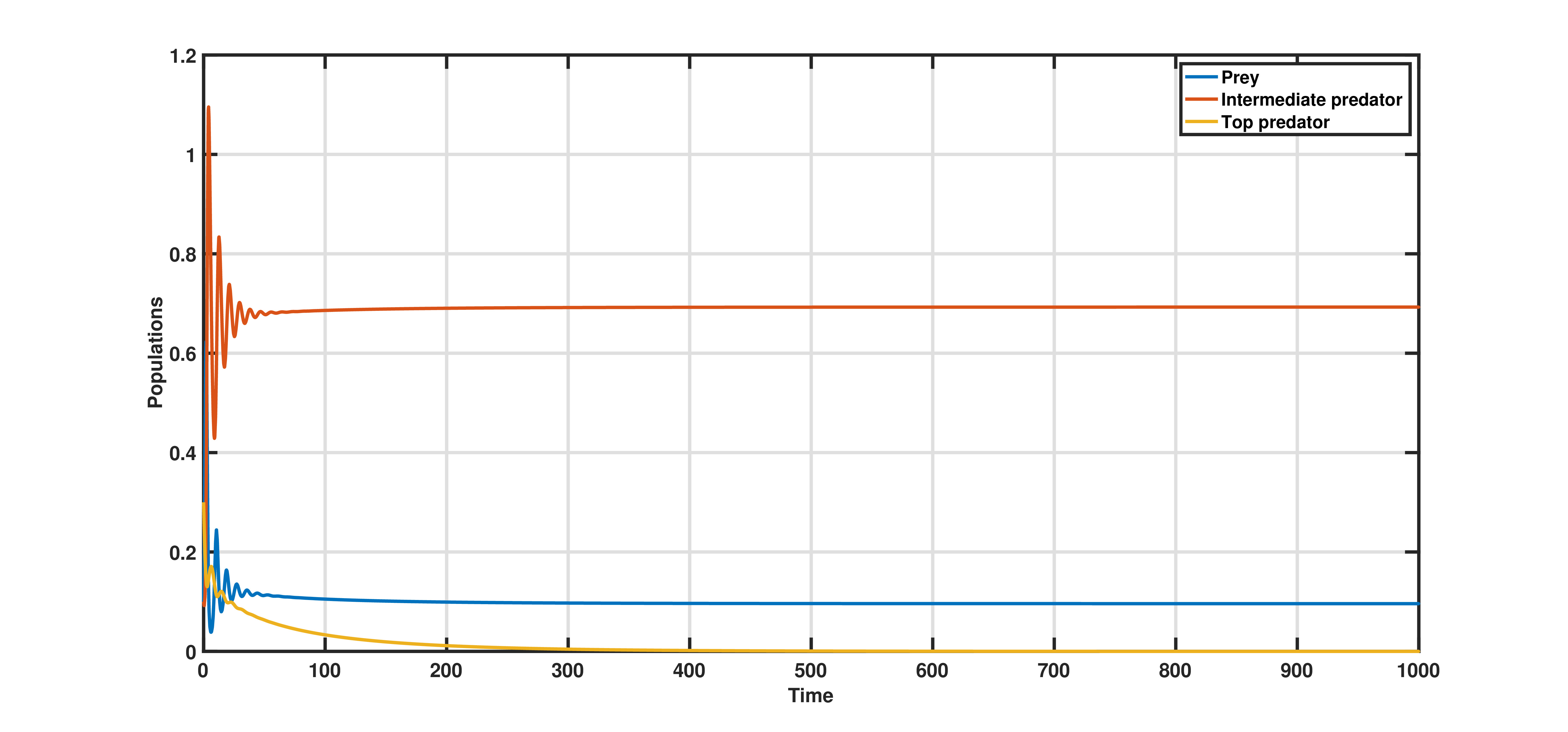}
         \caption{\emph{ The illustration depicts the scenario at $a_1=-3.2$ where the top predator vanishes from the system. This scenario is not ecologically viable as the parameter $a_1$ always has a positive value.}}
         \label{a1 3.2}
     \end{subfigure}
     \hfill
     \begin{subfigure}{0.45\textwidth}
         \centering
         \includegraphics[width=\textwidth]{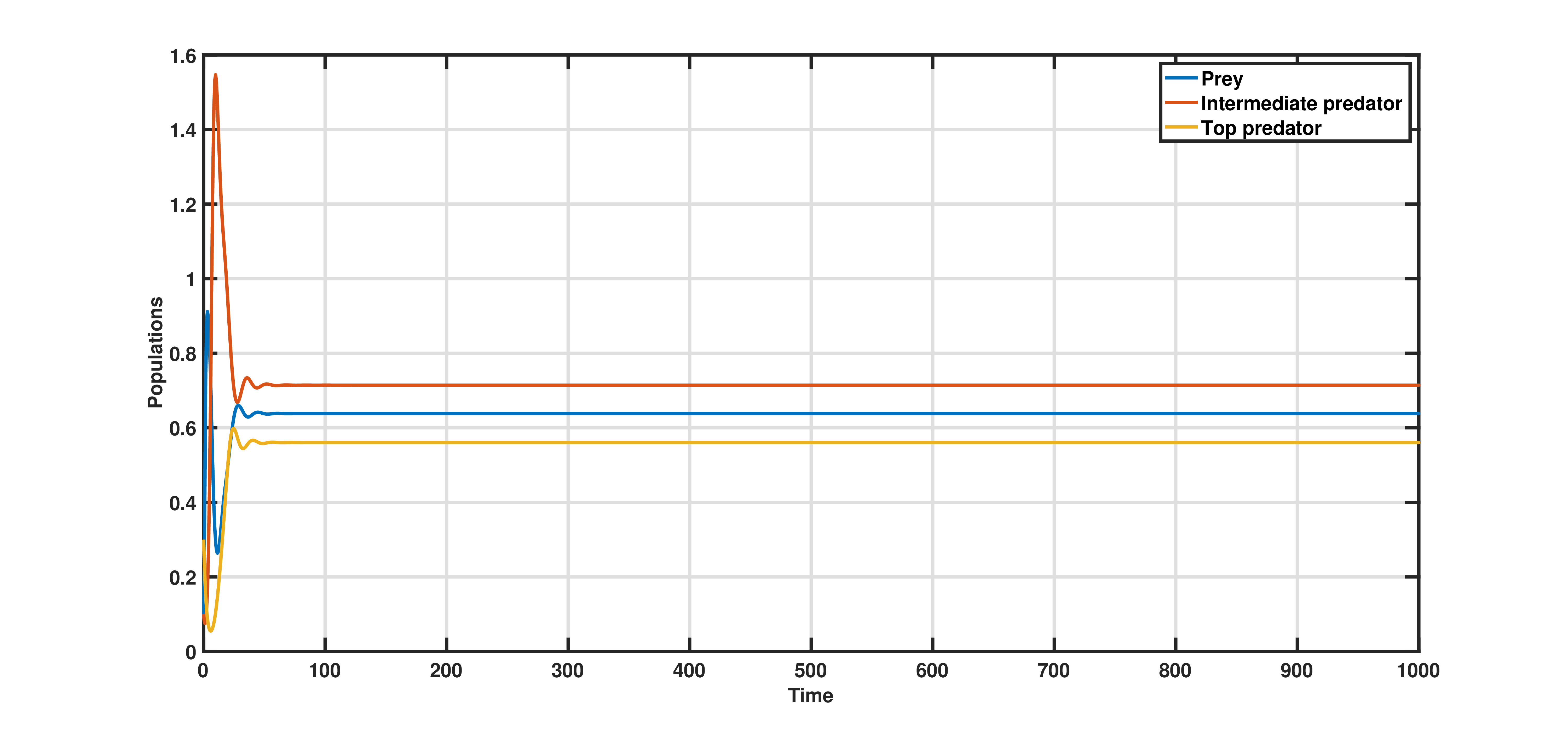}
         \caption{\emph{ This picture depicts the scenario at $a_1=0$, in which all populations coexist within the system.}}
         \label{a1 0}
     \end{subfigure}
     \hfill
     \begin{subfigure}{0.4 \textwidth}
         \centering
         \includegraphics[width=\textwidth]{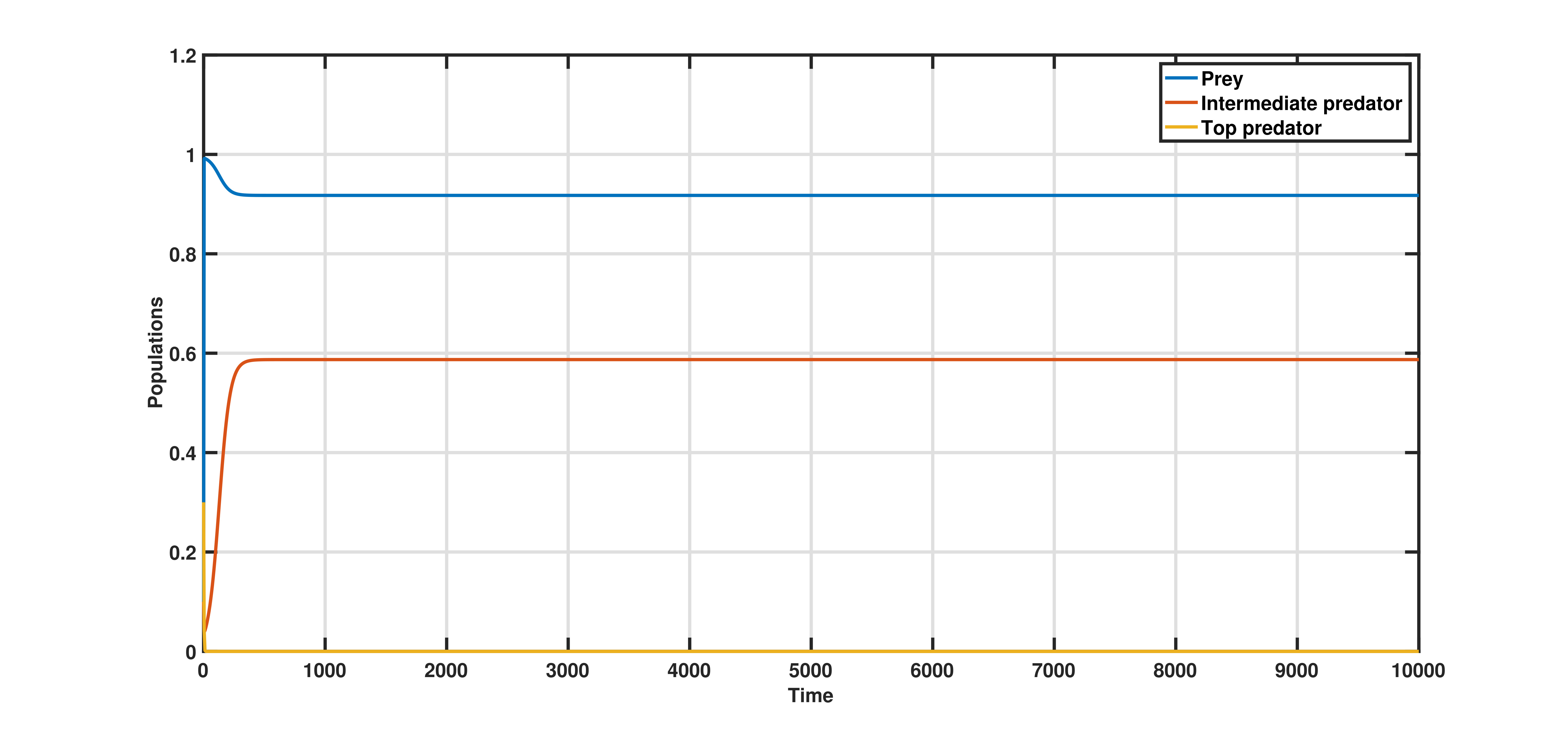}
         \caption{\emph{This picture illustrates the scenario at $a_1=1.46$, where the top predator becomes extinct within the ecosystem.}}
         \label{a1 1.46}
     \end{subfigure}
     \hfill
     \begin{subfigure}{0.4 \textwidth}
         \centering
         \includegraphics[width=\textwidth]{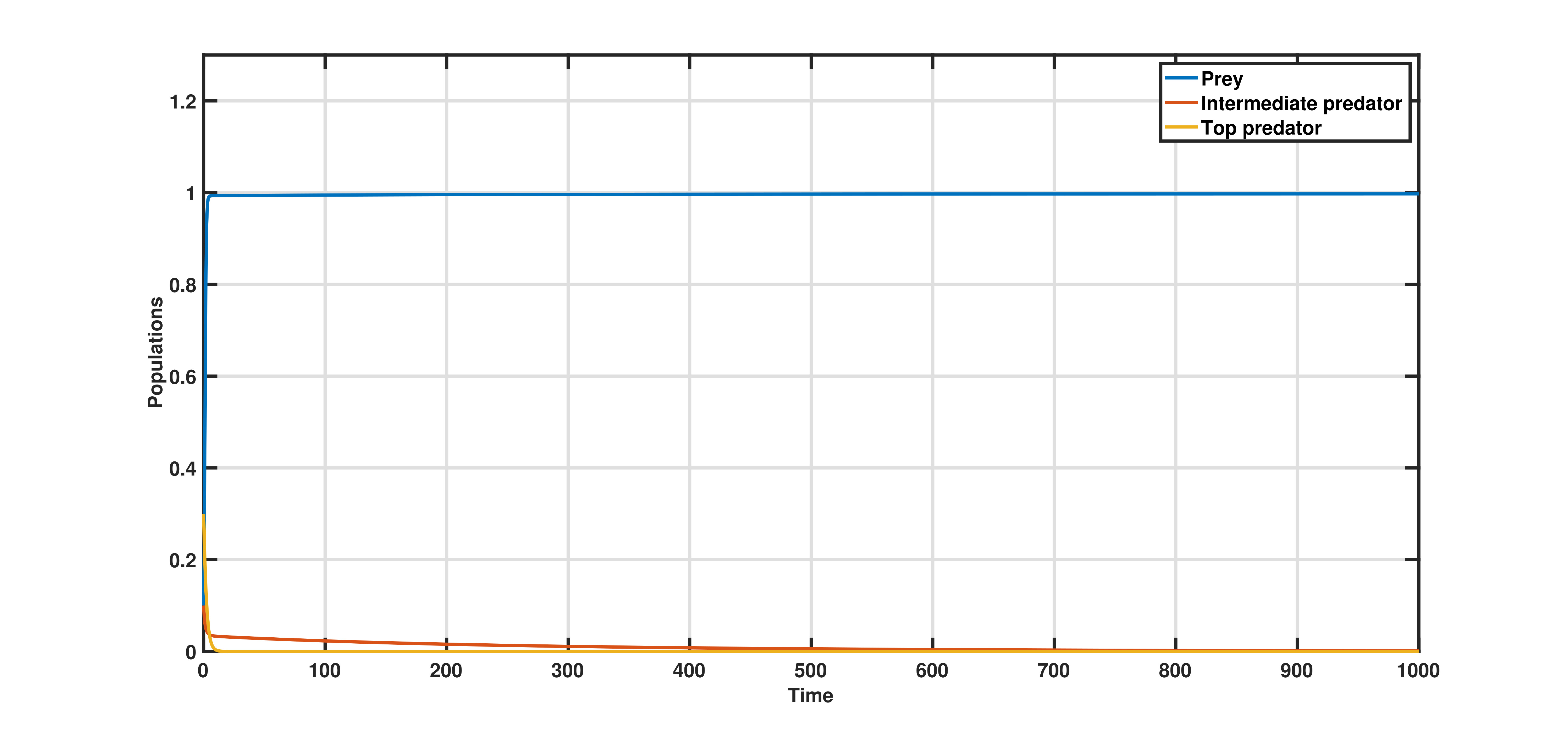}
         \caption{\emph{This figure demonstrates the scenario at $a_1=1.51$, where both predator becomes extinct within the ecosystem. }}
         \label{a1 1.51}
     \end{subfigure}
       \caption{\emph{These figures display the various scenarios that occur in the system (\ref{Final ode eq}) when the parameter $a_1$ is changed, while keeping the other parameter values constant at $r_1=2$, $r_5 = 1$, $\beta = 0.01$, $m_{fp}=0.5$, $m_{sp} = 0.6$, $r_2 = 1$, $d_1 = 0.25$, $r_3 =1$, $d_2 = 0.5$, $r_4 =3$, $a_2=1$, $b = 1$, $q=0.5$, $r = 0.01$, and $\alpha=1$.}}
        \label{a1effect}
\end{figure}

\subsection{Influence of the predator odour-related parameters $a_1$ and $a_2$ on the system (\ref{Frac eq}) dynamics for all $\alpha \in (0,1]$} In the given system (\ref{Firsteq}), the parameters associated with the odour of the intermediate predator and top predator are denoted as $a_1$ and $a_2$, respectively. Initially, we examine the influence of the parameter $a_1$ on the system (\ref{Frac eq}) for $\alpha=1$. With the following parameter values set, the figure (\ref{a1effect}) is produced: $r_1=2$, $r_5 = 1$, $\beta = 0.01$, $m_{fp}=0.5$, $m_{sp} = 0.6$, $r_2 = 1$, $d_1 = 0.25$, $r_3 =1$, $d_2 = 0.5$, $r_4 =3$, $a_2=1$, $b = 1$, $q=0.5$, $r = 0.01$, and continuously varying the parameter $a_1$. Based on the figure (\ref{a1effect}), it is evident that the parameter $a_1$ has a noticeable impact on the system (\ref{Final ode eq}). If the parameter $a_1<-3.024$, the top predator within the system is at risk of extinction. This is shown in figures (\ref{a1 equi}) and (\ref{a1 3.2}). This is because a transcritical bifurcation occurs at $a_1=-3.024$. In ecological circumstance, this is not possible as the parameter $a_1$ is always positive. However, when $-3.024<a_1<1.448$, then it is observed that the prey, intermediate predator and the top predator, all coexist within the system (\ref{Final ode eq}). This is evident in figure (\ref{a1 0}). Moreover, there is another transcritical bifurcation that occurs at $a_1=1.448$. As a result, the system once again faces the extinction of the top predator. This is illustrated in figure (\ref{a1 1.46}). On the other hand, there is another transcritical bifurcation that takes place at $a_1=1.503$. This bifurcation results in the removal of predators, leaving only the prey species in the system, as seen in figure (\ref{a1 1.51}). Thus, the parameter $a_1$ is capable of causing transcritical bifurcations in the system (\ref{Final ode eq}). Consequently, the coexistence of all species in the system is highly dependent on the value of this parameter. In addition, for $0<\alpha=0.98<1$, the system (\ref{Frac eq}) exhibits properties identical to those described above for $\alpha=1$.

\begin{figure}[H]
     \centering
     \begin{subfigure}{0.45\textwidth}
         \centering
         \includegraphics[width=\textwidth]{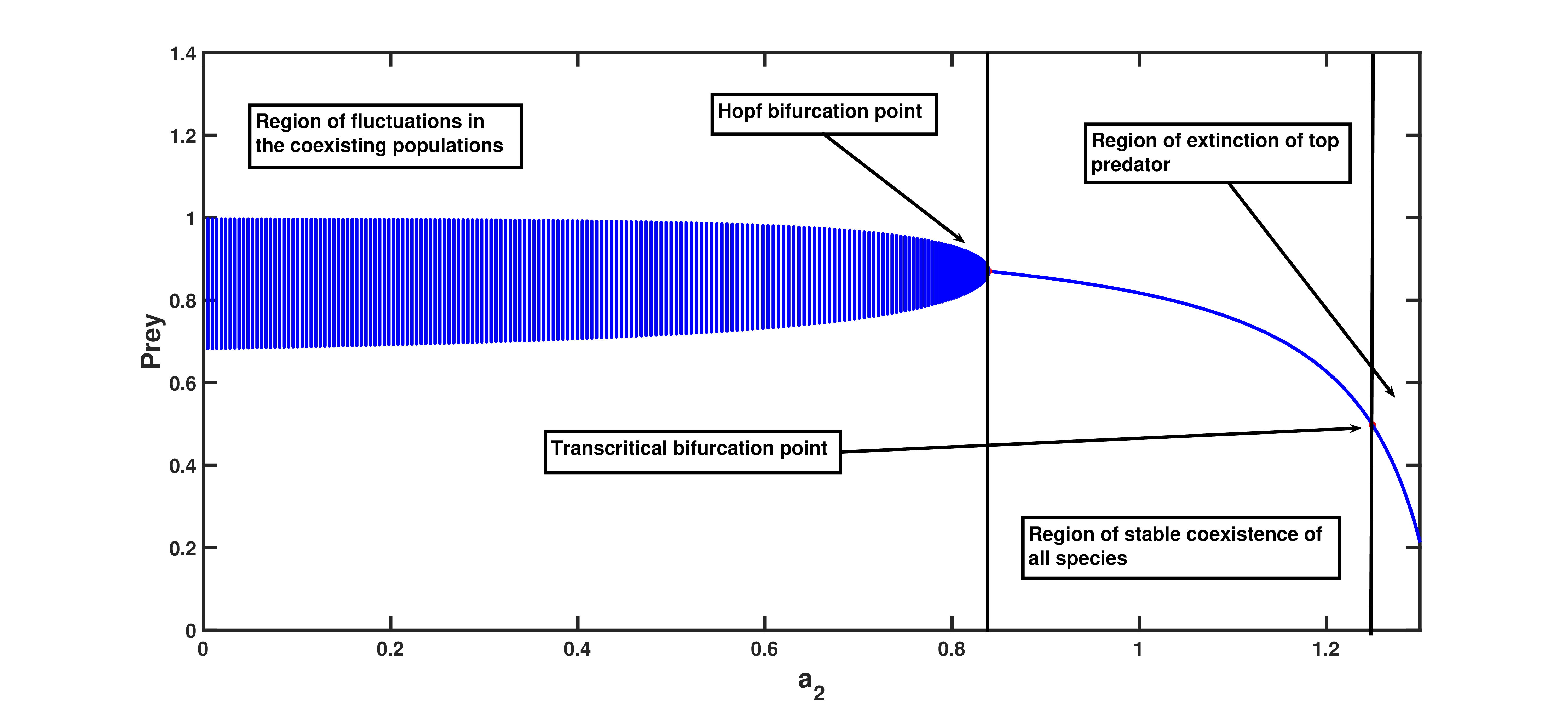}
         \caption{\emph{ This diagram shows the equilibrium curve of $E_{c}$. It exhibits the occurrence of a Hopf bifurcation and a transcritical bifurcation in the system.}}
         \label{a2 equi}
     \end{subfigure}
      \hfill
     \begin{subfigure}{0.45\textwidth}
         \centering
         \includegraphics[width=\textwidth]{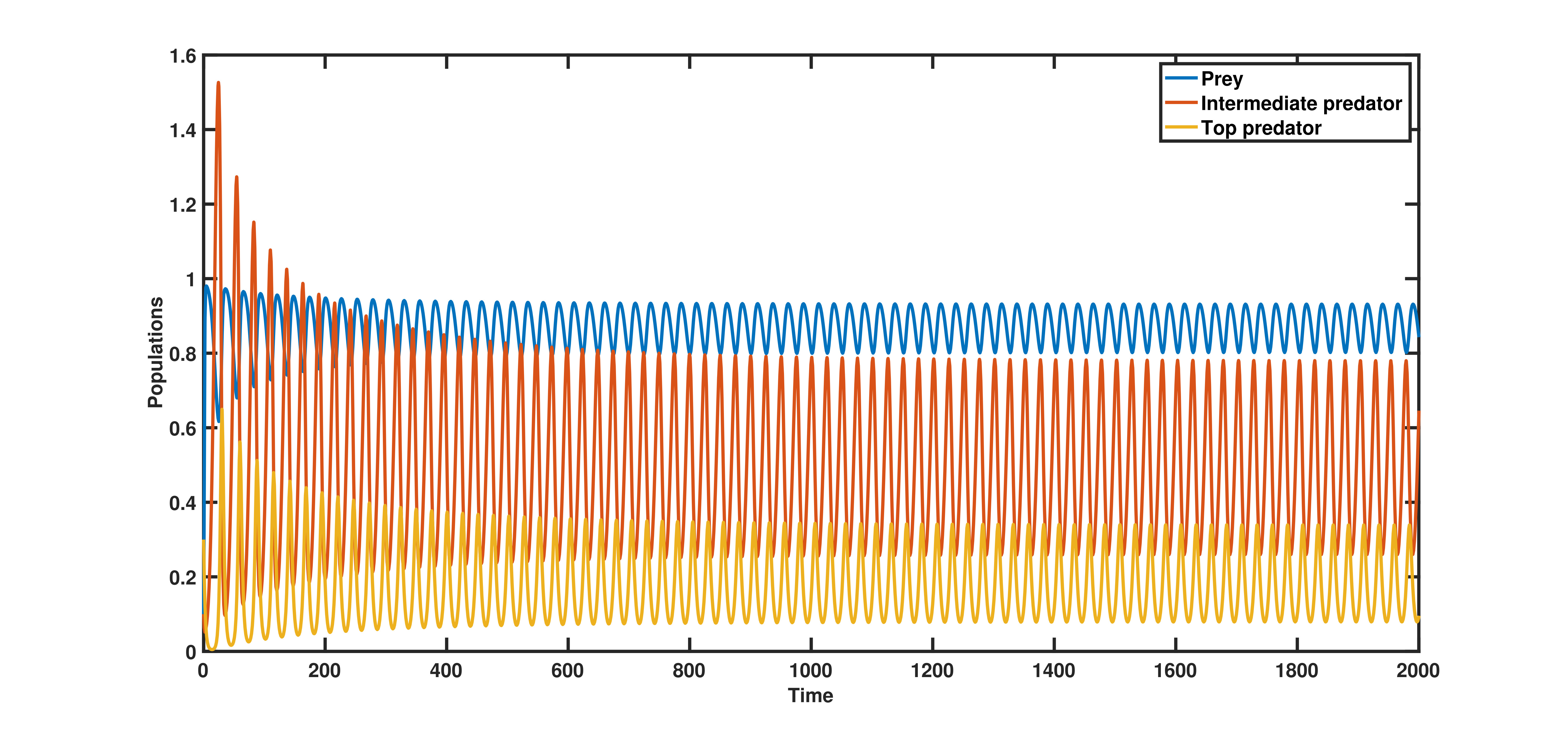}
         \caption{\emph{ This figure displays the time series of all three populations at $a_2=0.8$. It portrays the situation where fluctuations in the populations are evident.}}
         \label{a2 0.8}
     \end{subfigure}
     \hfill
     \begin{subfigure}{0.45\textwidth}
         \centering
         \includegraphics[width=\textwidth]{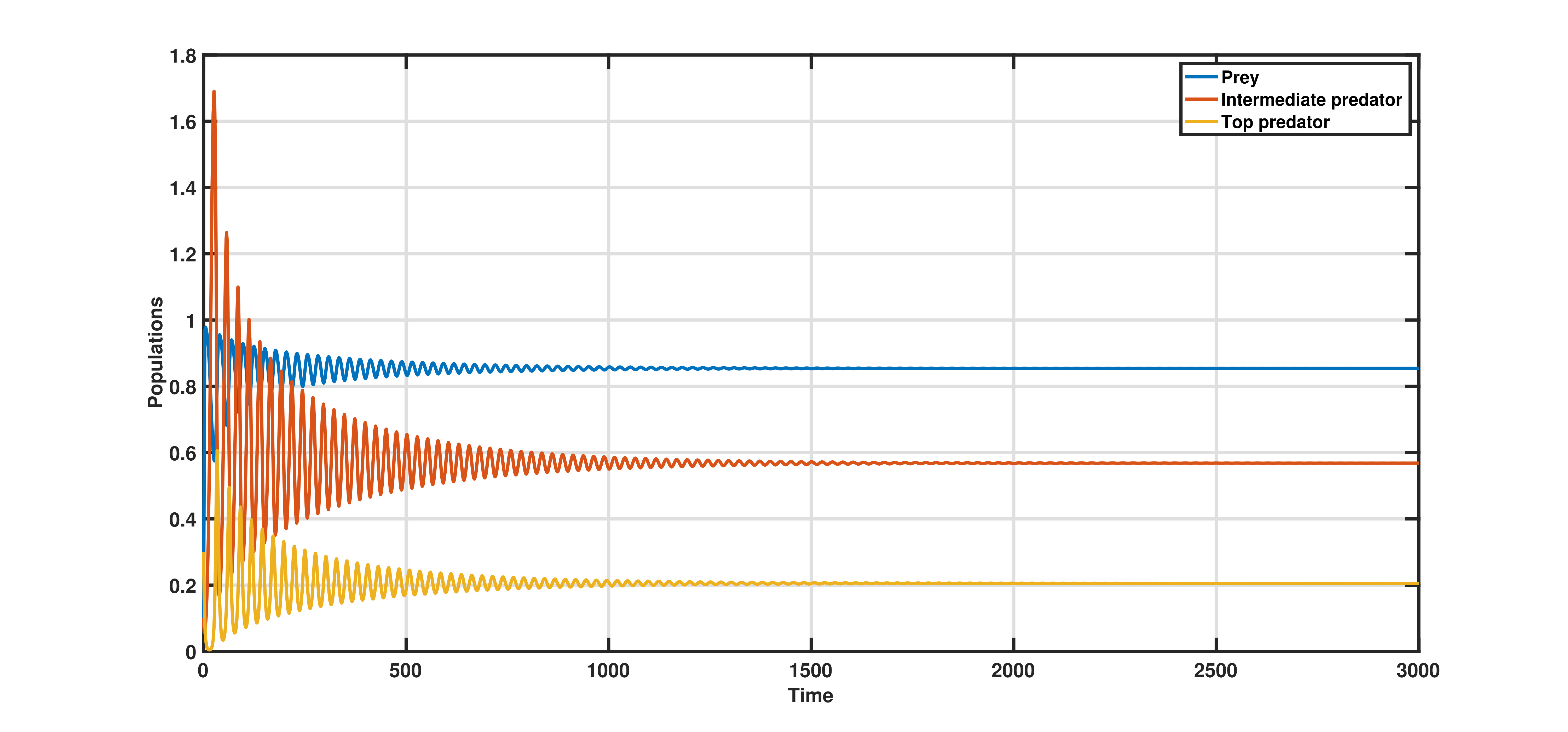}
         \caption{\emph{ This image illustrates the situation at $a_2=0.9$, where all populations reside together within the system.}}
         \label{a2 0.9}
     \end{subfigure}
     \hfill
     \begin{subfigure}{0.4 \textwidth}
         \centering
         \includegraphics[width=\textwidth]{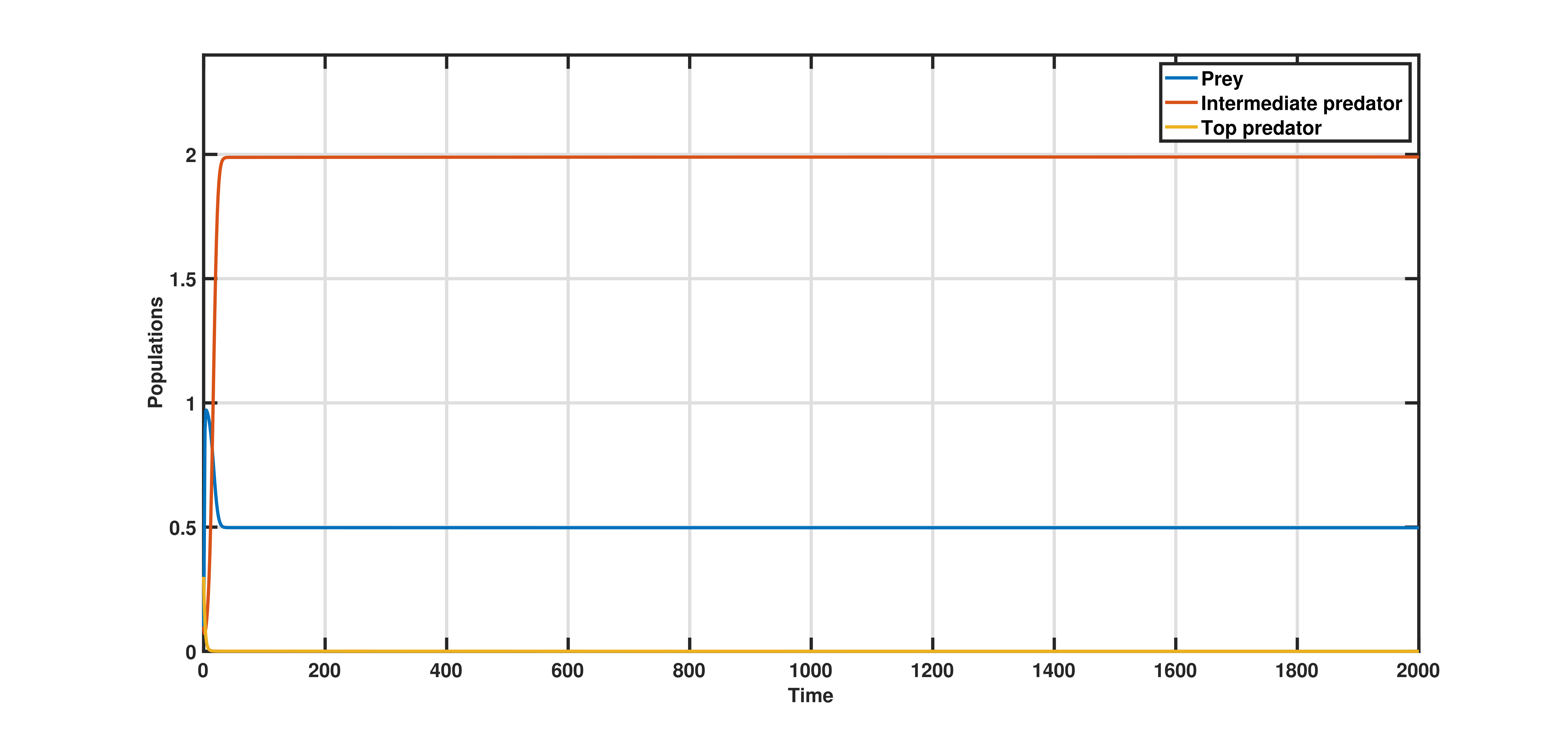}
         \caption{\emph{This graphic depicts the situation  at $a_2=1.25$, where tthe top predator goes extinct inside the system.}}
         \label{a2 1.25}
     \end{subfigure}
       \caption{\emph{The figures illustrate the different scenarios that arise in the system (\ref{Final ode eq}) when the parameter $a_2$ is modified, while keeping the other parameter values constant at $r_1=2$, $r_5 = 1$, $\beta = 0.01$, $m_{fp} = 0.5$, $m_{sp} = 0.6$, $r_2 = 1$, $d_1 = 0.25$, $r_3 =1$, $d_2 = 0.5$, $r_4 =3$, $a_1=1$, $b = 1$, $q=0.5$, $r = 0.01$, and $\alpha=1$.}}
        \label{a2effect}
\end{figure}

We now examine the effect of the parameter $a_2$ in the system (\ref{Frac eq}) with $\alpha=1$, i.e., the system (\ref{Final ode eq}). Figure (\ref{a2 equi}) represents the coexistence equilibrium curve which is generated using some fixed parameter values and constantly changing the value of the parameter $a_2$. The fixed parameter values are: $r_1=2$, $r_5 = 1$, $\beta = 0.01$, $m_{fp} = 0.5$, $m_{sp}=0.6$, $r_2 = 1$, $d_1 = 0.25$, $r_3 =1$, $d_2 = 0.5$, $r_4 =3$, $a_1=1$, $b = 1$, $q=0.5$, and $r = 0.01$. Based on the information shown in figure (\ref{a2 equi}), it is clear that the parameter $a_2$ plays a significant role in the system (\ref{Final ode eq}). It has the ability to cause a Hopf bifurcation and a transcritical bifurcation within the system. When $a_2<0.839221$1, all three populations within the system experiences fluctuations that are later stabilised by a hopf bifurcation occurring at $a_2=0.839221$. This is illustrated in the figure (\ref{a2 0.8}). Therefore, for a range of values between $0.839221$ and $1.249302$ for $a_2$, the system (\ref{Final ode eq}) exhibits the coexistence of all three species. The evidence is readily apparent in the figure (\ref{a2 0.9}). Nevertheless, when $a_2> 1.249302$, the top predator disappears from the system, rendering coexistence unattainable as a transcritical bifurcation takes place at $a_2= 1.249302$. This is depicted in the figure (\ref{a2 1.25}). Thus, the significance of the parameter $a_2$ in either facilitating or preventing coexistence is well established. 

Additionally, when the parameter $a_2$ is within the interval $[0,0.839221)$, the coexisting equilibrium point exhibits stability in the system (\ref{Frac eq}) for $0<\alpha=0.98<1$.  This sharply contrasts with the fluctuations observed in the system (\ref{Final ode eq}) for the interval $[0,0.839221)$. However, when the parameter $m_2$ is within the range $(0.839221,1.25]$, the system defined in (\ref{Frac eq}) for $0<\alpha=0.98<1$ demonstrates behavior similar to that of the system shown in (\ref{Final ode eq}).



\begin{figure}[H]
     \centering
     \begin{subfigure}{0.45\textwidth}
         \centering
         \includegraphics[width=\textwidth]{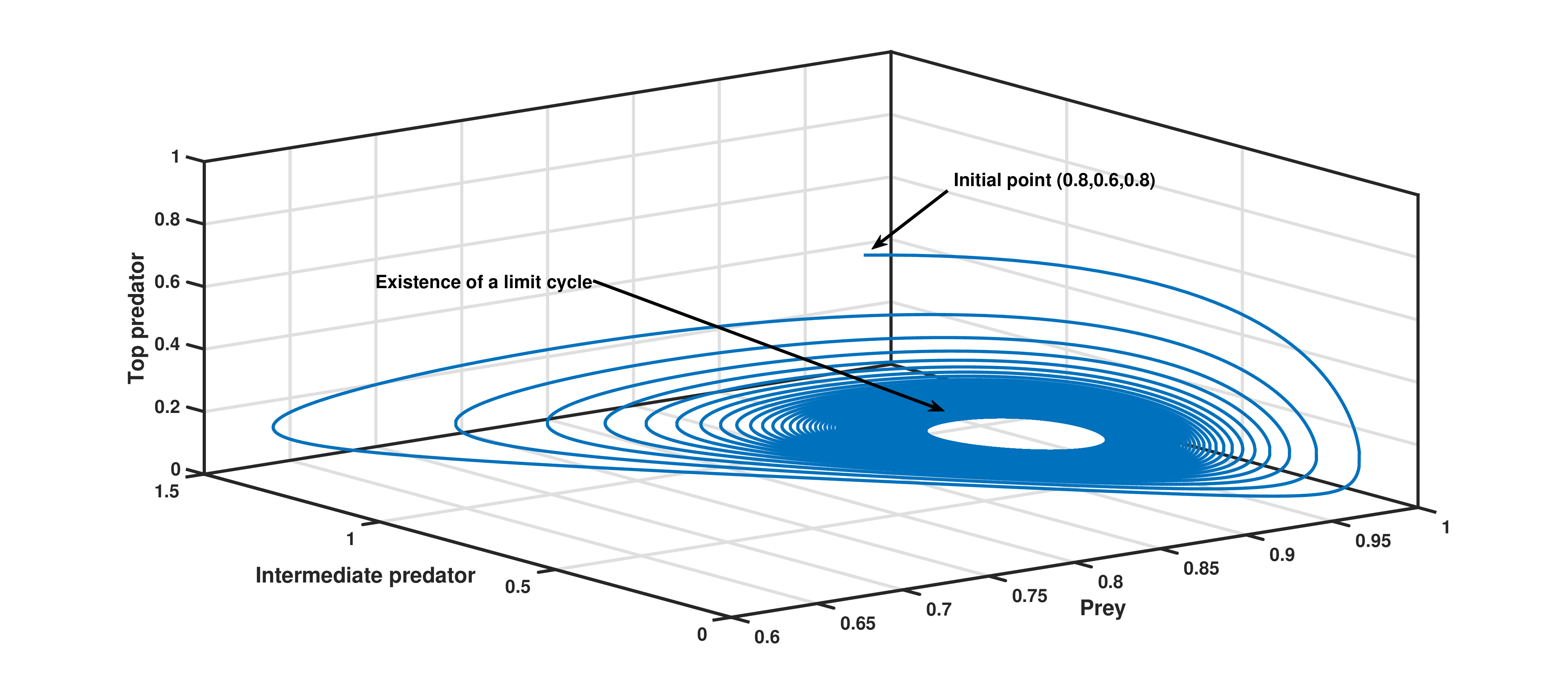}
         \caption{\emph{This diagram illustrates the scenario previous to the occurrence of a Hopf bifurcation at $q=1.04003$. The figure is generated using a value of $q=0.3$, along with the other provided parameter values.}}
         \label{q0.3}
     \end{subfigure}
      \hfill
     \begin{subfigure}{0.45\textwidth}
         \centering
         \includegraphics[width=\textwidth]{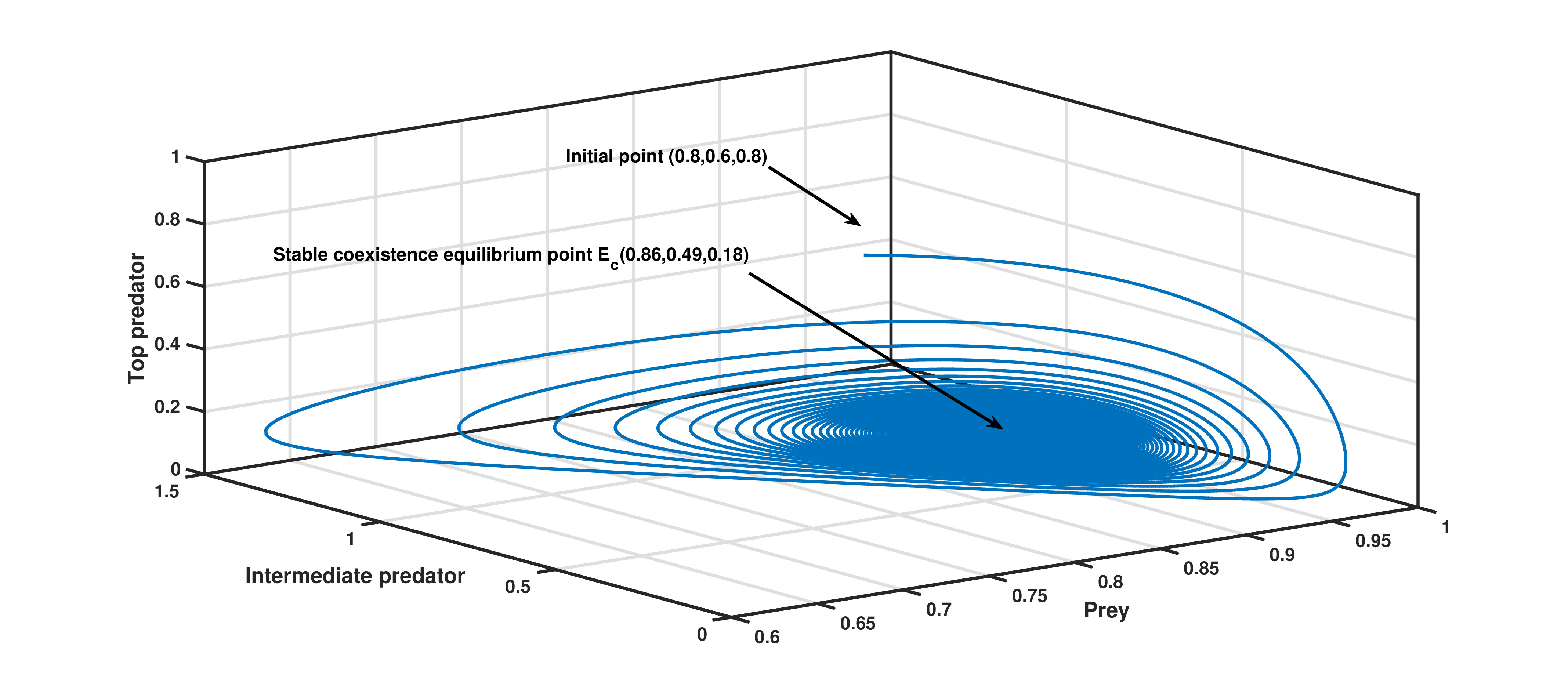}
         \caption{\emph{This diagram depicts the phase that emerges after the occurrence of a Hopf bifurcation at a certain point $q=1.04003$. The figure is generated by using a value of $q=2$ in conjunction with the other specified parameter values.}}
         \label{q2}
     \end{subfigure}
       \caption{\emph{These illustrations depict different scenarios that occur in the system (\ref{Final ode eq}) when the parameter $q$ is changed, while keeping the other parameter values constant. The values of the other parameters are $r_1 = 2$, $r_5 = 1$, $\beta = 0.01$, $r_2 = 1$, $m_1 = 0.5$, $r_4 = 3$, $m_2 = 0.5$, $d_1 = 0.25$, $r_3 = 1$, $d_2 = 0.5$, $b = 1$, $r = 0.01$, and $\alpha=1$.}}
        \label{qdiferent phase portraits}
\end{figure}

\subsection{Effect of harvesting (parameters $q,r$) on the long term dynamics of the system (\ref{Frac eq}) for all $\alpha \in (0,1]$}
Initially, we try to elucidate the significance of the catchability constant $q$ in the system (\ref{Final ode eq}). To be able to accomplish this, we have conducted simulations involving several figures that pertain to the significance of the catchability constant $q$. Figure (\ref{qequilibrium curve}) demonstrates that changes in the parameter $q$ can lead to various forms of bifurcations. Figure (\ref{qdiferent phase portraits}) shows phase portraits at different values of the parameter $q$ while keeping the other parameter values unchanged. The fixed parameter values are: $r_1 = 2$, $r_5 = 1$, $\beta = 0.01$, $m_1 = 0.5$, $r_2 = 1$, $m_2 = 0.5$, $r_4 = 3$, $d_1 = 0.25$, $d_2 = 0.5$, $b = 1$, $r_3 = 1$, and $r = 0.01$. From figure (\ref{qequilibrium curve}), it is apparent that changing the parameter $q$ leads to a Hopf bifurcation, which in turn affects the stability of the system (\ref{Final ode eq}). 

\begin{figure}[H]
     \centering
     \begin{subfigure}{0.45\textwidth}
         \centering
         \includegraphics[width=\textwidth]{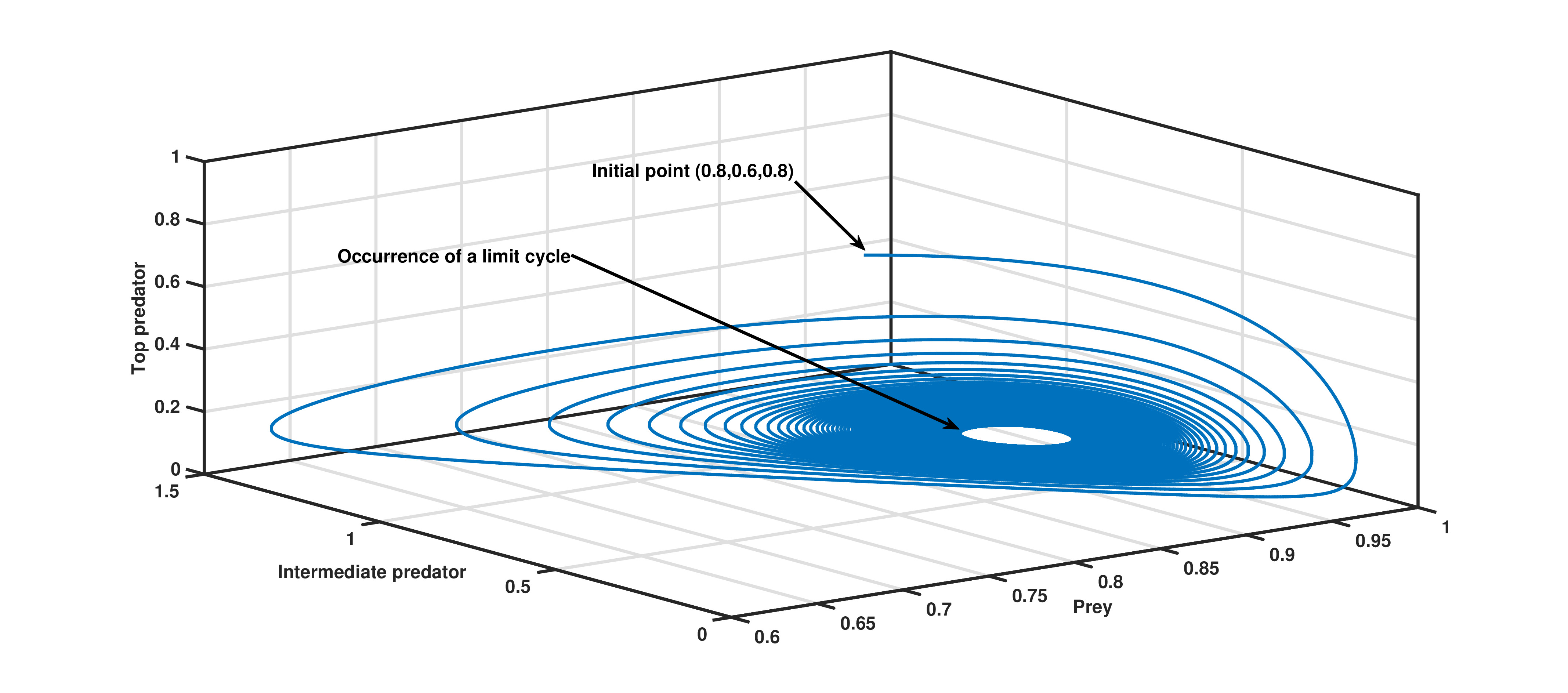}
         \caption{\emph{This diagram illustrates the phase that takes place ahead of the onset of a Hopf bifurcation at a specific value $r=0.020800$. The figure is generated by using a value of $r=0.015$ in conjunction with the other specified parameter values.}}
         \label{r 0.015}
     \end{subfigure}
      \hfill
     \begin{subfigure}{0.45\textwidth}
         \centering
         \includegraphics[width=\textwidth]{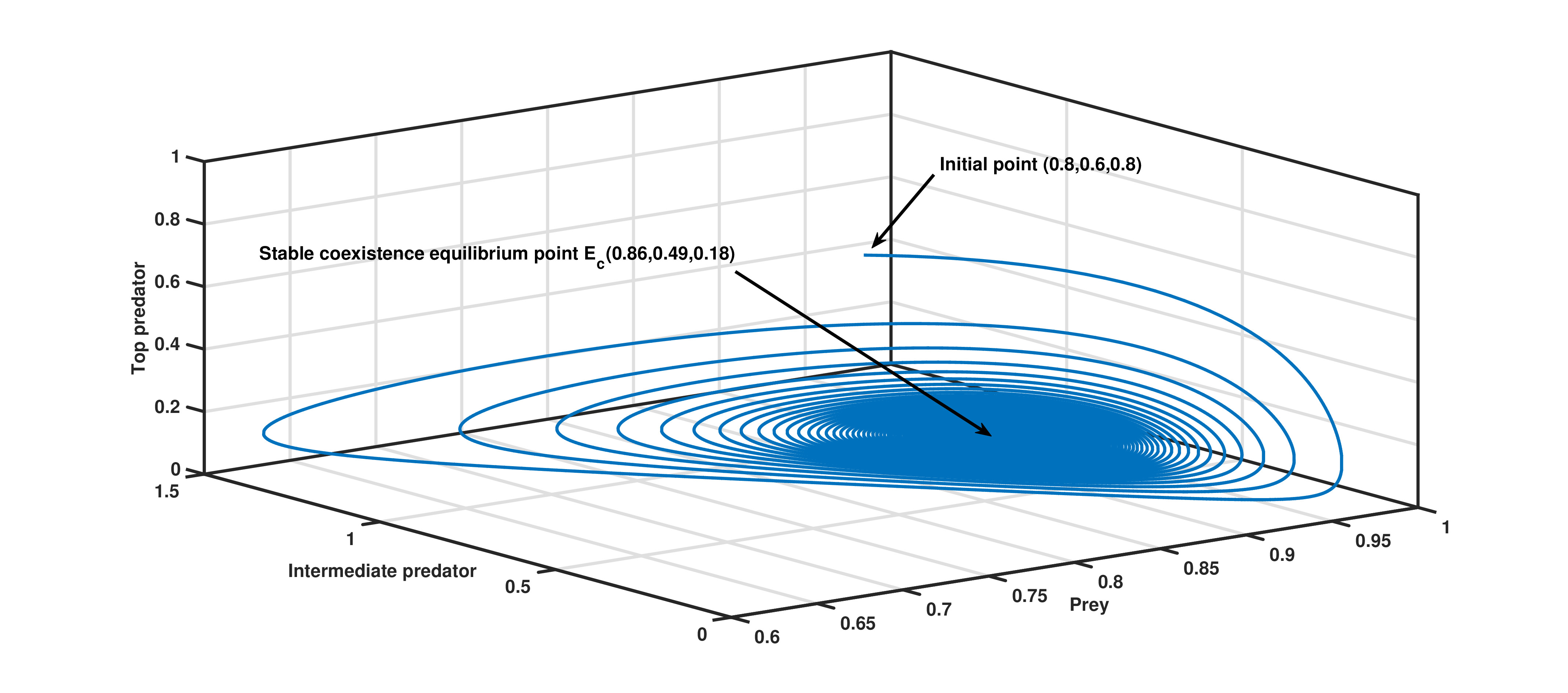}
         \caption{\emph{This diagram illustrates the situation that arises following the occurrence of a Hopf bifurcation at a specific point $r=0.020800$. The figure is generated by using a value of $r=0.04$ along with the other specified parameter values.}}
         \label{r 0.04}
     \end{subfigure}
     \hfill
     \begin{subfigure}{0.45\textwidth}
         \centering
         \includegraphics[width=\textwidth]{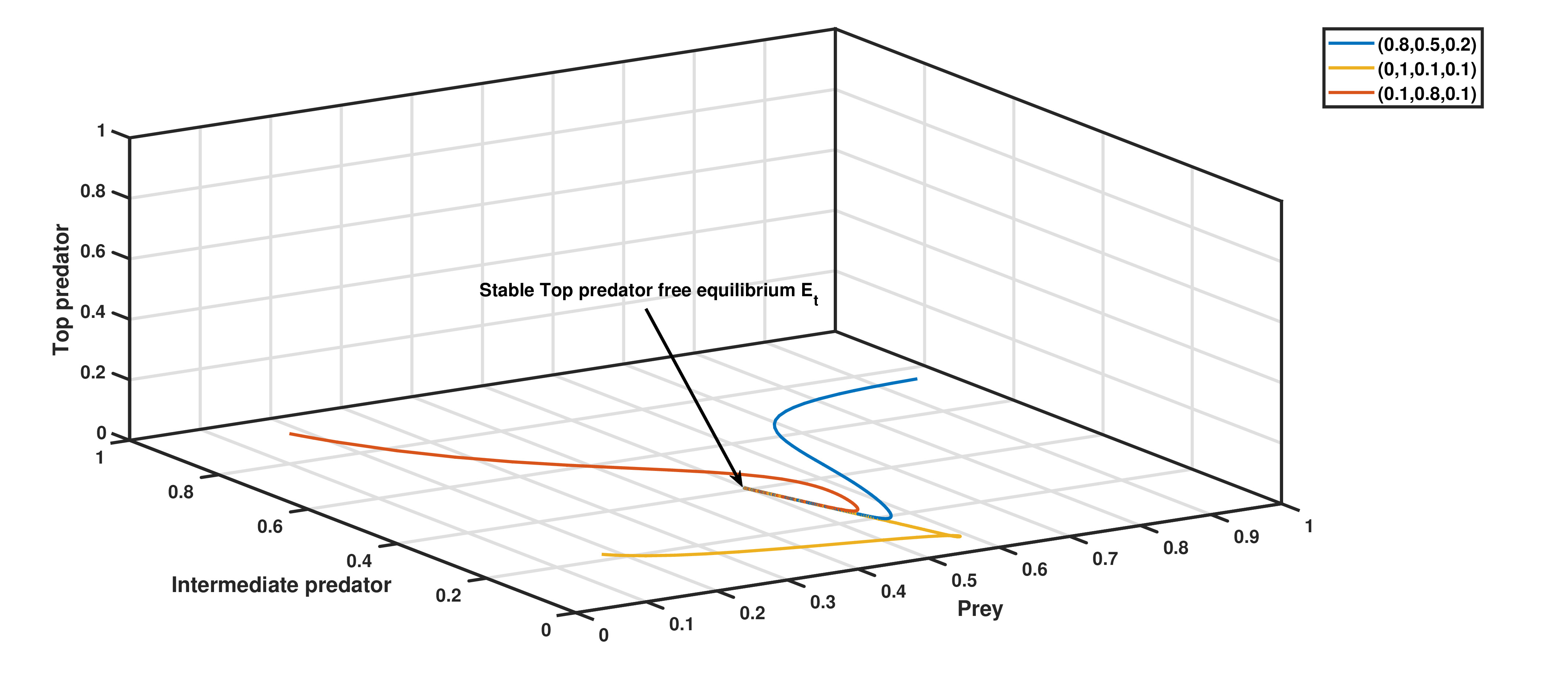}
         \caption{\emph{This diagram depicts the scenario that emerges after a transcritical bifurcation occurs at a particular point $r=1.5074136$. The figure is generated by using a value of $r=1.6$ in conjunction with the other specified parameter values.}}
         \label{r 1.6}
     \end{subfigure}
       \caption{\emph{These illustrations portray various scenarios that arise in the system (\ref{Final ode eq}) when altering the parameter $r$, while maintaining the other parameter values constant. The values of the other parameters are as follows: $r_1 = 2$, $r_5 = 1$, $\beta = 0.01$, $r_2 = 1$, $m_1 = 0.5$, $r_3 = 1$, $m_2 = 0.5$, $d_1 = 0.25$, $d_2 = 0.5$, $r_4 = 3$, $b = 1$, $q =0.5$, and $\alpha=1$.}}
        \label{rdiferent phase portraits}
\end{figure}
For the given parameter values, namely $r_1 = 2$, $r_5 = 1$, $\beta = 0.01$, $m_1 = 0.5$, $m_2 = 0.5$, $r_2 = 1$, $d_1 = 0.25$, $r_4 = 3$, $d_2 = 0.5$, $b = 1$, $q =0.3$, $r_3 = 1$, and $r = 0.01$, when the value of $q$ is less than a certain threshold $q=1.04003=q^{hbp}$, the system described by equation (\ref{Final ode eq}) exhibits instability around the point $E_c$. We obtain the values: $N_1=1.68>0$, $N_2=0.064>0$, $N_3=0.11>0$, and $N_1 N_2-N_3=-0.0018<0$, confirming the instability of $E_c$. Figure (\ref{q0.3}) depicts the precise scenario. More precisely, the system (\ref{Final ode eq}) exhibits fluctuations in the populations of the three species. As a result of a Hopf bifurcation at $q=1.04003=q^{hbp}$, the stability of the system (\ref{Final ode eq}) changes and becomes stable for all parameter values $q>q^{hbp}$, while keeping all other parameter values unchanged. 
It can be seen that the system (\ref{Final ode eq}) exhibits stable characteristics for the parameter values: $r_1 = 2$, $r_5 = 1$, $\beta = 0.01$, $m_1 = 0.5$, $m_2 = 0.5$, $d_1 = 0.25$, $r_2 = 1$, $d_2 = 0.5$, $r_3 = 1$, $b = 1$, $q =2$, $r_4 = 3$, and $r = 0.01$, as depicted in figure (\ref{q2}). Based on the given parameter values, it can be observed that $N_1=1.66>0$, $N_2=0.06>0$, $N_3=0.1>0$, and $N_1 N_2-N_3=0.002>0$. These results provide evidence for the stability of $E_c$. This discussion provides us with compelling evidence that highlights the significance of the parameter $q$ throughout the system (\ref{Final ode eq}). From the preceding discussions, it is evident that when the catchability constant $q$ is extremely low, meaning that the likelihood of capturing prey through harvesting efforts is small, the stable coexistence of all three populations becomes challenging. This is due to fluctuations in the populations of all three species within the system (\ref{Final ode eq}). This implies that coexistence is feasible in this scenario, albeit the population continues to fluctuate. As the likelihood of capturing prey, represented by the catchability constant $q$, grows, a stable coexistence of all three species becomes possible due to the emergence of a Hopf bifurcation. This indicates that for the three species in the system to coexist, a high catchability constant $q$ is required. Although a higher catchability constant $q$ can lead to a reduction in the prey population over time in the long-term dynamics. Thus, the catchability constant $q$ plays a crucial role within the system.

Moreover, it has been observed that in the system (\ref{Frac eq}) for  $0<\alpha<1$, the stability of the coexistence equilibrium point $E_{c}$ remains unchanged with changes in the parameter $q$. When the value of the parameter $q$ falls in the range $[0,q^{hbp})$, the system demonstrates behavior in contradiction with that of the system represented by equation (\ref{Final ode eq}), whereas for values $q>q^{hbp}$, the system's behavior aligns with that described by equation (\ref{Final ode eq}).

Now, we will shift our attention to the significance of the parameter $r$ in the system (\ref{Final ode eq}). In order to gain a deeper comprehension of the influence of this parameter on the system, we conduct simulations using various values of the parameter $r$. In order to produce the figure ({\ref{requilibrium curve}}), we need to change the parameter $r$ and employ the following parameter values: $r_1 = 2$, $r_5 = 1$, $\beta = 0.01$, $m_1 = 0.5$, $r_2 = 1$, $m_2 = 0.5$, $r_3 = 1$, $d_1 = 0.25$, $d_2 = 0.5$, $b = 1$, $r_4 = 3$, $q =0.5$. This figure elegantly depicts the manifestation of several bifurcations as the parameter $r$ undergoes modifications. In order to obtain figure ({\ref{r 0.015}}), we utilise the following parameter values: $r_1 = 2$, $r_5 = 1$, $\beta = 0.01$, $m_1 = 0.5$, $m_2 = 0.5$, $d_1 = 0.25$, $r_2 = 1$, $r_3 = 1$, $d_2 = 0.5$, $b = 1$, $q =0.5$, $r_4 = 3$, and $r=0.015$. Alternatively, in order to simulate figure ({\ref{r 0.04}}), we employ the following parameter values: $r_1 = 2$, $r_5 = 1$, $\beta = 0.01$, $m_1 = 0.5$, $m_2 = 0.5$, $d_1 = 0.25$, $r_2 = 1$, $r_3 = 1$, $d_2 = 0.5$, $b = 1$, $q =0.5$, $r_4 = 3$, and $r=0.04$. Based on the information provided in figure ({\ref{requilibrium curve}}), it can be inferred that if the parameter $r$ is smaller than $r=0.0208005=r^{hbp}$, the system (\ref{Final ode eq}) exhibits instability in the vicinity of $E_c$. When the parameter values are set to $r=0.015$, while keeping all other parameter values unchanged, we find that $N_1=1.67>0$, $N_2=0.064>0$, $N_3=0.10>0$, and $N_1 N_2-N_3=-0.0007<0$. This confirms the instability of $E_c$. This is accurately illustrated in figure ({\ref{r 0.015}}). A Hopf bifurcation occurs at the value of $r$ equal to $r=0.0208005=r^{hbp}$. This is evident as all the criteria specified in theorem (\ref{odehopfthm}) are met for these parameter values. 
The system (\ref{Final ode eq}) exhibits stability around $E_c$ for all values of the parameter $r^{tbp}>r>r^{hbp}$. It can be established numerically through the use of the following parameter values: $r_1 = 2$, $r_5 = 1$, $\beta = 0.01$, $m_1 = 0.5$, $m_2 = 0.5$, $r_2 = 1$, $r_3 = 1$, $d_1 = 0.25$, $d_2 = 0.5$, $b = 1$, $q = 0.5$, $r_4 = 3$, and $r=0.04$. At these parameter values, $N_1=1.6>0$, $N_2=0.06>0$, $N_3=0.1>0$, and $N_1 N_2-N_3=0.002>0$. This substantiates the stability of $E_c$ in accordance with theorem (\ref{lsco}). This is accurately demonstrated in figure ({\ref{r 0.04}}). A transcritical bifurcation occurs when the value of $r$ reaches $r=1.5074136=r^{tbp}$, causing a shift in the stability of the system (\ref{Final ode eq}) and making it unstable. In order to have a better understanding of this exact situation, we use parameter values: $r_1 = 2$, $r_2 = 1$, $\beta = 0.01$, $m_1 = 0.5$, $m_2 = 0.5$, $d_1 = 0.25$, $r_3 = 1$, $r_4 = 3$, $d_2 = 0.5$, $r_5 = 1$, $b = 1$, $q = 0.5$, and $r=1.6$. At these parameter values, we notice that $N_3$ is equal to -0.01, which is less than 0. This demonstrates the instability of $E_c$ as per theorem (\ref{lsco}). Furthermore, for these specific parameter values, all the conditions outlined in theorem (\ref{lsto}) are met, indicating the stability of $E_t$. This is precisely shown in figure (\ref{r 1.6}). The level of harvesting effort plays a crucial role in predator-prey systems, as it has far-reaching effects on both the target species and the overall ecosystem. Efficiently managing the harvesting effort is crucial for maintaining the population's long-term sustainability. The preceding calculations show that as the amount of harvesting effort increases, the prey population decreases, which is consistent with biological principles. In addition, it has been observed that when the level of harvesting effort is low, the populations of all three species exhibit fluctuations. When the level of harvesting effort is increased, the system becomes stable, allowing for the coexistence of all three populations. This is the result of the manifestation of a Hopf bifurcation. Interestingly, when the level of harvesting effort is further increased, the coexistence of all species becomes unattainable, and the top predator becomes extinct.

Furthermore, when \(0<\alpha<1\) in the system (\ref{Frac eq}), it is seen that for all values of the parameter \(r\) within the interval \([0,r^{hbp})\), the system exhibits stable behavior at the coexistence equilibrium point, in stark contrast to the fluctuating behavior of the system when \(\alpha=1\).  For values \( r > r^{hbp} \), the system's behavior for \( 0 < \alpha < 1 \) corresponds with that delineated by equation (\ref{Final ode eq}).

\subsection{Effect of prey odour (parameter $\beta$) on the populations of the system (\ref{Frac eq}) for all $\alpha \in (0,1]$}
Within this part, we explore the impact of prey odour on the system (\ref{Frac eq}). The parameter $\beta$ in this system reflects the coefficient that quantifies the impact of prey odour. Figure (\ref{betaequilibrium curve}) vividly demonstrates the significance of the parameter $\beta$ within the system (\ref{Frac eq}) for $\alpha=1$, i.e., system (\ref{Final ode eq}). Figure (\ref{betaequilibrium curve}) is obtained using the following parameter values: $r_1 = 2$, $r_5 = 1$, $\beta = 0.5$, $r_2 = 1$, $r_3 = 1$, $m_1 = 0.5$, $m_2 = 0.5$, $d_1 = 0.25$, $r_4 = 3$, $d_2 = 0.5$, $b = 1$, $q = 0.5$, and $r=0.01$. Based on figure (\ref{betaequilibrium curve}), it is apparent that the parameter $\beta$ significantly contributes to the stability of the system (\ref{Final ode eq}). Fluctuations in the populations of the three species within the system can be detected when the parameter value $\beta<0.020337=\beta^{hbp}$. As an example, we will examine the parameter values: $r_1 = 2$, $r_5 = 1$, $\beta = 0.015<\beta^{hbp}$, $m_1 = 0.5$, $m_2 = 0.5$, $d_1 = 0.25$, $d_2 = 0.5$, $r_2 = 1$, $r_3 = 1$, $b = 1$, $q = 0.5$, $r_4 = 3$, and $r=0.01$. Given this parameter configuration, we get the values $N_1=1.6>0$, $N_2=0.06>0$, $N_3=0.1>0$, and $N_1 N_2-N_3=-0.0007<0$. Therefore, this verifies the instability of $E_{c}$. This is readily apparent in figure (\ref{b2 0.015}). A Hopf bifurcation occurs at $\beta=0.020337=\beta^{hbp}$, leading to the stabilisation of the system.
Given the parameter values: $r_1 = 2$, $r_5 = 1$, $\beta =\beta^{hbp}$, $m_1 = 0.5$, $m_2 = 0.5$, $r_2 = 1$, $r_3 = 1$, $d_1 = 0.25$, $d_2 = 0.5$, $b = 1$, $r_4 = 3$, $q = 0.5$, and $r=0.01$, we find that 
all the requirements outlined in theorem (\ref{odehopfthm}) for the existence of Hopf bifurcation are met, resulting in a Hopf bifurcation at $\beta =\beta^{hbp}$.
\begin{figure}[H]
     \centering
     \begin{subfigure}{0.45\textwidth}
         \centering
         \includegraphics[width=\textwidth]{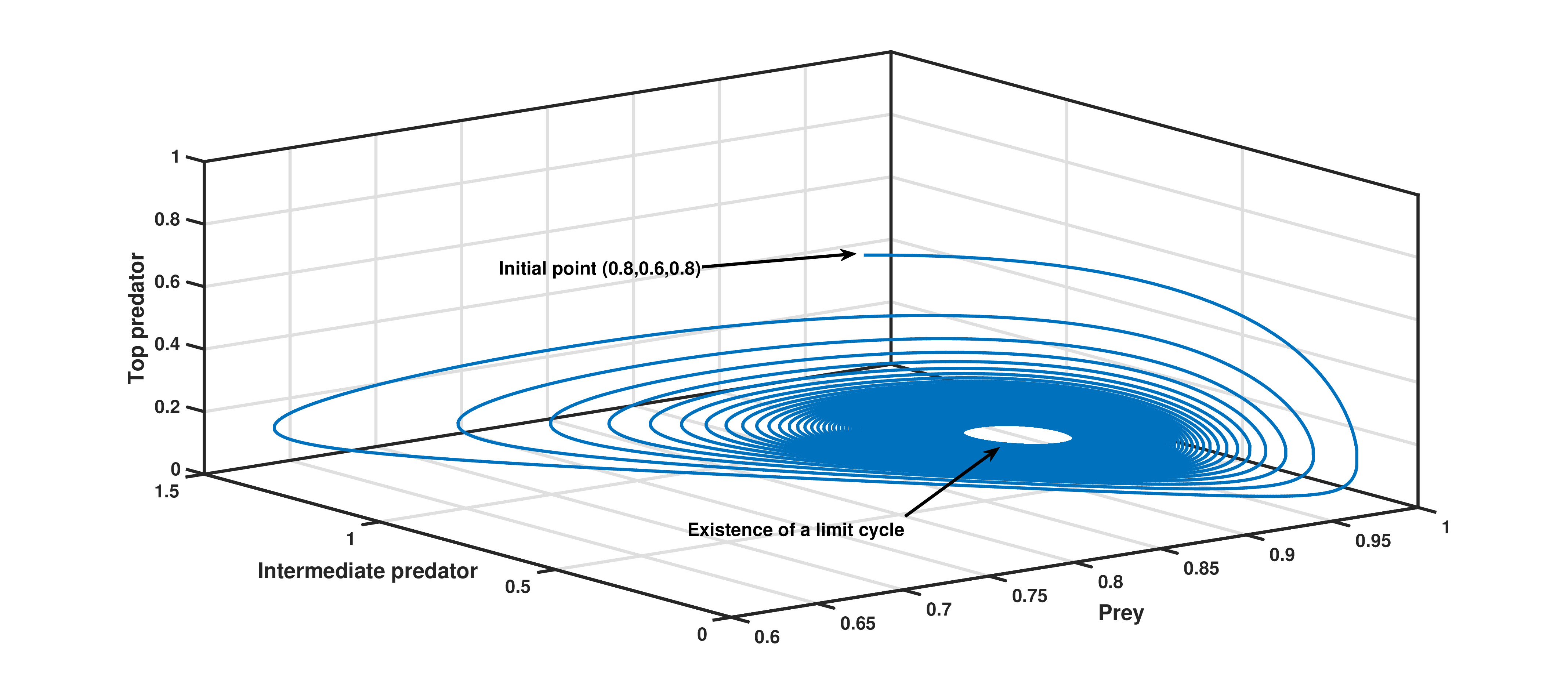}
         \caption{\emph{This figure depicts the phase preceding the commencement of a Hopf bifurcation at a specific value $\beta=0.020337$. The figure is generated using a value of $\beta=0.015$ together with the other stated parameter values}}
         \label{b2 0.015}
     \end{subfigure}
      \hfill
     \begin{subfigure}{0.45\textwidth}
         \centering
         \includegraphics[width=\textwidth]{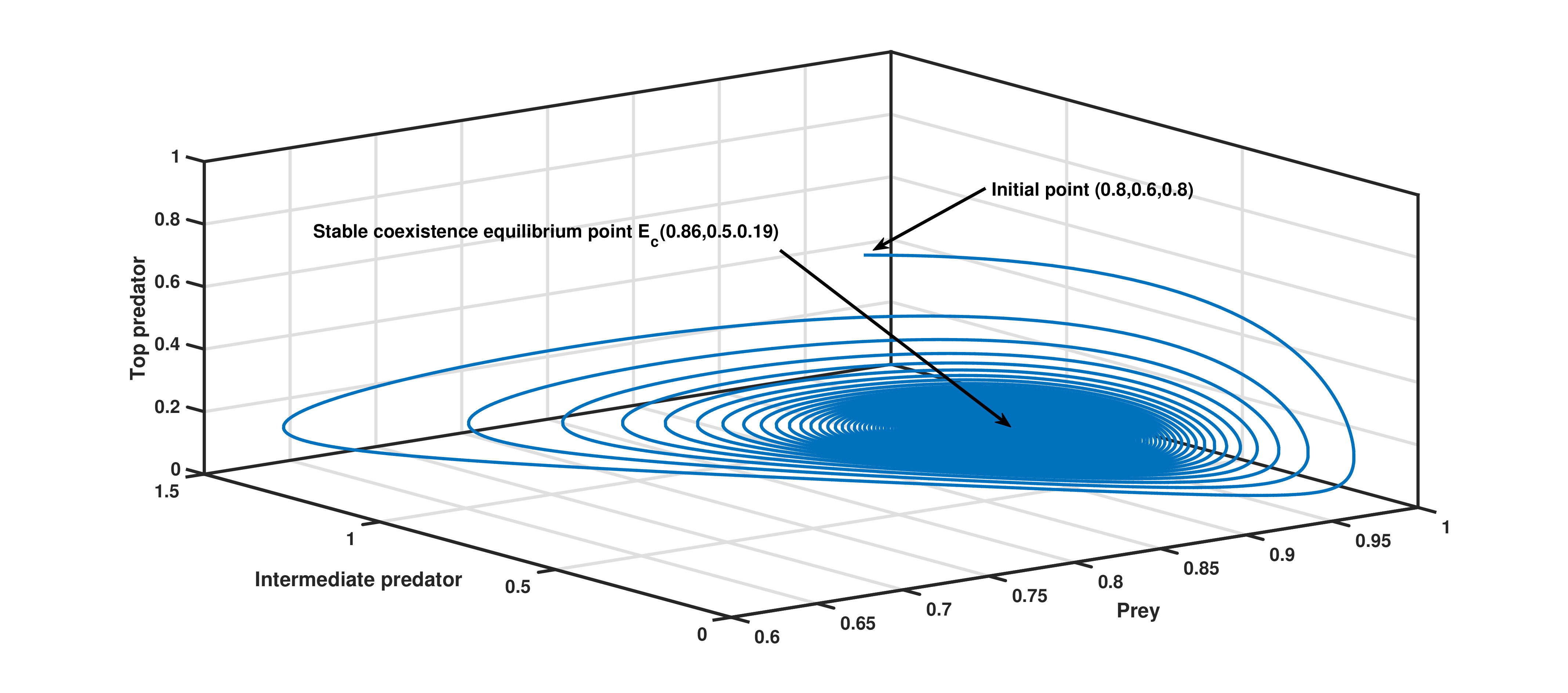}
         \caption{\emph{This diagram depicts the phase that occurs following the inception of a Hopf bifurcation at a precise value of $\beta=0.020337$. The figure is generated by using a value of $\beta=0.04$ along with the other specified parameter values.}}
         \label{b2 0.04}
     \end{subfigure}
       \caption{\emph{These illustrations depict different scenarios that occur while changing the parameter $\beta$, while keeping the other parameter values constant. In the given setting, the values of the other parameters are as follows: $r_1 = 2$, $r_4 = 3$, $m_1 = 0.5$, $m_2 = 0.5$, $d_1 = 0.25$, $r_2 = 1$, $r_3 = 1$, $d_2 = 0.5$, $b = 1$, $q = 0.5$, $r_5 = 1$, $r = 0.015$, and $\alpha=1$.}}
        \label{}
\end{figure}
When the parameter value $\beta$ is set to $\beta=0.04>\beta^{hbp}$, and all other parameter values remaining unchanged, we have the following results: $N_1=1.6$, $N_2=0.07$, $N_3=0.1$, and $N_1 N_2-N_3=0.002$. All of these values are greater than zero. Therefore, based on theorem (\ref{lsco}), the system (\ref{Final ode eq}) displays stability near $E_{c}$ given the specified parameter values. This specific situation is depicted in figure (\ref{b2 0.04}). Based on the previous discussion, it readily apparent that when the coefficient of prey odour effect increases, the intermediate predator becomes more attracted to the prey due to the odour released by the prey. As a result, the population of prey species decreases due to an increase in the encounter rate between the intermediate predator and prey, leading to a higher predation rate. This is precisely what happens within the system (\ref{Final ode eq}) as evident from figure (\ref{betaequilibrium curve}). Interestingly, it has been noted that a minimal amount of prey odour effect renders the stable coexistence of all three species within the system unattainable, resulting in oscillations in the populations of the three species. However, when the intensity of the prey odour is significantly enhanced, a Hopf bifurcation occurs, resulting in the long-term coexistence of all three populations within the system (\ref{Final ode eq}). Therefore, it has been discovered that the influence of prey odour aids in fostering a state of lasting cohabitation within this particular system.

Additionally, when \(0<\alpha<1\) in the system (\ref{Frac eq}), it is observed that for all values of the parameter \(\beta\) within the interval \([0,\beta^{hbp})\), the system demonstrates stable behavior near the coexistence equilibrium point, in sharp contrast to the fluctuating behavior of the system when \(\alpha=1\). For values \( \beta > \beta^{hbp} \), the system's behavior for \( 0 < \alpha < 1 \) resembles with that of the system given by the equation (\ref{Final ode eq}).

\section{Conclusion} \label{con} Our study delves into a model that investigates the intricate dynamics of a odour-mediated three species food chain. This model incorporates a commonly observed phenomenon, notably the influence of odour. It incorporates predator odour-induced prey refuge and prey harvesting. The model equations are formulated in two different types: one employing ordinary differential equations and the other utilising Caputo fractional-order differential equations to incorporate the memory effect. The paper explores the fundamental prerequisites, such as the existence, non-negativity, and uniqueness of the solutions of the system. The model's biologically plausible equilibrium states are established. All ecologically feasible equilibrium points are analysed in terms of their local stability. The article extensively covers the phenomenon of bifurcation, with a particular focus on the occurrence of Hopf bifurcations. An analysis has been conducted to compare the two systems, ODE (\ref{Final ode eq}) and FDE (\ref{Frac eq}). We perform numerical simulations with biologically attainable parameters to illustrate the behaviour of the system close to the equilibrium points. It is found that our system can attain a maximum four equilibrium states: the vanishing equilibrium point $E_{v}$, the axial equilibrium point $E_{a}$, the top predator free equilibrium point $E_{t}$, and the coexisting equilibrium point $E_{c}$. From an ecological standpoint, the vanishing equilibrium point $E_{v}$ represents the state in which all three species in the bio-system become extinct and the system breaks down. The axial equilibrium point $E_{a}$ represents the situation in which the two predators vanish, leaving only the prey species. The top predator free equilibrium point $E_{t}$ indicates the state in which only the top predator becomes extinct. The coexisting equilibrium point $E_{c}$ reflects the scenario in which all three species coexist in the system. Based on our research findings, it can be inferred that if the rate of harvesting exceeds the growth rate of the prey species, the system will ultimately reach the vanishing equilibrium point, resulting in the extinction of all three species and the collapse of the system. The local stability of the vanishing equilibrium point $E_{v}$ precludes the possibility of the existence of the axial equilibrium point $E_{a}$ and the top predator free equilibrium point $E_{t}$. Ecologically, reaching the vanishing equilibrium point refers to the instant at which all species in our system become extinct. At this point, there is no possibility of any growth in the species unless we reintroduce them, which is biologically true. The local stability of both the axial and top predator free equilibrium points negates the local stability of the vanishing equilibrium scores. From a biological perspective, this means that when the bio-system reaches a state where either just the prey survives in the system or the top predator becomes extinct from the system, the system has no way to collapse. In other words, all three species cannot become extinct simultaneously, as evidenced from our model. Upon satisfying certain parametric conditions, our discussed bio-system can achieve the coexistence equilibrium point. From an ecological point of view, it is evident that the coexistence of all three species is feasible within this system. 

Existing literature has already established that when prey species detect the presence of a nearby predator through olfactory cues, they become vigilant and employ various anti-predator strategies to minimise encounters. One effective strategy is for prey species to seek refuge in order to prevent predators from reaching them. In this study, we examine the consequences of refuge behaviour exhibited by prey and intermediate predators in response to their respective predators' odour. The prey's refuge behaviour towards the intermediate predator has been discovered to have a significant impact on the system. This parameter has the potential to trigger two Hopf bifurcations in the system, resulting in oscillations in the populations of the three species and giving birth to the intriguing bubbling phenomenon. Moreover, this parameter has the power to trigger a transcritical bifurcation in the system, altering the stability of the equilibrium points. Remarkably, it has been seen that when prey species seek refuge at a high rate, it can lead to the extinction of predators, making coexistence impossible. Moreover, the influence of the intermediate predator's odour on the dynamics of the system is established. This impact is readily apparent, as the parameter related to the intermediate predator's odour can trigger numerous transcritical bifurcations in the system. It has been discovered that it has a substantial influence on facilitating the coexistence of all three species in the system. Similarly, the refuge behaviour of the intermediate predator against the top predator also plays a crucial role in the long-term dynamics of the system. This parameter can additionally trigger oscillations in the system by inducing a Hopf bifurcation. Interestingly, an increase in this parameter may contribute to an overall reduction in the size of the prey population. Thus, the act of seeking refuge performed by individuals to avoid being preyed upon can have a significant impact on the long-term dynamics of the system. Additionally, it has been observed that the parameter associated with the odour of the top predator can trigger both a Hopf bifurcation and a transcritical bifurcation in the system. When the concentration of the top predator's odour is insignificant, no refuge is observed among the population of intermediate predators. However, even under such a scenario, it is still conceivable for the three species to coexist. Though fluctuations in the populations of all three species have been noted in this scenario. The study reveals that prey harvesting parameters have the potential to induce both Hopf bifurcation and transcritical bifurcations in the system, thus highlighting the significance of these parameters in the system. When the rate of successfully capturing prey through harvesting activities is low, it becomes challenging for all three species to maintain a stable coexistence. Furthermore, it is evident that for the three species in the system to coexist, a substantial catchability constant $q$ is required, despite the fact that a greater catchability constant $q$ can lead to a reduction in the prey population in the long-term dynamics. Additionally, it has been noted that an increase in the level of harvesting effort leads to a drop in the prey population, aligning with established biological principles. Surprisingly, when the intensity of the harvesting effort reaches an exceedingly high value, it becomes impossible for all species to cohabit, and the top predator faces the risk of extinction.
It has been found that the level of attraction of the intermediate predator to the prey, influenced by the prey's odour, is a critical factor in the coexistence of the three species in the system. It is intriguing that the absence of the prey odour effect in the system poses a challenge for all three species to live together, leading to fluctuations in their populations.  However, with a significant increase in the strength of the prey odour's effect, a Hopf bifurcation arises, leading to the sustained coexistence of all three populations in the system. In addition, as the impact of prey odour becomes stronger, the intermediate predator becomes more drawn to the prey, leading to a decrease in the population of the prey species due to an increase in the rate of encounters between the intermediate predator and the prey. This ultimately results in an elevated predation rate. The effect of memory on the system is also studied using a Caputo-type fractional-order derivative of order $\alpha$ into the modelling of the system. It is observed that as the order of the fractional derivative decreases, the system under consideration becomes more and more stable. From an ecological perspective, it is possible to infer that in this system, as an individual's memory retention ability diminishes, they become unable to recall their previous experiences pertaining to their early life history. As a result, their consciousness of their immediate surroundings drops, leading to the instabilities and fluctuating cohabitation of the three species within the system. Therefore, it can be concluded that the memory of the individuals in this system is crucial for promoting the stability of the cohabitation of the three species within the system. Furthermore, this model can be broadened to include the relationship between a species' odour and its rate of harvesting. In this study, we have found, through a literature survey, a substantial body of ecological studies indicating the correlation between the odour of a species and its harvesting rate, although this aspect is not addressed in this document. This opens avenues for further investigation.\\\\
\textbf{Data Availability Statement} No data of any source has been used in this study.

\end{document}